\documentclass[11pt]{article}
\usepackage[font=small,width=\textwidth]{caption}
\usepackage{amsmath,amsfonts,amssymb,amsthm,thmtools}
\usepackage{comment} 
\usepackage{mathabx}

\usepackage{url}
\usepackage{hyperref}
\hypersetup{
  colorlinks   = true, %Colours links instead of ugly boxes
  urlcolor     = blue, %Colour for external hyperlinks
  linkcolor    = blue, %Colour of internal links
  citecolor   = red %Colour of citations
}
\usepackage[capitalise, nameinlink]{cleveref}

\usepackage{tikz}
\usepackage{tikz-cd}
\usetikzlibrary{matrix}

\usepackage{thmtools}
\usepackage{soul}

\usepackage{fullpage}

\usepackage{todonotes}
\newcounter{notes}

    \newtheoremstyle{TheoremNum}
        {\topsep}{\topsep}              %%% space between body and thm
        {\itshape}                      %%% Thm body font
        {}                              %%% Indent amount (empty = no indent)
        {\bfseries}                     %%% Thm head font
        {.}                             %%% Punctuation after thm head
        { }                             %%% Space after thm head
        {\thmname{#1} \thmnote{\bfseries #3}}%%% Thm head spec

\theoremstyle{TheoremNum}
\newtheorem{lemn}{Lemma}
\newtheorem{thmn}{Theorem}
\newtheorem{propn}{Proposition}
\newtheorem{coron}{Corollary}

\theoremstyle{definition}
\declaretheorem[name=Definition,Refname={Definition,Definitions},parent=section]{defi}
\declaretheorem[Refname={Example,Examples},qed=\ensuremath{\blacksquare},sibling=defi]{example}
\declaretheorem[name=Guide for the reader,numbered=no]{rmk}

\theoremstyle{plain}
\declaretheorem[name=Theorem,Refname={Theorem, Theorems},sibling=defi]{theo}
\declaretheorem[name=Corollary,Refname={Corollary, Corollaries},sibling=defi]{coro}
\declaretheorem[name=Proposition,Refname={Proposition, Propositions},sibling=defi]{prop}
\declaretheorem[name=Lemma,Refname={Lemma, Lemmas},sibling=defi]{lem}
\declaretheorem[name=Question,Refname={Question, Questions},sibling=defi]{ques}

\theoremstyle{remark}
\declaretheorem[name=Remark,numbered=no]{rmq}

\newcommand{\Aut}{\operatorname{Aut}}
\newcommand{\Out}{\operatorname{Out}}
\newcommand{\Inn}{\operatorname{Inn}}

\newcommand{\Stab}{\operatorname{Stab}}

\newcommand{\rk}{\operatorname{rank}}

\newcommand{\ZZ}{\mathbb{Z}}

\newcommand{\NN}{\mathbb{N}}
\newcommand{\RR}{\mathbb{R}}

\newcommand{\GG}{\mathbb{G}}
\newcommand{\free}{\mathbb{F}}
\newcommand{\BB}{\mathfrak{B}}

\newcommand{\defeq}{\coloneq}

\newcommand{\norm}[1]{\left\Vert#1\right\Vert}
\newcommand{\abs}[1]{\left\vert#1\right\vert}

\title{Depth of free-by-cyclic groups}
\author{Spencer Dowdall, Yassine Guerch, Radhika Gupta, \\Jean Pierre Mutanguha and Caglar Uyanik}

\begin{document}
\maketitle

\abstract{For a free group automorphism, we prove that its {poset of attracting lamination orbits} is a canonical invariant of the associated mapping torus. 
That is, if a free-by-cyclic group splits as a mapping torus in two different ways, then the corresponding automorphisms have isomorphic posets of lamination orbits.
Further, we show that the {lamination depth}, the size of the largest chain in this poset, is a commensurability invariant of the free-by-cyclic group.}

\section{Introduction}
Every automorphism $\Phi$ of a finitely generated free group $\free$ gives rise to a \emph{mapping torus extension} of $\free$ by $\ZZ$ defined by the presentation
\[M(\Phi) \defeq \free \rtimes_{\Phi}\ZZ = \langle \free, t \mid t g t^{-1} = \Phi(g)\text{ for all }g\in \free\rangle.\]
Up to isomorphism, this group depends only on the outer class $\phi \defeq [\Phi]\in \Out(\free)$ of the {{monodromy}}/automorphism, so we often write $M(\phi)$ or $\free \rtimes_\phi \ZZ$ instead.
A group $\GG$ is \emph{free-by-cyclic} if it is isomorphic to such an extension $M(\phi)$ for some $\phi \in \Out(\free)$ with nontrivial~$\free$. However, it is well known that $\phi$ is not uniquely determined by $\GG$;
indeed, there may be infinitely many distinct outer automorphisms $\psi\in \Out(\free')$ -- even of free groups $\free'$ of distinct ranks! -- for which $M(\phi) \cong \GG\cong M(\psi)$. Thus, in principle, $\GG$ loses information about $\phi$.
It is natural to ask what features of the outer automorphism $\phi$ are in fact group invariants of the associated free-by-cyclic group $\GG\cong M(\phi)$, or even stronger, commensurability or quasi-isometry invariants of $\GG$.

For the introduction, $\free$ is a finitely generated nontrivial free group and $\phi \in \Out(\free)$.
By the preceding remarks, the rank of $\free$ is not a group invariant of~$M(\phi)$ since it is possible that $M(\phi) \cong M(\psi)$ for some other $\psi \in \Out(\free')$ with $\operatorname{rank}(\free) \neq \operatorname{rank}(\free')$.
In contrast, one can show using \cite{MSW_I, Papasoglu:GrpSplittings} that $\phi$ having finite order in $\Out(\free)$ is a quasi-isometry invariant of~$M(\phi)$ while the precise order of $\phi$ is not even a group invariant. 
As another example, recent work \cite{JPPoly, Macura, Gautero-Lustig, Ghosh, Dahmani-Li} shows that $\phi$ being \emph{polynomially growing} (see~\S\ref{sec:polynomial_growth}), and moreover the degree of polynomial growth, is a quasi-isometry invariant of $M(\phi)$. 
At the other extreme, $\phi$ being \emph{atoroidal} (that is, having no periodic conjugacy classes of nontrivial elements) is also a quasi-isometry invariant since Brinkmann proved this is equivalent to the group $M(\phi)$ being Gromov hyperbolic~\cite{Brinkmann}. Finally, recall that $\phi$ is \emph{irreducible} if it does not preserve a proper free factor system of $\free$. Using dynamical techniques, it was first shown in~\cite{DKL-Dynamics} that when $\phi$ is both irreducible and atoroidal, then the same holds for many other splittings of~$M(\phi)$ (for example, those in the same component of the symmetrized BNS-invariant). It was later proven in full generality~\cite{Mutanguha-Irreducibility} that irreducibility of $\phi$ is a group
invariant of~$M(\phi)$.

In their work proving the Tits Alternative for $\Out(\free)$, Bestvina--Feighn--Handel~\cite{BesFeiHan00} showed that every outer automorphism $\phi$ of $\free$ has a canonically defined (and possibly empty) finite set $\mathcal{L}^+(\phi)$ of \emph{attracting laminations}, each of which is an $\free$--invariant closed subset of the double boundary $\partial^2 \free$ (see \S\ref{subsec:weakattraction} for details). These laminations encode essential dynamical information about the automorphism and, since their introduction over two decades ago, have become fundamental tools in many facets of the theory of $\Out(\free)$.

One might hope that the cardinality of $\mathcal{L}^+(\phi)$ is an invariant of the free-by-cyclic group $M(\phi)$. This guess is motivated by prior results in two important special cases: Firstly, the set $\mathcal{L}^+(\phi)$ is empty precisely when $\phi$ is \emph{polynomially growing}, which we have seen is a quasi-isometry invariant.  Thus $\mathcal{L}^+(\phi)$ being empty is a quasi-isometry invariant. 
Secondly, $\mathcal{L}^+(\phi)$ consists of a single lamination whenever $\phi$ is irreducible. 
Since irreducibility is a group invariant, it follows that $\mathcal{L}^+(\psi)$ also has cardinality $1$ for all outer automorphisms~$\psi$ with $M(\psi)\cong M(\phi)$. However, this does not prove cardinality $1$ is an invariant as there are also reducible automorphisms $\phi$ for which $\abs{\mathcal{L}^+(\phi)} = 1$.

It turns out that the cardinality of $\mathcal{L}^+(\phi)$ is not the right quantity. For example, by considering automorphisms arising from homeomorphisms of punctured surfaces, 
one may find examples $M(\phi)\cong M(\psi)$ where, say, $|\mathcal{L}^+(\phi)| > |\mathcal{L}^+(\psi)| = 1$. 
The full story is more interesting and involved.

The lamination set $\mathcal{L}^+(\phi)$ has two additional structures. For the first structure, attracting laminations may be properly nested inside each other (as subsets of $\partial^2\free$), making $\mathcal{L}^+(\phi)$ a partially ordered set. By the \emph{(lamination) depth} of $\phi$, we mean the length $\delta(\phi)$ of the longest chain in the poset $\mathcal{L}^+(\phi)$. Our original goal for this project was to determine whether the depth $\delta(\phi)$ is a group invariant of $M(\phi)$, which was posed as a question in \cite{JPPoly}. 
Our first main result (\cref{Thm:depthinvariant}) answers this in the affirmative, and we deduce a stronger statement.

\begin{coron}[\ref{commensurability}]
    Let~$\free$ be a finitely generated nontrivial free group and $\phi \in \Out(\free)$.
    The depth $\delta(\phi)$ of the outer automorphism~$\phi$ is a commensurability invariant of the free-by-cyclic group $M(\phi) = \free\rtimes_\phi \ZZ$. 
    That is, if $\psi\in \Out(\free')$ is any outer automorphism for which $M(\phi)$ and $M(\psi)$ are commensurable (i.e.~admit finite index subgroups that are isomorphic), then $\delta(\phi) = \delta(\psi)$.
\end{coron}

Since commensurable groups are quasi-isometric, this provides evidence for the conjecture posed in \cite[Conjecture 5.5]{JPPoly} that depth is a quasi-isometry invariant of the mapping torus.

The second additional structure is that the lamination set $\mathcal{L}^+(\phi)$ is invariant under the action of $\phi$ on the double boundary $\partial^2 \free$. As this action preserves nesting, the partial order descends to the quotient, turning the set $\mathcal{L}^+_\phi$ of $\phi$-orbits of attracting laminations into a canonical poset associated to the automorphism $\phi$. Note that the $\phi$-orbits of attracting laminations are in one-to-one correspondence with the exponentially growing strata in any relative train track representative of $\phi$ \cite[Lemma~3.1.13]{BesFeiHan00}. 
The length of the longest chain in $\mathcal{L}^+_\phi$ is also the length of the longest chain in $\mathcal{L}^+(\phi)$, which is a group invariant by the first result. 
 We must work a bit harder to show the group invariance of this poset of orbits:

\begin{thmn}[\ref{thm:posetlaminationinvariant}] 
    Let~$\free$ be a finitely generated nontrivial free group and $\phi \in \Out(\free)$.
    The poset $\mathcal{L}^+_\phi$ of lamination orbits of the outer automorphism $\phi$ is a group invariant of the free-by-cyclic group $M(\phi) = \free\rtimes_\phi \ZZ$: For any outer automorphism $\psi\in \Out(\free')$, any isomorphism $M(\phi) \to M(\psi)$ induces an order-isomorphism $\mathcal{L}^+_\phi\to \mathcal{L}^+_\psi$.
\end{thmn}

We remark that, in contrast to the depth $\delta(\phi)$, the poset $\mathcal{L}^+_\phi$ associated to $\phi$ is not a commensurability invariant of the free-by-cyclic group $M(\phi)$. Indeed, this can be seen by 
passing to a power of~$\phi$ and the corresponding mapping torus extension. 

In addition to the depth $\delta(\phi)$, we define the \emph{(lamination) depth spectrum} of $\phi$, denoted $\delta \mathcal{S}(\phi)$, to be the set consisting of the lengths of maximal chains in $\mathcal{L}^+(\phi)$ (equivalently~$\mathcal{L}^+_\phi$) recorded without multiplicity. As with depth, we obtain invariance of the depth spectrum:

\begin{coron}[\ref{cor:depthspectrum}]
    Let~$\free$ be a finitely generated nontrivial free group and $\phi \in \Out(\free)$.
    The depth spectrum $\delta\mathcal{S}(\phi)$ of the outer automorphism $\phi$ is a commensurability invariant of the free-by-cyclic group $\free\rtimes_\phi \ZZ$.
\end{coron}

\subsection{Contrasting the 3-manifold setting}
\label{sec:3-manifold_motivation}

Much of the theory of free group automorphisms is motivated by parallels with surface homeomorphisms, and the same is true for the results of this paper. 
To add context and give a better geometric appreciation of the lamination poset,
let us now briefly describe how the story plays out in the surface setting.

For an oriented compact surface $S$ with negative Euler characteristic, the Nielsen--Thurston classification states that every orientation preserving homeomorphism $f\colon S\to S$ is isotopic to a homeomorphism that is either finite-order, reducible, or pseudo-Anosov. Recall that \emph{reducible} means the map setwise preserves a collection, termed a \emph{reducing system},  of disjoint essential simple closed curves in~$S$. On the other hand, \emph{pseudo-Anosov} means there exists a constant $\lambda > 1$ and a pair $(\Lambda_\pm, \mu_\pm)$ of \emph{attracting} and \emph{repelling} arational measured (singular) foliations of~$S$ such that 
$f(\Lambda_\pm, \mu_\pm) = (\Lambda_\pm,\lambda^{\pm1}\mu_\pm)$. 

It follows from the classification that, up to isotopy, every homeomorphism $f$ has a canonical \emph{Thurston normal form} with the following structure: There is a canonical finite (possibly empty) reducing system  $\mathcal{C}$ so that,
for each component $\Sigma$ of $S\setminus\mathcal{C}$ with an $f$-orbit of size $k \ge 1$, the restriction $f^k\vert{\Sigma}\colon \Sigma\to \Sigma$ is either finite order or pseudo-Anosov; $\Sigma$ is called a finite-order or pseudo-Anosov piece of $f$ accordingly. 
Since each pseudo-Anosov piece comes with its own attracting measured foliation (for the map $f^k\vert\Sigma$), we see that every homeomorphism $f$ has a canonical finite (possibly empty) set $\mathcal{L}^+(f)$ of attracting measured foliations: one for each pseudo-Anosov piece in the Thurston normal form. These are the analogs of, and were indeed the motivation for, the attracting laminations of a free group automorphism.

Thurston's geometrization conjecture (now proven by Perelman) posits that every compact oriented $3$-manifold can be canonically decomposed into geometric pieces. In the special case that $M$ fibers over the circle with compact fibers of negative Euler characteristic, 
this decomposition follows from Thurston's hyperbolization theorem and has the following structure:
Such an $M$ can be cut along a canonical finite (possibly empty) collection of essential tori and Klein bottles so that each complementary component is either: Seifert fibered and admits a complete $\mathbb{H}^2\times \mathbb{R}$-structure of finite volume; or atoroidal and admits a complete $\mathbb H^3$-structure of finite volume.

While this decomposition is intrinsic to the manifold $M$, it can also be read off from the dynamics of \emph{any}  monodromy as follows:

To say that $M$ is fibered means it is homeomorphic to the mapping torus $M_f$ of a surface homeomorphism $f\colon S\to S$. This surface homeomorphism has its Thurston normal form---the canonical reduction system $\mathcal{C}$ that decomposes~$f$ into finite-order and pseudo-Anosov pieces. Each $f$-orbit of curves in the reduction system $\mathcal{C}$ gives rise to an essential torus or Klein bottle in the mapping torus $M_f$. After cutting $M_f$ along the resulting collection $\mathcal{T}$ of tori or Klein bottles, the complementary components exactly correspond to $f$-orbits of components of $S\setminus\mathcal{C}$. Moreover, $f$-orbits of finite-order pieces of $S\setminus \mathcal{C}$ give rise to $(\mathbb H^2 \times \mathbb R)$-components of $M_f\setminus\mathcal{T}$ and, by Thurston's hyperbolization theorem, $f$-orbits of pseudo-Anosov pieces of $S\setminus\mathcal{C}$ give rise to $\mathbb H^3$-components.

By uniqueness, this decomposition of the mapping torus $M_f$ must be the geometric decomposition of the manifold $M$ into $\mathbb H^2\times \mathbb R$ and $\mathbb H^3$ pieces. In particular, the hyperbolic components of $M\cong M_f$ are in bijective correspondence with the $f$-orbits of pseudo-Anosov pieces of $S\setminus\mathcal{C}$ and thus with the $f$-orbits of the set $\mathcal{L}^+(f)$ of measured foliations; therefore, despite the well-known fact that $M$ may fiber as the mapping torus $M \cong M_g$ for infinitely many essentially distinct surface homeomorphisms $g\colon S'\to S'$, the uniqueness of the geometric decomposition shows that the sets of measured foliation orbits are naturally identified for all such fibrations. 

\subsection{Sketch of key ideas}
The above discussion proves the analog of 
\Cref{thm:posetlaminationinvariant} for surface homeomorphisms and explains why \emph{orbits} of laminations are relevant. It also highlights important differences  and challenges that must be overcome in our free group setting. Chief among these is that the result for surface homeomorphisms rests upon monumental work regarding the structure of $3$-manifolds. Frankly, there is no geometrization theorem for free-by-cyclic groups. Secondly, we see that the attracting measured foliations for surfaces do not interact with each other: they are supported on disjoint subsurfaces and are decidedly not nested. The nesting of laminations is something unique to the free group setting, making it unclear whether the depth or poset structure of $\mathcal{L}^+_\phi$ should be preserved in different splittings of the group.

Accordingly, our approach to \Cref{thm:posetlaminationinvariant} takes inspiration from the $3$-manifold setting but differs greatly in execution. For example, for a surface homeomorphism $f$, the measured foliations are supported on subsurfaces that give rise to canonical submanifolds of the mapping torus~$M_f$. One might hope that for a free group outer automorphism $\phi\in \Out(\free)$, the laminations would be supported by subgroups of $\free$ that induce canonical subgroups of $M(\phi)$. 
While this turns out to be true, constructing these supporting subgroups is significantly more involved, and proving that they induce canonical objects in $M(\phi)$ is rather circuitous (see \S\ref{sec:grp_inv_poset}).

To start, each lamination orbit $[\Lambda]\in \mathcal{L}^+_\phi$ gives rise to a canonical \emph{nonattracting subgroup system} $\mathcal{N}[\Lambda]$. 
Using these (specifically, the systems $\mathcal{N}[\Lambda']$ for those $[\Lambda']\in \mathcal{L}_\phi^+$ that are not contained in $[\Lambda]$), one may also construct a \emph{supporting subgroup system} $\mathcal{S}[\Lambda]$ that carries the lamination $[\Lambda]$ (see~\S\ref{sec:nonattracting_ss}).
Both $\mathcal{N}[\Lambda]$ and $\mathcal{S}[\Lambda]$ in turn give rise to their own mapping tori, which are free-by-cyclic subgroup systems of $\GG = M(\phi)$ (see~\S\ref{sec:mapping_tori_of_subgroup_system}).
We focus on the mapping tori $\mathcal{M}[\Lambda]$ coming from the nonattracting subgroup systems $\mathcal{N}[\Lambda]$, which are much easier to deal with.
Unfortunately, these subgroup systems are generally 
not known to be canonical in $\GG$. 
However, we show in the case when $\GG$ has \emph{no cyclic splittings} that the subgroup systems $\mathcal{M}[\Lambda]$ for 
\emph{maximal} elements $[\Lambda]$ of $\mathcal{L}^+_\phi$ are canonical in~$\GG$.
Indeed, they are algebraically characterized as the maximal proper malnormal free-by-cyclic subgroup systems of $\GG$ that contain the \emph{polynomial subgroup system} $\mathcal{P}(\GG)$ (see \S\ref{subsec:growthtypecanon} for the definition). 
This is the content of the next theorem and 
forms a key ingredient in our argument.

\begin{thmn}[\ref{thm:maximallamcanonical}]
Let~$\free$ be a finitely generated nontrivial free group and $\phi \in \Out(\free)$.
If $\GG = \free \rtimes_\phi \ZZ $  
has no cyclic splittings, then the set 
\[\left\{\mathcal M[\Lambda]\,:\,[\Lambda] \text{ maximal in } \mathcal L^+_{\phi}\right\} \] is the set of maximal proper malnormal free-by-cyclic subgroup systems in $\GG$ supporting $\mathcal P(\GG)$. Moreover, this set is in one-to-one correspondence with the collection of maximal lamination orbits in $\mathcal L^+_\phi$.
\end{thmn}

This theorem is essentially a consequence of \Cref{Prop:invariantmalnormal}, which gives sufficient conditions for a malnormal subgroup system of~$\free$ to be supported by the nonattracting subgroup system $\mathcal{N}[\Lambda]$ of some lamination orbit $[\Lambda] \in \mathcal L^+_\phi$.
In order to apply \Cref{thm:maximallamcanonical}, one needs the free-by-cyclic group $\GG$ to have no cyclic splittings. While this is not true in general, we show in \Cref{prop:passing to no cyclic splittings} how to pass to a canonical subgroup system of $\GG$ with no cyclic splittings.

These two key ingredients set up a recursive procedure: First apply \Cref{prop:passing to no cyclic splittings} to $\mathcal G_0 = \GG$ and pass to a canonical free-by-cyclic subgroup system~$\mathcal G_0'$ with no cyclic splittings, then apply \Cref{thm:maximallamcanonical} to $\mathcal G_0'$ and extract the canonical subgroup systems corresponding to maximal lamination orbits. Next, take their meet (i.e.~intersection)~$\mathcal G_1$, and repeat from the first step until the procedure terminates. In this way, we obtain a canonical filtration in $\GG$ by free-by-cyclic subgroup systems whose length is the depth of~$\phi$ (\Cref{Thm:depthinvariant}). 

Finally, in~\S\ref{sec:grp_inv_poset}, we use this canonical filtration to prove that the mapping tori of the supporting subgroup systems $\mathcal{S}[\Lambda]$ are in fact canonical in $\GG$ and, moreover, that this canonical poset of subgroup systems of~$\GG$ is order-isomorphic to $\mathcal{L}^+_\phi$ (\Cref{thm:posetlaminationinvariant}).

A few results in the paper utilize the machinery of relative train track maps and topological trees. In order to avoid the arguments becoming too cumbersome, we postpone these details to the appendices. \Cref{app:passingtofinite} contains the proof of \Cref{Prop:powers} (invariance of depth under passing to finite index subgroups), which requires the precise definition of attracting laminations and relative train track maps. 
\Cref{app:genericapprox} constructs an element that approximates the generic leaves of an attracting lamination (\Cref{lem:approximategenline}).
\Cref{app:toptree} contains the explicit construction of a canonical topological tree associated to a collection of attracting laminations (\Cref{Thm:NSS}), implicit in the results of \cite{Mutanguha22}.

\medskip
\begin{rmk}
We note that, due to the recursive nature of our argument, we generally do not work with the groups $\free$ and $\GG$ directly, but rather work within certain subgroup systems. That is, all of the machinery is developed for: the restriction $\phi\vert\mathcal{A}$ of an outer automorphism $\phi$ to an invariant subgroup system $\mathcal{A}$ of $\free$; and the associated mapping torus $\mathcal{G} = \mathcal{M}(\phi\vert\mathcal{A})$, a free-by-cyclic subgroup system of $\GG$. However, to simplify the definitions and proofs, the reader can often pretend that $\mathcal{A} = \free$, $\phi\vert\mathcal{A} = \phi$, and $\mathcal G = \GG$.
\end{rmk}

\subsubsection*{Further directions} 
Given 
$\phi \in \Out(\free)$,  in this paper, we focus on the collection $\mathcal{L}^+(\phi)$ of attracting laminations, which is empty if and only if $\phi$ is polynomially growing. Nevertheless, one can associate a collection of laminations (not attracting anymore), denoted $\Omega(\phi)$, to a polynomially growing outer automorphism $\phi$  as explained by
Feighn--Handel in \cite{FeighnHandel:Recognition}. Handel--Mosher~\cite{HandelMosher20} later referred to $\Omega(\phi)$ as the set of \emph{weak accumulation sets of eigenrays and twistors of $\phi$}. In the same paper~\cite{FeighnHandel:Recognition}, Feighn--Handel show that $\Omega(\phi)$ is a partially ordered set and the length of a maximal chain in~$\Omega(\phi)$ is equal to the degree of polynomial growth of $\phi$, which is known to be a quasi-isometry invariant of $\free \rtimes_\phi \ZZ$.
Let $\Omega_\phi$ denoted the poset of $\phi$-orbits of laminations in~$\Omega(\phi)$.
Based on \Cref{thm:posetlaminationinvariant}, we expect a positive answer to the following question. 

\begin{ques} Is $\Omega_\phi$, the poset of lamination orbits associated to a polynomially growing automorphism $\phi$, a group invariant of $\free \rtimes_\phi \ZZ$? 
\end{ques}

Moreover, given any $\phi \in \Out(\free)$, one can look at the union of attracting laminations and weak accumulation sets of eigenrays and twistors, which is again a partially ordered set. We are curious to know if this full poset of laminations associated to $\phi$ or its depth is an invariant of $\free \rtimes_\phi \ZZ$. This invariant would be useful for distinguishing hyperbolic free-by-cyclic groups.

\subsubsection*{Acknowledgments}

We would like to thank the American Institute of Mathematics for their hospitality during the October 2023 workshop “Rigidity properties of free-by-cyclic groups” where this work started. We thank Mladen Bestvina, Martin Bridson, and Gilbert Levitt for helpful conversations.  

SD gratefully acknowledges support from NSF grants DMS--2005368 and DMS--2405061.  YG thanks the LABEX MILYON (ANR-10-LABX-0070) for its financial support during his postdoctoral position; during which part of this project was written. RG gratefully acknowledges partial support from SERB research grant SRG/2023/000123 and a grant from the Simons Foundation whilst on the Isaac Newton Institute
Ramanujan Fellowship. CU gratefully acknowledges support from the NSF grant DMS--2439076.

\section{Preliminaries} \label{sec:prelims}

\subsection{Subgroup systems}\label{subsec:ss}

Let $G$ be a group. A \emph{subgroup system} of $G$ is a (possibly empty) set $\mathcal{A}$ of conjugacy classes of nontrivial subgroups of $G$;
moreover, we require that if $[A], [B] \in \mathcal A$ are distinct conjugacy classes, then $A$ cannot be a subgroup of $B$.
A subgroup system has \emph{finite type} if it consists of finitely many conjugacy classes of finitely generated subgroups.
It is \emph{proper (resp.~cyclic)} if it consists of conjugacy classes of proper (resp.~cyclic) subgroups. 
A subgroup $A\le G$ is \emph{malnormal} if for any $h\in G \setminus A$,  the intersection $hAh^{-1}\cap A$ is trivial. 
A subgroup system $\mathcal{A}$ of $G$ is \emph{malnormal} if each $A$ with $[A]\in\mathcal{A}$ is malnormal in $G$ and, for each $A, B$ with $[A] \neq [B]$ in $\mathcal A$, the intersection $A \cap B$ is trivial.

There is a natural partial order, denoted $\sqsubseteq$, on the set of malnormal subgroup systems where  
{$\mathcal{B} \sqsubseteq \mathcal{A}$} if, for every $[B] \in \mathcal{B}$, there exists $[A] \in \mathcal{A}$ with $B \leq A$. 
In this case, we say that $\mathcal A$ \emph{supports} $\mathcal B$.
We will write $\mathcal B \sqsubsetneq \mathcal A$ if $\mathcal B \sqsubseteq \mathcal A$ and $\mathcal A \ne \mathcal B$. 
Note that malnormality is needed to ensure this relation is antisymmetric, for otherwise one might have a situation $A'  \lneq B \lneq A$ with $[A]=[A']\in \mathcal A$. 
This partial order is a meet-semilattice, i.e. every pair $\mathcal A, \mathcal B$ of malnormal subgroup systems has a greatest lower bound $\mathcal A \wedge \mathcal B$;
specifically, the \emph{meet} 
\[\mathcal A \wedge \mathcal B \defeq \{ [A\cap B]\,:\, [A] \in \mathcal A, [B] \in \mathcal B, \text{ and } A \cap B \text{ is not trivial} \}.\]
% When $G$ is a free group, then the meet of two malnormal subgroup systems with finite type also has finite type itself 
When $G$ is a free group, the following lemma extends the partial order beyond malnormal subgroup systems to the collection $G^{\le}_{\mathrm{fg}}$ of subgroup systems of~$G$ consisting of conjugacy classes of finitely generated subgroups. 
\begin{lem}\label{lem:partialorder_free}
    Let $G$ be a free group. The collection $G^{\le}_{\mathrm{fg}}$  has a partial order defined by $\mathcal B \sqsubseteq \mathcal A$ if, for every $[B] \in \mathcal B$, there exists $[A] \in \mathcal A$ with $B \le A$.
    Moreover, when restricted to subgroup systems with finite type, the partial order is a meet-semilattice.
\end{lem}

\noindent This lemma (and its proof) will serve as a blueprint for the partial order constructed in \S\ref{subsec:poset_of_freebycyclic}.

\begin{proof}
    Let $A \le G$ be a finitely generated subgroup. 
    We first observe that if $A' \le A$ is a conjugate of $A$, then $A' = A$.
    Let $\iota \colon G \to G$ be the inner automorphism that maps $A$ into itself with image $\iota(A) = A'$.
    As free groups are subgroup separable and $A$ is finitely generated, the automorphism $\iota$ must map $A$ onto itself, i.e. $A'  = A$  (see \cite[Lemma~6.0.6]{BesFeiHan00}).

    This observation implies the relation $\sqsubseteq$ is antisymmetric.
    Indeed, let $\mathcal A, \mathcal B \in G^{\le}_{\mathrm{fg}}$, and suppose $\mathcal A \sqsubseteq \mathcal B$ and $\mathcal B \sqsubseteq \mathcal A$.
    Pick an arbitrary $[A'] \in \mathcal A$.
    Then $A' \le B$ for some $[B] \in \mathcal B$.
    Similarly, $B \le A$ for some $[A] \in \mathcal A$.
    Since $\mathcal A$ is a subgroup system and $A' \le A$, we must have $[A'] = [A]$, i.e. $A'$ is conjugate to $A$.
    By the previous paragraph, $A' = A$.
    Consequently, $A' = B = A$ and $[A'] = [B]$.
    As $[A'] \in \mathcal A$ was arbitrary, we have shown $\mathcal A \subseteq \mathcal B$.
    The same argument shows $\mathcal B \subseteq \mathcal A$; hence $\mathcal A = \mathcal B$.
    It is straightforward to check that the relation $\sqsubseteq$ is reflexive and transitive.

    Now, suppose $\mathcal A, \mathcal B$ are subgroup systems of $G$ with finite type.
    We start with a collection mimicking the meet of malnormal subgroup systems:
    \[\mathcal A \hat\wedge \mathcal B \defeq \{ [A\cap B]\,:\, [A] \in \mathcal A, [B] \in \mathcal B, \text{ and } A \cap B \text{ is not trivial} \}.\]
    This need not be a subgroup system since, without the malnormality assumption, distinct conjugacy classes may be nested.
    Nevertheless, it is a finite collection \cite[Proposition~2.1]{Neumann}.
    For each $[A \cap B] \in \mathcal A \hat\wedge \mathcal B$, the singletons $\{ [A\cap B] \}$ are subgroup systems partially ordered by~$\sqsubseteq$.
    Then $\mathcal A \wedge \mathcal B$ is defined to be the subset of maximal classes in $\mathcal A \hat\wedge \mathcal B$.
    One can verify that $\mathcal A \wedge \mathcal B$ is the meet, i.e. the greatest lower bound of $\mathcal A$ and $\mathcal B$.
\end{proof}

Subgroup systems arise naturally in the context of relatively hyperbolic groups. For example, if $G$ is a hyperbolic group and $\mathcal A$ is a proper subgroup system of $G$ with finite type, then $G$ is hyperbolic relative to $\mathcal A$ if and only if 
$\mathcal{A}$ is malnormal and consists of (conjugacy classes of) quasi-convex subgroups \cite[Theorem 7.11]{Bowditch}  (see also~\cite{Osin, DGO}).
For a relatively hyperbolic pair $(G, \mathcal{A})$, we say that an element $g \in G$ is \emph{$\mathcal{A}$-peripheral} if there exists $[A] \in \mathcal{A}$ with $g \in A$;
we similarly define $\mathcal{A}$-peripheral subgroups.

Let $\mathcal A$, $\mathcal B$ be two malnormal subgroup systems of $G$ with finite type such that $\mathcal B \sqsubseteq \mathcal A$. 
For $[A] \in \mathcal{A}$, the \emph{restriction} of $\mathcal B$ to $A$ is 
\[ \mathcal B|A \defeq \{[B]_A \colon [B]\in \mathcal{B}\text{ and }B\le A\},\]
where $[\cdot]_A$ denotes the $A$-conjugacy class.
So the restriction $\mathcal B|A$ is a subgroup system of~$A$.

\subsection{Trees and splittings}

A \emph{$G$-tree} is a simplicial metric tree equipped with a minimal action of $G$ by isometries.
The tree is \emph{trivial} if it is a point.
A \emph{splitting} of $G$ is a $G$-equivariant isometry class of nontrivial $G$-trees. A splitting of $G$ is \emph{free} if edge stabilizers are trivial, and it is \emph{cyclic} if edge stabilizers are cyclic (possibly trivial). Let $S$ be a free splitting of~$G$. The set of conjugacy classes of nontrivial vertex stabilizers of $S$ is a malnormal subgroup system of $G$ called a \emph{proper free factor system}. 
If $G$ is finitely generated, then its proper free factor systems have finite type by Bass--Serre theory.

The automorphism group $\Aut(G)$ acts naturally on the right on the set of splittings of $G$ by precomposition of the action. Since we are considering $G$-equivariant isometry classes of trees, the action of $\Aut(G)$ factors through an action of $\Out(G)$. Note that $\Out(G)$ preserves the set of free splittings and the set of cyclic splittings of $G$. 

Let $S$ be a set with a $G$-action and $\Phi \in \Aut(G)$. 
A function $f \colon S \to S$ is \emph{$\Phi$-equivariant} if: for every $s \in S$ and every $g \in G$, we have $f(g \cdot s)=\Phi(g) \cdot f(s).$
If, additionally, $S$ is a splitting of $G$ that is not isometric to~$\RR$ and $\Phi$ preserves $S$, i.e. $S \cdot \Phi = S$, then {there exists a unique $\Phi$-equivariant isometry} $\iota_\Phi \colon S \to S$ (for instance, see~\cite[\S2.1]{Lev05}).

Let $\mathcal{G}$ be a subgroup system of $G$. A \emph{splitting} of $\mathcal{G}$ is a set $\mathcal S = \{T_{G} \,:\, [G]\in \mathcal G \}$ where each $T_{G}$ is a $G$-equivariant isometry class of $G$-trees (one for each conjugacy class $[G]$) and some $T_{G}$ is a splitting of $G$, that is, at least one $T_G$ is nontrivial.
A splitting $\mathcal S$ of $\mathcal{G}$ is \emph{cyclic} (resp.~\emph{free}) if each of the splittings in $\mathcal S$ is cyclic (resp.~free).
Now suppose $\mathcal G$ is a malnormal subgroup system of~$G$.
A \emph{proper free factor system} of $\mathcal{G}$ is the subgroup system of $G$ consisting of the conjugacy classes of nontrivial vertex stabilizers of a free splitting of $\mathcal{G}$.
One can check that proper free factor systems of $\mathcal G$ are malnormal subgroup systems of $G$.
When $\mathcal G$ has finite type, then so do its proper free factor systems by Bass-Serre theory again.

We now describe some properties of free factor systems of malnormal subgroup systems.

\begin{lem}\label{lem:transitivityffs}
    Let $\mathcal G$ be a malnormal subgroup system of $G$ and let $\mathcal F_1$ be a free factor system of~$\mathcal G$.
    If $\mathcal F_2$ is a free factor system of $\mathcal F_1$, then $\mathcal F_2$ is a free factor system of $\mathcal G$.
\end{lem}

\begin{proof}
    If $\mathcal F_1 = \mathcal G$ or $\mathcal F_2 = \mathcal F_1$, then there is nothing to show. 
    So we may assume both free factor systems are proper.
    Let $\mathcal S_1$ be a free splitting of $\mathcal G$ whose system of nontrivial vertex stabilizers is $\mathcal F_1$ and let $\mathcal S_2$ be a free splitting of $\mathcal F_1$ whose system of nontrivial vertex stabilizers is $\mathcal F_2$. 
    Then one can blow-up the splitting $\mathcal S_1$ 
    by equivariantly replacing the vertices with nontrivial stabilizers with the splitting $\mathcal S_2$.
    This produces a new free splitting 
    $\mathcal S_1'$ of $\mathcal G$ whose system of nontrivial vertex stabilizers is $\mathcal F_2$.
    Hence, $\mathcal F_2$ is a free factor system of $\mathcal G$.
\end{proof}

\begin{lem}\label{lem:meet_ffs_is_ffs}
    Let $\mathcal H \sqsubseteq \mathcal G$ be malnormal subgroup systems of $G$.
    If $\mathcal F$ is a free factor system of~$\mathcal G$, then $\mathcal H \wedge \mathcal F$ is a free factor system of $\mathcal H$.
\end{lem}

\begin{proof}
    We may assume $\mathcal F \neq \mathcal G$ and $\mathcal H \wedge \mathcal F \neq \mathcal H$.
    Let $\mathcal S$ be a free splitting of $\mathcal G$ whose system of nontrivial vertex stabilizers is $\mathcal F$.
    For each $[H] \in \mathcal H$, malnormality of~$\mathcal G$ implies there is a unique $T_H \in \mathcal S$ with an $H$-action induced by the free splitting. Let $T_H' \subseteq T_H$ be the minimal subtree for~$H$. Then the nontrivial vertex stabilizers in $H$ of $T_H'$ are precisely the nontrivial intersections of $H$ with representatives of classes in $\mathcal F$. As $\mathcal H \wedge \mathcal F \neq \mathcal H$, the set $\mathcal S' \defeq \{T_H' \,:\, [H] \in \mathcal H \}$ is a free splitting of $\mathcal H$, and its nontrivial vertex stabilizers are $\mathcal H \wedge \mathcal F$. 
\end{proof}

\begin{lem}\label{lem:meet_ffs}
    Let $\mathcal G$ be a malnormal subgroup system of $G$ with finite type. 
    If $(\mathcal F_n)_{n \in \NN}$ is a sequence of free factor systems of $\mathcal G$, then the sequence $(\bigwedge_{n \leq N} \mathcal F_n)_{N \in \NN}$ stabilizes. 
\end{lem}

\begin{proof}
    The \emph{complexity} of a system $\mathcal A=\{[A_1],\ldots,[A_k]\}$ is $c(\mathcal A)= -k + \sum_{i=1}^k 2\rk(A_i)$, where $\rk(A_i)$ is the smallest size of a generating set of $A_i$. 
    It follows from Bass--Serre theory that $c(\mathcal F) > c(\mathcal F')$ for any proper nesting  $\mathcal F \sqsupsetneq \mathcal F'$ of free factor systems of $\mathcal G$.
    So any nested sequence of free factor systems of $\mathcal G$ stabilizes.
    Note that $(\bigwedge_{n \leq N} \mathcal F_n)_{N \in \NN}$ is a nested sequence of free factor systems of $\mathcal G$ by \cref{lem:transitivityffs,lem:meet_ffs_is_ffs};
    therefore, it stabilizes. 
\end{proof}

For a sequence $(\mathcal F_n)_{n \in \NN}$ of free factor systems of a malnormal subgroup system $\mathcal G$ with finite type, the \emph{meet} of $(\mathcal F_n)_{n \in \NN}$, denoted by $\bigwedge_{n \in\NN} \mathcal F_n$, is the free factor system $\bigwedge_{n \le N} \mathcal F_n$ of~$\mathcal G$ for large enough $N \in \NN$ given by \cref{lem:transitivityffs,lem:meet_ffs_is_ffs}, which is well-defined by \cref{lem:meet_ffs}.

\subsection{Free-by-cyclic groups and mapping tori}\label{subsec:freebycyclic_mappingtori}

A group is \emph{free-by-cyclic} if it is a group extension of a finitely generated nontrivial free group by an infinite cyclic group. Note that free-by-cyclic groups are not free.

We fix once-and-for-all a free-by-cyclic group $\GG$ that is a group extension of a finitely generated {nonabelian} free group $\free$ by an infinite cyclic group.
As the infinite cyclic group is free, the extension splits, and there is a natural isomorphism $\GG \cong M(\Phi)$, where $\Phi \in \Aut(\free)$ and $M(\Phi)$ is the associated \emph{mapping torus} group with the relative presentation:
\[
M(\Phi) \defeq \free \rtimes_\Phi \ZZ = \langle \free, t \mid tgt^{-1} = \Phi(g) \text{ for all } g \in \free \rangle. \]
Denote by $\pi \colon M(\Phi) \to \ZZ$ the natural projection with $\ker \pi = \free$ and $\pi(t)  = 1$.
Up to isomorphism, the mapping torus depends only on the outer automorphism $\phi \defeq [\Phi]$ in $\Out(\free)$, and we also write $M(\phi)$ or $\free \rtimes_\phi \ZZ$ instead.
Throughout the paper, $\phi \in \Out(\free)$, $\Phi \in \phi$, and $\GG = M(\Phi)$.

Let $S$ be an $\free$-set, i.e.~a set with an $\free$-action, and let $f \colon S \to S$ be a bijection.
We chose the relation $tgt^{-1}=\Phi(g)$ for $M(\Phi)$ so that the $\Phi$-equivariance of $f$ is precisely the condition needed to extend the $\free$-action on $S$ to an $M(\Phi)$-action by $t \cdot s = f(s)$.

\subsection{Poset of free-by-cyclic subgroup systems of \texorpdfstring{$\GG$}{G}}
\label{subsec:poset_of_freebycyclic}

\begin{defi}
    A subgroup system of $\GG$ is  \emph{free-by-cyclic} if it consists of conjugacy classes of free-by-cyclic subgroups.
\end{defi}

The following lemmas extend the partial order beyond malnormal subgroup systems of $\GG$ to the collection of free-by-cyclic subgroup systems, in a manner similar to \Cref{lem:partialorder_free}. 
%The next two lemmas will show that the collection of all free-by-cyclic subgroup systems of $\GG$ has a natural partial order. Recall that, in general, we only defined a partial order on the collection of malnormal subgroup systems of an arbitrary group $G$.

\begin{lem}\label{lem:feighnhandel}
    Let $\GG \defeq \free \rtimes_\phi \ZZ$ and $\pi \colon \GG \to \ZZ$ be the natural projection.
    If $H\le \GG$ is a free-by-cyclic subgroup, then $A \defeq H \cap \free$ is a $\phi$-periodic finitely generated subgroup of $\free$, and $H = \langle A, xt^k\rangle$ for some element $x \in \free$ and integer $k\ge 1$ that generates $\pi(H)$;
    moreover, $H$ is naturally identified with $M(\psi)$ for some $\psi \in \Out(A)$.
\end{lem}

\noindent A subgroup $A \le \free$ is \emph{$\phi$-periodic} if the conjugacy class $[\Phi^k(A)] = [A]$ for some integer $k \ge 1$.

\begin{proof}
    As $H$ is not free, $\pi(H) \le \ZZ$ is not trivial.
    Let $k \ge 1$ be a generator of $\pi(H)$.
    Then $xt^k \in H$ for some $x \in \free$, and $H = \langle A, xt^k \rangle$, where $A \defeq H \cap \free$.
    As~$A$ is normal in~$H$, it is preserved by conjugation in~$\GG$ by $xt^k$.
    In other words, $[\Phi^k(A)] = [A]$, i.e.~$A$ is $\phi$-periodic.
    The Feighn--Handel coherence theorem implies $A$ is finitely generated \cite{FeighnHandel99} (see \cite[Lemma~3.5]{JPPoly} for details).
    Since $\pi(H)$ is infinite cyclic, $H$ is naturally identified with $M(\Psi) = A \rtimes_\Psi \ZZ$, where $\Psi \colon A \to A$ is the restriction of conjugation by $xt^k$.
    The element $x \in \free$ is well-defined up to left multiplication by an element of $\ker(\pi|H) = A$.
    Thus, the outer class $[\Psi] \in \Out(A)$ is well-defined.
\end{proof}

\begin{lem}\label{lem:partialordersubgroupsys}
    Let $\GG \defeq \free \rtimes_\phi \ZZ$ and $H \le \GG$ be a free-by-cyclic subgroup.
    If a subgroup $H' \le H$ is a conjugate of $H$, then $H' = H$.
\end{lem}
\begin{proof}
    By \cref{lem:feighnhandel}, we have $H = \langle A, xt^k \rangle$ for the $\phi$-periodic finitely generated subgroup $A \defeq H \cap \free$, some element $x \in \free$, and the integer $k \ge 1$ that generates $\pi(H) \le \ZZ$.
    By hypothesis, $H' = (yt^n) H (yt^n)^{-1} \le H$ for some $y \in \free$ and $n \in \ZZ$.
    Then $(yt^n) A (yt^n)^{-1} = H' \cap \free \le A$ as $\free$ is normal in $\GG$.
    Note that conjugation in $\GG$ by $yt^n$ restricts to an automorphism of $\free$ that maps $A$ into itself.
    As $\free$ is subgroup separable 
    and $A$ is finitely generated, the automorphism must map $A$ onto itself, i.e.~$H' \cap \free = (yt^n)A (yt^n)^{-1} = A$ (see \cite[Lemma~6.0.6]{BesFeiHan00}).
    Note that $\pi(H') = \pi(H)$ in $\ZZ$.
    Since both $H'$ and $H$ have the same $\pi$-image and the same intersection with the kernel $\ker \pi = \free$, they must be equal.
\end{proof}

As in the proof of \Cref{lem:partialorder_free}, the property in \Cref{lem:partialordersubgroupsys} is precisely what we need to define the following partial order.

\begin{coro}
The collection of free-by-cyclic subgroup systems of a free-by-cyclic group $\GG$ has a partial order defined by $\mathcal{H} \sqsubseteq \mathcal{G}$ if, for every $[H] \in \mathcal{H}$, there exists $[G] \in \mathcal{G}$ with $H \leq G$. \qed
\end{coro}
% \begin{proof}
%     We need to show that the relation $\sqsubseteq$ is antisymmetric---reflexivity and transitivity are straightforward.
%     Suppose $\mathcal H, \mathcal G$ are free-by-cyclic subgroup systems of $\GG$, with $\mathcal H \sqsubseteq \mathcal G$ and $\mathcal G \sqsubseteq \mathcal H$.
%     Pick an arbitrary $[H'] \in \mathcal H$.
%     Then $H' \leq G$ for some $[G] \in \mathcal G$.
%     Similarly, $G \le H$ for some $[H] \in \mathcal H$.
%     Since $H' \le H$, we must have $[H'] = [H]$, i.e.~$H'$ is conjugate to $H$.
%     By \cref{lem:partialordersubgroupsys}, $H' = H$.
%     Consequently, $H' = G = H$ and $[H'] = [G]$.
%     As $[H'] \in \mathcal H$ was arbitrary, we have shown $\mathcal H \subseteq \mathcal G$.
%     By the same argument, $\mathcal G \subseteq \mathcal H$; hence $\mathcal H = \mathcal G$.
% \end{proof}

The partial order on free-by-cyclic subgroup systems of $\GG$ with finite type is a meet-semilattice as well, where we define the meet operation like in \Cref{lem:partialorder_free}. First, we define the collection
\[ \mathcal H \hat\wedge \mathcal G \defeq \big\{ [H\cap G]\,:\, [H] \in \mathcal H, [G] \in \mathcal G, \text{ and } H \cap G \text{ is not trivial} \big\}.\]
This collection need not be a subgroup system since distinct conjugacy classes may be nested.
An application of \cref{lem:feighnhandel} (and \cite[Proposition~2.1]{Neumann}) shows that $\mathcal H \hat\wedge \mathcal G$ consists of finitely many conjugacy classes of  cyclic or free-by-cyclic subgroups. 
The free-by-cyclic conjugacy classes (treated as singletons) are partially ordered by $\sqsubseteq$.
The \emph{meet} $\mathcal H \wedge \mathcal G$ consists of the maximal free-by-cyclic classes in $\mathcal H\hat\wedge\mathcal G$.
When applied to malnormal free-by-cyclic subgroup systems, this modification of the previous meet operation just amounts to ignoring cyclic intersections.

\subsection{Mapping tori over subgroup systems of \texorpdfstring{$\free$}{F}}
\label{sec:mapping_tori_of_subgroup_system}
A natural way to construct free-by-cyclic subgroup systems of $M(\phi) = \free \rtimes_\phi \ZZ$ is to consider the mapping torus of a $\phi$-invariant malnormal subgroup system $\mathcal A$ of $\free$ with finite type.
A subgroup system $\mathcal{A}$ of $\free$ is \emph{$\phi$--invariant} if $\phi[A] \defeq [\Phi(A)] \in \mathcal{A}$ for all $[A]\in \mathcal{A}$. 
So~$\phi$ permutes the conjugacy classes in~$\mathcal A$.
Fix a $\langle \phi \rangle$-orbit $\langle \phi \rangle[A] \subseteq \mathcal A$, and let $k \ge 1$ be the size of the orbit---the orbit is finite since~$\mathcal A$ has finite type.
As $\phi^k[A] = [A]$, we get $\Phi^k(A) = x^{-1}Ax$ for some $x \in \free$.
Let $\Phi' \in \Aut(\free)$ be the automorphism given by $g \mapsto x\Phi^k(g)x^{-1}$, and let $\Psi \in \Aut(A)$ be the restriction of~$\Phi'$ to~$A$.
By malnormality of~$\mathcal A$, the element $x \in \free$ is unique up to left multiplication by an element of~$A$.
So the \emph{first return} outer class $\psi \defeq [\Psi] \in \Out(A)$ is well-defined and may be denoted $\phi^k|A \defeq \psi$. 
We naturally identify $M(\psi) \defeq A \rtimes_\psi \mathbb Z$ with the subgroup $\langle A, xt^k \rangle \le M(\phi)$.
Since~$A$ is malnormal in~$\free$, the subgroup~$M(\psi)$ is the normalizer of~$A$ in~$M(\phi)$.
Note that two such subgroups $M(\psi_1)$ and $M(\psi_2)$ are conjugate in $M(\phi)$ if and only if the orbits $\langle\phi\rangle[A_1]$ and $\langle\phi\rangle[A_2]$ are equal.
The \emph{mapping torus} of~$\mathcal A$ is the following finite collection:
\[ M(\phi|\mathcal A) \defeq \Big\{ [M(\phi^k|A)]~:~\langle \phi \rangle [A] \subseteq \mathcal A \text{ is a } \langle\phi\rangle\text{-orbit of size $k$} \Big\}. \]

\begin{defi}\label{def:restrictions}
Expanding on the above notation, define $\phi\vert \mathcal{A}$ to be 
the set of first return outer classes $\phi^k|A$ over all $[A] \in \mathcal A$.
We moreover say $\phi|\mathcal{A}$ \emph{preserves a splitting} $\mathcal S$ of $\mathcal{A}$ if, for each $[A] \in \mathcal{A}$, the restriction $\phi^k|A$ preserves the corresponding $A$-tree in~$\mathcal S$. 

\end{defi}

When $\mathcal A$ consists of a single $\langle \phi \rangle$-orbit, then the singleton $M(\phi|\mathcal A)$ is a subgroup system of~$M(\phi)$. The next lemma shows that $M(\phi|\mathcal A)$ is always a subgroup system of~$M(\phi)$ and gives a natural correspondence between nested $\phi$-invariant systems in~$\free$ with nested mapping tori.

\begin{lem}\label{lem:inclusionsubsys}
Let~$\mathcal A$ and~$\mathcal B$ be $\phi$-invariant malnormal subgroup systems of~$\free$ with finite type.
\begin{enumerate} 
\item Then $M(\phi|\mathcal A)$ and $M(\phi|\mathcal B)$ are subgroup systems of~$M(\phi)$ with finite type; and
\item $\mathcal A \sqsubseteq \mathcal B$ if and only if $M(\phi\vert\mathcal A) \sqsubseteq M(\phi\vert \mathcal B)$.
\end{enumerate}
\end{lem}

\begin{proof}
We first prove (2) under the assumption $\langle \phi \rangle$ acts transitively on $\mathcal A$ and $\mathcal B$.
Let $\mathcal A = \langle \phi \rangle [A]$ and $\mathcal B = \langle \phi \rangle [B]$, and let $M(\phi^k|A) = \langle A, xt^{k}\rangle$ and $M(\phi^\ell|B) = \langle B, yt^\ell\rangle$ be the associated subgroups of $M(\phi)=\free\rtimes_\phi\ZZ$ as identified above. 
Now, if $\mathcal A \sqsubseteq \mathcal B$, i.e.~$A \le B$, then $x\Phi^k(B)x^{-1} \cap B$ is not trivial. Since $\mathcal{B}$ is malnormal, we get $[\Phi^k(B)] = [B]$ and $k$ is a multiple of $\ell$, the size of $\mathcal B$; therefore, $\langle B, xt^k \rangle$ is a finite index subgroup of $M(\phi^{\ell}\vert B)$, and $M(\phi^k|A) = \langle A, xt^k\rangle$  is a subgroup of $M(\phi^\ell|B)$, i.e. $M(\phi|\mathcal A) \sqsubseteq M(\phi|\mathcal B)$. Conversely, if $M(\phi^k\vert A) \le M(\phi^\ell \vert B)$, then $A = M(\phi^k\vert A)\cap \free$ is a subgroup of $B = M(\phi^\ell \vert B)\cap \free$. 

We can now prove (1). Suppose $[M(\phi^{k_1}|A_1)]$ and $[M(\phi^{k_2}|A_2)]$ are distinct elements of $M(\phi|\mathcal A)$.
We need to show $M(\phi^{k_1}|A_1)$ cannot be a subgroup of $M(\phi^{k_2}|A_2)$.
By definition of $M(\phi|\mathcal A)$, the orbits $\langle \phi \rangle [A_1]$ and $\langle \phi \rangle [A_2]$ in~$\mathcal A$ are distinct. In particular, $\langle \phi \rangle [A_1]$ is not nested in $\langle \phi \rangle [A_2]$ as malnormal subgroup systems of~$\free$. By the previous paragraph, $M(\phi^{k_1}|A_1)$ is not a subgroup of $M(\phi^{k_2}|A_2)$.
The same argument proves $M(\phi|\mathcal B)$ is a subgroup system of~$M(\phi)$.

Finally, we prove (2) in full generality.
Suppose $\mathcal A \sqsubseteq \mathcal B$, and let $[M(\phi^k|A)] \in M(\phi|\mathcal A)$.
By $\phi$-invariance, we have $\langle \phi \rangle[A] \sqsubseteq \langle \phi \rangle[B]$ for some $[B] \in \mathcal B$.
We may assume $A \le B$. By the first paragraph, $M(\phi^k|A) \le M(\phi^\ell|B)$, where $[M(\phi^\ell|B)] \in M(\phi|\mathcal B)$.
As $[M(\phi^k|A)] \in M(\phi|\mathcal A)$ was arbitrary, we have $M(\phi|\mathcal A) \sqsubseteq M(\phi|\mathcal B)$.
Conversely, suppose $M(\phi|\mathcal A) \sqsubseteq M(\phi|\mathcal B)$, and let $[A] \in \mathcal A$.
Then $M(\phi^k|A) \le M(\phi^\ell|B)$ for some $[B] \in \mathcal B$.
By the first paragraph again, $A \le B$. So $\mathcal A \sqsubseteq \mathcal B$ as $[A] \in \mathcal A$ was arbitrary.
\end{proof}

There is a {similar correspondence} describing meets of mapping tori.

\begin{lem}\label{lem:meet_of_mapping_tori}
If $\mathcal{A}$ and $\mathcal{B}$ are two $\phi$-invariant malnormal subgroup systems of $\free$ with finite type, then $M(\phi|(\mathcal A \wedge \mathcal B)) = M(\phi|\mathcal A) \wedge M(\phi|\mathcal B)$.
\end{lem}

\begin{proof}
Recall that $\mathcal A \wedge \mathcal B$ is malnormal with finite type \cite[Proposition~2.1]{Neumann}.
It is also $\phi$-invariant by uniqueness of the meet.
Let $[C] \in \mathcal A \wedge \mathcal B$, i.e. $C = A \cap B$ is nontrivial for some $[A] \in \mathcal A$ and $[B] \in \mathcal B$.
Since $M(\phi^k|A)$ and $M(\phi^\ell|B)$ are normalizers of~$A$ and~$B$ in $M(\phi)$, their intersection is in $M(\phi^n|C)$, the normalizer of~$C$. 
Conversely, let $M(\phi^n|C) = \langle C, zt^n \rangle$ for some $z \in \free$.
By malnormality and $\phi$-invariance of~$\mathcal A$, the power~$n$ is a multiple of~$k$, and $zt^n$ normalizes~$A$, i.e.~$M(\phi^n|C) \le M(\phi^k|A)$. 
Thus $M(\phi^n|C) = M(\phi^k|A) \cap M(\phi^\ell|B)$, and $[M(\phi^n|C)]$ is in $M(\phi|\mathcal A) \hat\wedge \mathcal M(\phi|\mathcal B)$.
By the definition of the meet, $\{ [M(\phi^n|C)] \} \sqsubseteq M(\phi|\mathcal A) \wedge \mathcal M(\phi|\mathcal B)$.
As $[C] \in \mathcal A \wedge \mathcal B$ was arbitrary, we have shown $M(\phi|(\mathcal A \wedge \mathcal B)) \sqsubseteq  M(\phi|\mathcal A) \wedge M(\phi|\mathcal B)$.

Conversely, let $[G] \in M(\phi|\mathcal A) \wedge M(\phi|\mathcal B)$. 
By definition of the meet, the free-by-cyclic subgroup $G = M(\phi^k|A) \cap M(\phi^\ell|B)$ for some $[A] \in \mathcal{A}$ and $[B] \in \mathcal{B}$.
Moreover, the intersection $M(\phi^k|A) \cap M(\phi^\ell|B) = \langle C, zt^n \rangle$ for the nontrivial subgroup $C = A \cap B$. 
Therefore, $[C] \in \mathcal A \wedge \mathcal B$ and  $[G] = [M(\phi^n|C)] \in M(\phi|(\mathcal A \wedge \mathcal B))$.
As $[G] \in M(\phi|\mathcal A) \wedge M(\phi|\mathcal B)$ was arbitrary, this paragraph shows $M(\phi|\mathcal A) \wedge M(\phi|\mathcal B) \subseteq M(\phi|(\mathcal A \wedge \mathcal B))$. 
In particular,  $M(\phi|\mathcal A) \wedge M(\phi|\mathcal B) \sqsubseteq M(\phi|(\mathcal A \wedge \mathcal B))$.
\end{proof}

In general, the subgroup system $M(\phi\vert \mathcal A)$ need not be malnormal in $M(\phi)$. We now characterize this failure.  
Following \cite{Dahmani-Li}, we say distinct subgroups $A,B$ are \emph{$\phi$-twins} if there exists $n\ge 1$ and $w\in \free$ so that the $\free$-automorphism $\Psi \colon g\mapsto w \Phi^n(g) w^{-1}$ simultaneously preserves the two subgroups.
A $\phi$-invariant malnormal subgroup system $\mathcal{A}$ of $\free$ \emph{has $\phi$-twins} if there are $\phi$-twins $A,B$ with $[A],[B]\in \mathcal A$.

\begin{lem}\label{lem:malnormalnotwin}
Let $\mathcal A$ be a $\phi$-invariant malnormal subgroup system of $\free$ with finite type. Then $M(\phi\vert \mathcal A)$ is malnormal in $M(\phi)$ if and only if $\mathcal A$ has no $\phi$-twins.
\end{lem}
\begin{proof}

For $[A] \in \mathcal A$, let $M(\phi^k|A) = \langle A, xt^{k}\rangle$ be the associated subgroup of $M(\phi) = \free\rtimes_\phi\ZZ$ identified above. 
By construction, $[M(\phi^k|A)] \in M(\phi|\mathcal A)$, and all conjugacy classes in $M(\phi|\mathcal A)$ are of this form.
Let $n \ge 1$, $w \in \free$, and $\Psi \colon \free \to \free$ be the automorphism given by $g\mapsto w\Phi^n(g)w^{-1}$.
If~$\Psi$ preserves~$A$, then $\langle A, wt^{n}\rangle \le M(\phi^k|A)$ (as~$A$ is malnormal in~$\free$). So $wt^n\in M(\phi^k|A)$. Conversely, if $wt^n\in M(\phi^k|A)$, then~$\Psi$ preserves~$A$ (as~$A$ is normal in $M(\phi^k|A)$).

We first show that $M(\phi^k|A)$ is malnormal in $M(\phi)$ if and only if the orbit $\langle \phi \rangle[A]$ in $\mathcal A$ has no $\phi$-twins.
Let $h = yt^\ell \in M(\phi) \setminus M(\phi^k|A)$ and set $B \defeq y\Phi^\ell(A)y^{-1} \neq A$.
Note that $[B] = \phi^\ell[A]$,  the conjugate $h M(\phi^k|A) h^{-1} = M(\phi^k|B)$, and $M(\phi^k|A) \cap M(\phi^k|B) \cap \free = A \cap B$ is trivial as $B \neq A$ and $\langle \phi \rangle[A]$ is malnormal in $\free$.
So $M(\phi^k|A) \cap M(\phi^k|B)$ is nontrivial if and only if it contains some element $wt^n$ with $w \in \free$ and $n \ge 1$; equivalently, $A,B$ are $\phi$-twins.

To conclude, suppose no $\langle \phi \rangle$-orbit in $\mathcal A$ has $\phi$-twins.
Choose distinct orbits $\langle \phi \rangle[A], \langle \phi \rangle[B]$ in $\mathcal A$.
Then the conjugacy classes $[M(\phi^k|A)]$, $[M(\phi^k|B)]$ in $M(\phi|\mathcal A)$ are distinct.
By the same argument as in the previous paragraph, $M(\phi^k|A) \cap M(\phi^k|B)$ is nontrivial if and only if 
$A,B$ are $\phi$-twins.
\end{proof}

In fact, all malnormal free-by-cyclic subgroup systems of finite type arise in this way:

\begin{lem}\label{lem:FbC_SS_MT}
If $\mathcal{G}$ is a malnormal free-by-cyclic subgroup system of $\GG$ with finite type, then $\mathcal G = M(\phi|\mathcal{A})$ for some $\phi$-invariant malnormal subgroup system $\mathcal{A}$ of $\free$ with finite type.

\end{lem}
\begin{proof}
Let $\pi \colon \GG \to \ZZ$ be the natural projection for the presentation $\GG \defeq \free \rtimes_\phi \ZZ$.
By \cref{lem:feighnhandel}, for every $[G] \in \mathcal{G}$, $G = \langle A, xt^k \rangle$ for some $x \in \free$, where $A = G\cap \free$ is a $\phi$-periodic finitely generated subgroup of $\free$ and $k \ge 1$ generates $\pi(G)$. 
Notice that $A$ is malnormal in $\free$.
Indeed, suppose $hAh^{-1} \cap A$ is not trivial for some $h \in \free$. 
Then $hGh^{-1}\cap G$ is not trivial.
By malnormality of~$G$ in $\GG$, we have $h \in G$ and $h \in \free \cap G = A$.
Let $\kappa \ge 1$ be the size of the $\langle \phi \rangle$-orbit of $[A]$ in $\free$.
We claim that $\kappa = k$. 
As $xt^k A t^{-k}x^{-1} = A$, we have $\phi^k[A] = [A]$ and $k$ is multiple of $\kappa$.
If $\kappa < k$, then $yt^{\kappa} A t^{-\kappa}y^{-1} = A$ for some $y \in \free$, and $yt^\kappa \notin G$ (as $\kappa$ does not generate $\pi(G))$;
therefore, $yt^\kappa G t^{-\kappa}y^{-1} \cap G$ is not trivial, contradicting that $G$ is malnormal in $\GG$.

Define $\mathcal{A}$ to be the collection of $\free$-conjugacy classes $[A]$ corresponding to all $[G] \in \mathcal{G}$. Consider any $[A_1],[A_2]\in \mathcal{A}$, where $A_i = G_i\cap \free$ with $[G_i]\in \mathcal{G}$.
If $A_1 \ne A_2$, then $G_1\ne G_2$. Since $\mathcal{G}$ is malnormal, this implies $A_1\cap A_2 \le G_1\cap G_2$ is trivial. Therefore, $\mathcal A$ is a malnormal subgroup system. 
If $[G_1] = [G_2]$, say with $G_2 = (w t^m)G_1(w t^m)^{-1}$, then $A_2 = w \Phi^m(A_1) w^{-1}$, and so the $\free$-conjugacy classes $[A_1],[A_2]$ are in the same $\phi$-orbit. 
Since $\mathcal{G}$ has finite type and each $[A]$ has finite $\phi$-orbit, the subgroup system $\mathcal{A}$ has finite type.
Note that the $\free$-conjugacy class $\phi[A_1] = [\Phi(A_1)]$ corresponds to $[G_1] = [t G_1 t^{-1}] \in \mathcal G$,
i.e.~$\phi[A_1] \in \mathcal A$.
So $\mathcal A$ is $\phi$-invariant.
Finally, $\mathcal G = M(\phi|\mathcal A)$ as $k$ is the size of the orbit $\langle \phi \rangle[A] \subseteq \mathcal A$ by the previous paragraph.
\end{proof}

\subsection{Malnormal within a subgroup system}

When $\mathcal{H}\sqsubseteq \mathcal{G}$ are nested free-by-cyclic subgroup systems of $\mathbb G$,  we shall colloquially refer to $\mathcal{H}$ as a subgroup system of $\mathcal{G}$. In this context, we shall need the following notion of malnormality:

\begin{defi}\label{def:malnormalinsystem}
Let $\mathcal{G}, \mathcal H$ be free-by-cyclic subgroup systems of $\GG$ such that $\mathcal H \sqsubseteq \mathcal G$. 
We say~$\mathcal H$ is \emph{malnormal} in~$\mathcal{G}$ if it is the meet of~$\mathcal{G}$ with a malnormal free-by-cyclic subgroup system of~$\GG$.
Similarly, let $\mathcal A, \mathcal B$ be $\phi$-invariant malnormal subgroup systems of~$\free$ with finite type such that $\mathcal B \sqsubseteq \mathcal A$.
Then~$\mathcal B$ \emph{has $\phi|\mathcal A$-twins} if it is the meet of~$\mathcal A$ with a $\phi$-invariant malnormal subgroup system of~$\free$ with finite type and $\phi$-twins.
\end{defi}

The following extends \cref{lem:FbC_SS_MT}, the characterization of malnormal free-by-cyclic subgroup systems:

\begin{lem}\label{lem:FbC_SS_MappingTorus}
Let $\mathcal{A}$ be  a $\phi$-invariant malnormal subgroup system of $\free$ with finite type. 
If $\mathcal{H}$ is a malnormal free-by-cyclic subgroup system in $M(\phi|\mathcal{A})$ with finite type, then $\mathcal{H} = M(\phi|\mathcal{D})$ for some $\phi$-invariant malnormal subgroup system $\mathcal{D} \sqsubseteq \mathcal{A}$ of~$\free$ with finite type. 
\end{lem}
\begin{proof} We have that $\mathcal{H} = \mathcal{H}' \wedge M(\phi|\mathcal{A})$ for a malnormal free-by-cyclic subgroup system $\mathcal{H}'$ of $\GG$ with finite type. 
By \cref{lem:FbC_SS_MT}, $\mathcal{H}' = M(\phi|\mathcal{C})$ for some $\phi$-invariant malnormal subgroup system $\mathcal{C}$ of $\free$ with finite type. 
Now, by \cref{lem:meet_of_mapping_tori}, $M(\phi|\mathcal{C}) \wedge M(\phi|\mathcal{A}) = M(\phi|(\mathcal{C}\wedge \mathcal{A}))$. Then $\mathcal{D} \defeq \mathcal{C} \wedge \mathcal{A}$ is the desired $\phi$-invariant malnormal subgroup system of~$\free$ with finite type. 
 \end{proof}

\subsection{Polynomial growth}
\label{sec:polynomial_growth}

Let $\phi \in \Out(\free)$ and $g \in \free$, and 
fix a basis of $\free$. 
The \emph{length} of the conjugacy class $[g]$ of $g$ with respect to the basis is 
\[\norm{[g]} \defeq \min_{g' \in [g]} \abs{g'},\]
where $\abs{g'}$ is the \emph{word length} of $g$ with respect to the chosen basis.  

We say that $g$ has \emph{polynomial growth under (forward) $\phi$-iteration} if there exists a polynomial $P \in \RR[X]$ such that 
\[ \norm{\phi^n[g]} \leq P(n) \quad \text{for every }n \in \NN.\]

This property does not depend on the choice of the basis. We say $\phi$ is \emph{polynomially growing} if every element has polynomial growth under $\phi$-iteration; otherwise,~$\phi$ is \emph{exponentially growing}.

In~\cite[Proposition 1.4]{Levitt09}, Levitt showed that there is a unique subgroup system~$\mathcal P(\phi)$ of $\free$ whose subgroup representatives are the maximal subgroups consisting entirely of elements with polynomial growth under $\phi$-iteration; further, $\mathcal P(\phi)$ is a $\phi$-invariant malnormal subgroup system with finite type.
Note that $\phi$ is exponentially growing if and only if the subgroup system~$\mathcal P(\phi)$ is proper. 
Define 
\[\mathcal P(M(\phi))\defeq M(\phi|\mathcal{P}(\phi)),\] which is a free-by-cyclic subgroup system  of the mapping torus group $M(\phi) \defeq \free \rtimes_\phi \ZZ$.
It follows from the proof of~\cite[Proposition 1.4]{Levitt09} that $\mathcal P(\phi)$ has no $\phi$-twins (see also \cref{cor:NSS} below), so $\mathcal P(M(\phi))$ is malnormal in~$M(\phi)$.
We shall observe later in \S\ref{subsec:growthtypecanon} that the subgroup system $\mathcal P(M(\phi))$ is independent of the presentation $M(\phi)$.

Let $\mathcal A$ be a $\phi$-invariant malnormal subgroup system of $\free$ with finite type. Then $\phi|\mathcal{A}$ is \emph{polynomially growing} if every $\mathcal A$-peripheral element in $\free$ has polynomial growth under $\phi$-iteration; equivalently, $\phi|\mathcal{A}$ is polynomially growing if and only if $\mathcal{A} \sqsubseteq \mathcal P(\phi)$.
Define $\mathcal P(\phi|\mathcal A) \defeq \mathcal A \wedge \mathcal P(\phi)$ and $\mathcal P(M(\phi|\mathcal A)) \defeq M(\phi|\mathcal A) \wedge \mathcal P(M(\phi))$.
Note that $\mathcal P(\phi|\mathcal A)$ has finite type too.

\subsection{Weak attraction theory and attracting laminations}
\label{subsec:weakattraction}

Let $\partial_\infty \free$ be the Gromov boundary of $\free$ and 
\[\partial^2\free\defeq(\partial_\infty \free \times \partial_\infty \free-\Delta)/{\sim}\] 
be the \emph{double boundary} of $\free$, where $\Delta$ denotes the diagonal subset and $\sim$ is the equivalence relation generated by the flip $(x,y){\sim}(y,x)$. 
The double boundary $\partial^2\free$ is equipped with the quotient topology. 
The (continuous) action of $\free$ on $\partial_\infty \free$ induces a diagonal action of $\free$ on $\partial^2 \free$;
moreover, this extends to an action of $\Aut(\free)$ on $\partial^2 \free$ via the natural identification $\free \cong \Inn(\free)$.

The \emph{space of lines} of $\free$ is the quotient $\mathfrak{B}(\free)\defeq\free \backslash \partial^2\free$ of the double boundary by the $\free$-action. 
A \emph{line} of $\free$ is an element of $\BB(\free)$.
The conjugacy class $[g]$ of a nontrivial $g \in \free$ naturally determines a (periodic) line ${[[g]]} \in \BB(\free)$, called the \emph{axis} of $[g]$. 
The quotient topology on $\BB(\free)$ is also called the \emph{weak topology} to emphasize that it is non-Hausdorff.
A \emph{lamination} of $\free$ is a closed nonempty subset of $\BB(\free)$.
Note that $\Out(\free)$ naturally acts on the space $\BB(\free)$ of lines by homeomorphisms.

Fix an outer automorphism $\phi \in \Out(\free)$. Bestvina--Feighn--Handel associated to $\phi$ a collection $\mathcal{L}^+(\phi)$ of closed subsets $\Lambda \subseteq \BB(\free)$ called the \emph{attracting laminations} of $\phi$~\cite[Definition~3.1.5]{BesFeiHan00}. 
We give the precise definition of an attracting lamination in \cref{app:A}.
 For $\Lambda \in \mathcal{L}^+(\phi)$ and $g \in \free$, we say that $g$ is \emph{weakly attracted} to $\Lambda$ (under forward $\phi$-iteration) if $\Lambda$ (as a subset of $\BB(\free)$) is contained in the accumulation points of the sequence $({[[\phi^n(g)]]})_{n \in \NN}$ (in the weak topology).
The set $\mathcal{L}^+(\phi)$ is finite (see proof of \cite[Lemma~3.1.13]{BesFeiHan00}), and it is empty if and only if $\phi$ is polynomially growing.

\begin{defi}
    The set $\mathcal{L}^+(\phi)$ of all attracting laminations for $\phi$ is partially ordered by inclusion.
    The \emph{(lamination) depth} $\delta(\phi)$ of $\phi$ is the maximum size of a chain in $\mathcal L^+(\phi)$,
    and the \emph{depth spectrum} $\delta\mathcal S(\phi)$ is the set consisting of the sizes of all maximal chains in $\mathcal L^+(\phi)$.
\end{defi}

We first record the following result on the invariance of the depth of an outer automorphism under passing to positive powers and restrictions to finite index invariant subgroups. We postpone its proof to  \cref{app:passingtofinite} as it relies on the exact definition of attracting laminations.

\begin{prop}\label{Prop:powers} Suppose $A \le \free$ has finite index, $\Phi \colon \free \to \free$ is an automorphism such that $\Phi(A) = A$, and $\Psi \colon A \to A$ is the restriction of $\Phi$ to $A$.
Set $\phi \defeq [\Phi]$ in $\Out(\free)$ and $\psi \defeq [\Psi]$ in $\Out(A)$. Then $\delta(\phi^m) = \delta(\phi) = \delta(\psi)$ and  $\delta\mathcal S(\phi^m) = \delta\mathcal S(\phi) = \delta\mathcal S(\psi)$ for all $m \ge 1$.

In fact, there is a natural surjective strict-order preserving map $\pi \colon \mathcal L^+(\psi) \to \mathcal L^+(\phi)$, and chains in $\mathcal L^+(\phi)$ can be lifted (via~$\pi$) to chains of the same length in $\mathcal L^+(\psi)$.
\end{prop}

The next example describes the poset of attracting laminations for an outer automorphism.
We will use it as a running example to illustrate our constructions.

\begin{example}\label{ex:lamination_poset}
Consider $\free = \langle a, b, c, d, e, f \rangle$ and an outer automorphism $\phi \in \Out(\free)$ with representative
\[\Phi \colon a \mapsto ab,\; b \mapsto bab,\; c \mapsto bcd,\; d \mapsto dcd,\; e \mapsto aba^{-1}b^{-1}ef,\; f \mapsto fef.\]
See \Cref{fig:ex_lamination_poset}.

\tikzset{node style ge/.style={circle}}

\begin{figure}
       
       \centering

\begin{minipage}[b]{0.35\textwidth}
\hspace{0.5cm}
\begin{tikzpicture}[scale=1.6]

\draw[color=green!80!black] (0,0) plot[smooth, tension=1] coordinates {(0,0) (10:.7) (0:1.15) (350:.7) (0,0)};

\draw[color=green!80!black] (0,0) plot[smooth, tension=1] coordinates {(0,0) (170:.7) (180:1.15) (190:.7) (0,0)};

 \node [right][color=green!80!black]       at (0:1.15)  {$d$};

 \node [left][color=green!80!black]       at (180:1.15)  {$c$};

\draw[color=red] (0,0) plot[smooth, tension=1] coordinates {(0,0) (310:.7) (300:1.15) (290:.7) (0,0)};

\draw[color=red] (0,0) plot[smooth, tension=1] coordinates {(0,0) (250:.7) (240:1.15) (230:.7) (0,0)};

 \node [left][color=red]        at (240:1.15)  {$a$};
 \node [right][color=red]         at (300:1.15)  {$b$};

\draw[color=blue] (0,0) plot[smooth, tension=1] coordinates {(0,0) (70:.7) (60:1.15) (50:.7) (0,0)};

\draw[color=blue] (0,0) plot[smooth, tension=1] coordinates {(0,0) (130:.7) (120:1.15) (110:.7) (0,0)};

 \node [left][color=blue]       at (120:1.15)  {$e$};
 \node [right][color=blue]        at (60:1.15)  {$f$};

\end{tikzpicture}
\end{minipage}
\begin{minipage}[b]{0.35\textwidth}
\hspace{0.8cm}
\begin{tikzpicture}[scale=0.6][baseline=(A.center)]
  \tikzset{BarStyle/.style =   {opacity=.4,line width=4 mm,line cap=round,color=#1}}

\matrix (A) [matrix of math nodes ,column sep=0 mm] 
{ \textcolor{blue}{e}  & \mapsto  & \textcolor{red}{ab\overline{a}\overline{b}}\textcolor{blue}{ef} \\
 \textcolor{blue}{f} &\mapsto & \textcolor{blue}{fef}  \\
    \textcolor{green!80!black}{c} & \mapsto & \textcolor{red}{b}\textcolor{green!80!black}{cd} \\
     \textcolor{green!80!black}{d} & \mapsto  &\textcolor{green!80!black}{dcd}  \\
       \textcolor{red}{a} & \mapsto  &\textcolor{red}{ab}   \\
        \textcolor{red}{b} & \mapsto  &\textcolor{red}{bab}  \\
};

\end{tikzpicture}

\end{minipage}
\begin{minipage}[b]{0.25\textwidth}
\begin{tikzpicture}[scale=1.1][baseline=(A.center)]
  \tikzset{BarStyle/.style =   {opacity=.4,line width=4 mm,line cap=round,color=#1}}

\matrix (A) [matrix of math nodes, nodes = {node style ge},,column sep=0 mm] 
{ \textcolor{green!80!black}{\Lambda_{cd}} \ & \ &\textcolor{blue}{\Lambda_{ef}} \\
  \ & \  & \ &\ \\
  \textcolor{red}{\Lambda_{ab}} &\ & \\ &  \\
    \ & \ & \ & \ & \ \\
};
\draw[thick] (A-1-1) to (A-3-1);

\end{tikzpicture}
\end{minipage}

\caption{Illustration for \cref{ex:lamination_poset}}
    \label{fig:ex_lamination_poset}

\end{figure}

Let the lamination $\Lambda_{ab}$ be the set of accumulation points of the sequence $([[\phi^n(g)]])_{n \in \mathbb{N}}$ for $g = a$ (or $g = b$). 
Similarly, define $\Lambda_{cd}$ and $\Lambda_{ef}$.
An application of Bestvina-Feighn-Handel relative train track theory will show that  
 $\mathcal{L}^+(\phi) = \{\Lambda_{ab}, \Lambda_{cd}, \Lambda_{ef}\}$.
Note that $c$ (resp.~$e$) is not weakly attracted to $\Lambda_{ef}$ (resp.~$\Lambda_{cd}$) since $c,d$ (resp.~$e,f$) do not map over $e,f$ (resp.~$c,d$).
So the two laminations are not nested.
For large $n \gg 1$, the line $[[\phi^n(c)]]$ contains a long segment of $[[\phi^n(b)]]$;
therefore, $c$ is weakly attracted to $\Lambda_{ab}$, and $\Lambda_{ab} \subseteq \Lambda_{cd}$.
Moreover, $e$ is not weakly attracted to $\Lambda_{ab}$  since $\Phi(aba^{-1}b^{-1}) = aba^{-1}b^{-1}$.  Thus, the only nesting in $\mathcal L^+(\phi)$ is $\Lambda_{ab} \subseteq \Lambda_{cd}$, and the depth $\delta(\phi) = 2$.
Finally, the line $[[aba^{-1}b^{-1}]]$ is in $\Lambda_{ef}$ since $[[\phi^n(e)]]$ contains arbitrarily long segments of $[[aba^{-1}b^{-1}]]$ for $n \gg 1$. 
\end{example}

In the above example, the outer automorphism $\phi$ acts trivially on $\mathcal L^+(\phi)$, but that is not always the case.
In general, we will consider the $\phi$-orbits of laminations. 

The poset $\mathcal L^+(\phi)$ admits a natural order-preserving action of the cyclic group $\langle \phi \rangle$.
We denote the $\langle \phi \rangle$-orbit of $\Lambda \in \mathcal{L}^+(\phi)$ by $[\Lambda]$. Since every $\Lambda \in \mathcal{L}^+(\phi)$ has a finite $\langle \phi \rangle$-orbit, $[\Lambda]$ is a closed subset of $\BB(\free)$ and hence there is a natural partial order on the set of $\langle \phi \rangle$-orbits of laminations. We denote the poset of $\langle \phi \rangle$-orbits by $\mathcal{L}^+_\phi$.
As $\phi$ acts on $\mathcal L^+(\phi)$ by an order-isomorphism, the depth~$\delta(\phi)$ is also the maximum size of a chain in~$\mathcal L^+_\phi$.
Similarly, the depth spectrum $\delta\mathcal S(\phi)$ consists of the sizes of all maximal chains in $\mathcal L^+_\phi$.
Our goal for this paper is to show that the poset $\mathcal{L}^+_\phi$ is naturally a group invariant of $\free \rtimes_\phi \ZZ$.

Each finitely generated subgroup $A$ of $\free$ has its own double boundary $\partial^2 A$ that is naturally identified with a closed subset of $\partial^2 \free$ using the embedding induced by the inclusion $A\hookrightarrow \free$. 
If $A$ is also malnormal, then the image of $\partial^2 A$ in $\BB(\free)$ under the quotient is a closed subset identified with the intrinsic space of lines $\BB(A)$ of $A$. Further, suppose $\mathcal A$ is a malnormal subgroup system of $\free$ with finite type. 
Then $\bigcup \{\partial^2 A \mid [A]\in \mathcal A\}$ is an $\free$-invariant closed subset of $\partial^2 \free$ whose  image $\BB(\mathcal A)$ is a closed subset of $\BB(\free)$ and a finite disjoint union of the sets $\BB(A)$ for $[A]\in \mathcal A$. 
Note that, for any malnormal subgroup systems $\mathcal A_1, \mathcal A_2$ of $\free$ with finite type,  $\BB(\mathcal A_1) \cap \BB(\mathcal A_2) = \BB(\mathcal A_1 \wedge \mathcal A_2)$, and if $\mathcal A_1 \sqsubseteq \mathcal A_2$, then $\BB(\mathcal A_1) \subseteq \BB(\mathcal A_2)$.
Finally, for any given subset $L \subseteq \BB(\free)$ of lines in $\free$, we say that $\mathcal A$ \emph{carries} $L$ if  $L \subseteq \BB(\mathcal A)$.

Now assume $\mathcal A$ is a $\phi$-invariant malnormal subgroup system of $\free$ with finite type.
An \emph{attracting lamination} of $\phi|\mathcal A$, denoted $\Lambda \in \mathcal L^+(\phi|\mathcal A)$, is an attracting lamination of $\phi$ that is carried by $\mathcal A$.
So $\mathcal L^+(\phi|\mathcal A)$ is a subset of $\mathcal L^+(\phi)$.
As with $\phi$, the subset $\mathcal L^+(\phi|\mathcal A)$ is empty if and only if $\phi|\mathcal A$ is polynomially growing.
Since $\mathcal A$ is $\phi$-invariant, the $\langle \phi \rangle$-action on $\BB(\free)$ will preserve $\BB(\mathcal A)$.
In particular, the $\langle \phi \rangle$-orbit $[\Lambda]$ is a closed subset of $\BB(\mathcal A)$ for any $\Lambda \in \mathcal L^+(\phi|\mathcal A)$.
Let $\mathcal L^+_{\phi|\mathcal A} \subseteq \mathcal L^+_\phi$ denote the partially ordered subset of $\langle \phi \rangle$-orbits in $\mathcal L^+(\phi|\mathcal A)$.
We will need the following folklore lemma, whose proof is postponed to \cref{app:genericapprox}.

\begin{lem}\label{lem:approximategenline}
For $[\Lambda] \in \mathcal L^+_{\phi|\mathcal A}$, some $\mathcal A$-peripheral element $g \in \free$ has the following property: 
\begin{quote}for any $[\Lambda']\in \mathcal L^+_{\phi|\mathcal A}$, $g$ is weakly attracted to $[\Lambda']$ if and only if  $[\Lambda'] \subseteq [\Lambda]$. 
\end{quote}
\end{lem}

\subsection{JSJ decompositions of relatively one-ended free groups}
\label{sec:JSJ_theory}

This section follows~\cite[\S3.3]{GuirardelLevitt2015} (see also~\cite{GuirardelLevitt2016} for a general presentation of JSJ decompositions of groups). We state the results of the section for the free group $\free$ for simplicity but all the statements extend to torsion-free relatively hyperbolic groups.

Let~$\mathcal{N}$ be a proper malnormal subgroup system of~$\free$ with finite type. Recall that, in this case, the group $\free$ is hyperbolic relative to the set $\mathcal{N}$.
Partition $\mathcal N = \mathcal P \sqcup \mathcal H$ into conjugacy classes of noncyclic and cyclic subgroups respectively. In the general theory, only the noncyclic subset $\mathcal P$ (not $\mathcal N$) is necessarily a subgroup system of $\free$. However, in this section, we assume that~$\mathcal N$ is a proper malnormal subgroup system. 
A subgroup~$H$ of~$\free$ is \emph{(relatively) elementary} if it is cyclic or $\mathcal{N}$-peripheral.
We denote by $\Out(\free,\mathcal N)$ the setwise stabilizer in $\Out(\free)$ of~$\mathcal N$.

An \emph{$(\free,\mathcal N)$-tree} is an $\free$-tree in which every conjugacy class in $\mathcal N$ is elliptic;
it is \emph{elementary} if edge stabilizers are elementary.
Similarly, a \emph{splitting} of~$(\free,\mathcal N)$ is a splitting of~$\free$ consisting of $(\free,\mathcal N)$-trees.
We say that $\free$ is \emph{one-ended relative to} $\mathcal N$ if there is no free splitting of $(\free, \mathcal N)$. 
In that case, Guirardel--Levitt associate to the pair $(\free,\mathcal N)$ a canonical 
elementary 
$(\free,\mathcal N)$-tree $T_{\mathcal N}$ called the \emph{(elementary) JSJ tree}.

\begin{theo}[see~{\cite[Theorem~9.18]{GuirardelLevitt2016}}]\label{Thm:JSJ}
Let~$\mathcal{N}$ be a proper malnormal subgroup system of~$\free$ with finite type and let $\mathcal{N}= \mathcal P \sqcup \mathcal H$ be the partition into conjugacy classes of noncyclic and cyclic subgroups respectively. 
If~$\free$ is one-ended relative to $\mathcal{N}$, then there is a canonical 
elementary 
$(\free,\mathcal N)$-tree~$T_{\mathcal N}$ with the following properties.

\begin{enumerate}
\item Any vertex with an elementary stabilizer is adjacent to one with a nonelementary stabilizer.
\item The tree $T_{\mathcal N}$ is preserved by $\Out(\free,\mathcal N)$.

Let $\Out^0(\free,\mathcal N)$ be the maximal finite index subgroup of $\Out(\free,\mathcal N)$ which acts trivially on the quotient graph $\free \backslash T_{\mathcal N}$.
\item For each vertex $v \in VT_{\mathcal N}$, exactly one of the following holds:
\begin{enumerate}
\item The stabilizer $\Stab(v)$ of $v$ is elementary; moreover, either it is cyclic or its conjugacy class is in $\mathcal{P}$.
\item 
The natural homomorphism \[\Out^0(\free,\mathcal N) \to \Out(\Stab(v))\] has finite image (see~\cite[Theorem~3.9]{GuirardelLevitt2015}).
\item The stabilizer $\Stab(v)$ is identified with the fundamental group of a compact surface~$\Sigma$. Edge groups and representatives of conjugacy classes in~$\mathcal N$ intersect $\Stab(v)$ in the fundamental groups of the boundary components of~$\Sigma$ and the image
of the natural homomorphism 
\[\Out^0(\free,\mathcal N) \to \Out(\Stab(v))\] is infinite and contained in the mapping class group of~$\Sigma$. \qed
\end{enumerate}
\end{enumerate}
\end{theo} 

\begin{rmq}
    The tree $T_\mathcal{N}$ is not necessarily a splitting since it can be a point. This is why we call~$T_\mathcal{N}$ a \emph{JSJ tree} instead of a \emph{JSJ splitting}. However, even when $T_\mathcal{N}$ is a point, the trichotomy for the vertex stabilizers in \cref{Thm:JSJ}(3) gives information about $\Out(\free,\mathcal{N})$ that will be sufficient for our purpose (see, for instance, \cref{lem:norigidvertex} and the proof of \cref{Prop:invariantmalnormal}).
 \end{rmq}

Let $T_{\mathcal N}$ be the JSJ tree for $(\free,\mathcal N)$ as in \cref{Thm:JSJ}.
A vertex $v \in VT_{\mathcal N}$ is: \emph{elementary} if it satisfies Condition~(3a); \emph{rigid} if it satisfies Condition~(3b); and finally, \emph{flexible} if it satisfies Condition~(3c).

We will also need JSJ forests for subgroup systems of $\free$. Let $\mathcal{A}=\{[A_1],\ldots,[A_k]\}$ be a malnormal subgroup system of $\free$ with finite type such that $\mathcal{N} \sqsubsetneq \mathcal{A}$. 
A \emph{splitting} $\mathcal S = \{T_1,\ldots,T_k\}$ of the pair $(\mathcal{A}, \mathcal{N})$ is a splitting of~$\mathcal A$ such that each~$T_i$ is a class of $(A_i, \mathcal N|A_i)$-trees.
We say that $\mathcal{A}$ is \emph{one-ended relative} to $\mathcal{N}$ if there is no free splitting of $(\mathcal{A},\mathcal{N})$. 

When $\mathcal{A}$ is one-ended relative to $\mathcal{N}$, there exists a canonical 
elementary
forest $\mathcal T_\mathcal{N}$ for the pair $(\mathcal{A},\mathcal{N})$ constructed as follows. 
Pick a conjugacy class $[A] \in \mathcal{A}$. If $[A] \in \mathcal N$, then define $T^A_{\mathcal N}$ to be a point as an $A$-tree. 
Otherwise, since $\mathcal{A}$ is one-ended relative to $\mathcal{N}$, the subgroup~$A$ is one-ended relative to $\mathcal{N}|A$.  
Let $T_\mathcal{N}^A$ be the JSJ tree for $(A,\mathcal{N}|A)$ given by \cref{Thm:JSJ}. The forest $\mathcal T_\mathcal{N} \defeq \{T_\mathcal{N}^A:[A]\in \mathcal{A}\}$ is the \emph{JSJ forest} for $(\mathcal{A},\mathcal{N})$. We keep using the trichotomy rigid, flexible, or elementary to classify vertices of the JSJ forest.

\section{Subgroup systems associated to attracting laminations}\label{sec:nonattracting_ss}

Let $\phi \in \Out(\free)$.
To prove that the depth of $\phi$ is a group invariant of its mapping torus, we need to translate the dynamical information given by an attracting lamination into algebraic data. 
This is done through \emph{nonattracting subgroup systems}.
We will also introduce \emph{supporting subgroup systems}, which will be needed to prove the group invariance of the poset $\mathcal L^+_\phi$.

\subsection{Nonattracting subgroup systems}\label{subsec:nonattracting_ss}

Handel--Mosher proved that the set of all conjugacy classes of elements of $\free$ which are not weakly attracted to $[\Lambda] \in \mathcal L^+_\phi$ determines a $\phi$-invariant proper malnormal subgroup system $\mathcal{N}[\Lambda]$ with finite type, called the \emph{nonattracting subgroup system}~\cite[Theorem~F]{HandelMosher20}.
Their construction also applies to subsets of $\mathcal L^+_\phi$, and the following theorem is an alternative construction of these subgroup systems:

\begin{theo}[see~{\cite[\S3]{Mutanguha22}}]\label{Thm:NSS}
For $\phi \in \Out(\free)$ and $\mathcal L \subseteq \mathcal L^+_\phi$, there is a canonical topological tree $T_{\mathcal L}$ with a non-nesting continuous action by $\free$ such that:
\begin{enumerate}
\item the action is minimal with trivial arc stabilizers;
\item for any automorphism $\Phi \colon \free \to \free$ representing~$\phi$, there is an expanding $\Phi$-equivariant homeomorphism $h_{\mathcal L} \colon T_{\mathcal L} \to T_{\mathcal L}$; and
\item an element $g \in \free$ is elliptic in $T_{\mathcal L}$ if and only if it is not weakly attracted to any $[\Lambda] \in \mathcal L$.
\end{enumerate}
\end{theo}

A topological tree can be thought of as an $\RR$-tree without the metric.
Non-nesting actions on topological trees are the analogues of isometric actions; in particular, they allow us to classify elements of $\free$ as elliptic/loxodromic.
Finally, the definition of expanding for $\Phi$-equivariant homeomorphisms is a bit technical (due to lack of a metric) but the only property we will use is that their fixed points are unique if they exist; this is analogous to expanding homotheties.

\begin{rmq}
    Technically, the stated theorem does not appear in the cited reference;
    however, it follows readily from the discussion in the citation. For completeness, we sketch in \cref{app:toptree} how to deduce \cref{Thm:NSS}. 
\end{rmq}

Suppose $\mathcal L \subseteq \mathcal L^+_\phi$ and $T_{\mathcal L}$ is the corresponding topological tree given by \cref{Thm:NSS}.
The \emph{nonattracting subgroup system} $\mathcal N(\mathcal L)$ for $\mathcal L$ is the set of conjugacy classes of nontrivial point stabilizers of $T_{\mathcal L}$. 
For singletons $\mathcal L = \{[\Lambda]\}$ in $\mathcal L^+_\phi$, we recover Handel--Mosher's definition: $\mathcal N(\mathcal L) = \mathcal N[\Lambda]$.
When $\mathcal L = \mathcal L^+_\phi$, we recover Levitt's subgroup system: $\mathcal N(\mathcal L) = \mathcal P(\phi)$.
The \emph{nonattracting systems} for $\phi$ is the collection $\mathcal N(\phi) \defeq 
\left\{ \mathcal N[\Lambda]\,:\, [\Lambda]\in \mathcal L_\phi^+ \right\}$.

More generally, let~$T$ be a topological tree with a non-nesting action of~$\free$ that is minimal with trivial arc stabilizers.
Then the set $\mathcal{N}(T)$ of conjugacy classes of nontrivial point stabilizers of~$T$ is a malnormal subgroup system of~$\free$. 
Such subgroup systems include nonattracting subgroup systems as well as free factor systems of~$\free$.
Let $\mathcal N(T')$ be another such subgroup system of~$\free$.
Their meet $\mathcal{N}(T) \wedge \mathcal{N}(T')$ can also be described as the subgroup system consisting of the conjugacy classes of the maximal subgroups of $\free$ that are elliptic in both $T$ and $T'$. 
For instance, if $\mathcal{L} \subseteq \mathcal{L}_\phi^+$, then $\bigwedge \{ \mathcal{N}[\Lambda]:[\Lambda]\in \mathcal{L}\} = \mathcal{N}(\mathcal{L})$ by \cref{Thm:NSS}(3).

In our inductive argument, we will consider  nonattracting subgroup systems for restrictions to invariant subgroup systems. To that end, 
let $\mathcal A$ be a $\phi$-invariant malnormal subgroup system of~$\free$ with finite type.
For any $\mathcal L \subseteq \mathcal L^+_{\phi|\mathcal A}$, define $\mathcal N(\mathcal L|\mathcal A)$ to be the meet $\mathcal A \wedge \mathcal N(\mathcal L)$.
Similarly, define $\mathcal N[\Lambda|\mathcal A] \defeq \mathcal A \wedge \mathcal N[\Lambda]$ for $[\Lambda] \in \mathcal L^+_{\phi|\mathcal A}$ and $\mathcal N(\phi|\mathcal A) \defeq 
\left\{ \mathcal N[\Lambda|\mathcal A]\,:\, [\Lambda]\in \mathcal L_{\phi|\mathcal A}^+ \right\}$.
By \cref{lem:approximategenline}, for $\mathcal L \subseteq \mathcal L^+_{\phi|\mathcal A}$, the subgroup system $\mathcal N(\mathcal L|\mathcal A) = \mathcal A$ if and only if $\mathcal L$ is empty.

\begin{example}\label{ex:non-attracting_subgroup_system}
Recall from \cref{ex:lamination_poset} the automorphism of the free group $\free= \langle a, b, c, d, e, f\rangle$
\[\Phi \colon a \mapsto ab, ~ b \mapsto bab,~ c \mapsto bcd,~ d \mapsto dcd,~ e \mapsto aba^{-1}b^{-1}ef,~ f \mapsto fef. \]  
The following are some nonattracting subgroup systems of $\phi$:
\begin{align*}
    \mathcal{N}[\Lambda_{cd}] & =  [\langle a, b, e, f \rangle] & \mathcal{N}(\{[\Lambda_{cd}], [\Lambda_{ef}]\}) & = \mathcal{N}(\Lambda_{cd}) \wedge \mathcal{N}(\Lambda_{ef}) = [\langle a, b \rangle] \\
    \mathcal{N}[\Lambda_{ef}] & =  [\langle a, b, c, d \rangle] & \mathcal{N}(\mathcal{L}_{\phi}^+) & = [\langle aba^{-1}b^{-1}\rangle] = \mathcal{P}(\phi)\\
    \mathcal{N}[\Lambda_{ab}] & =  [\langle aba^{-1}b^{-1}, e, f \rangle] & \mathcal{N}[\Lambda_{ab}|\mathcal{A}] & = [\langle aba^{-1}b^{-1}\rangle ]  \text{ if } \mathcal{A} = [\langle a, b \rangle]. \qedhere
\end{align*}
\end{example}

\begin{coro}\label{cor:NSS} Let $\phi \in \Out(\free)$, $\mathcal A$ be a $\phi$-invariant malnormal subgroup system of $\free$ with finite type, and $\mathcal L \subseteq \mathcal L^+_{\phi|\mathcal A}$. 
Then the following statements hold:
\begin{enumerate}
\item  $\mathcal N(\mathcal L|\mathcal A)$ is a malnormal subgroup system of~$\free$ with finite type;

\item  $\mathcal N(\mathcal L|\mathcal A)$ is $\phi$-invariant and has no $\phi|\mathcal A$-twins; and

\item $\mathcal N(\mathcal L|\mathcal A)$ supports $\mathcal P(\phi|\mathcal A)$.
\end{enumerate}

\end{coro}

\begin{proof}
Let $T_{\mathcal L}$ be the topological tree corresponding to~$\mathcal L$ given by \cref{Thm:NSS}.
By Condition~(1) of the theorem (trivial arc stabilizers) and Gaboriau--Levitt index theory~\cite[Theorem~III.2]{GabLev} (see also~\cite[Theorem~A.7]{Mutanguhaexistence}), $\mathcal N(\mathcal L)$ is a malnormal subgroup system with finite type.
As the meet of two malnormal subgroup systems of~$\free$ with finite type, $\mathcal N(\mathcal L| \mathcal A)$ is a malnormal subgroup system of $\free$ with finite type too.

Pick an automorphism $\Phi \in \phi$. 
By Condition~(2) ($h_{\mathcal L}$ is $\Phi$-equivariant), $\mathcal N(\mathcal L)$ is $\phi$-invariant.
For an arbitrary $n \ge 1$ and $w \in \free$, define the $\free$-automorphism $\Psi \colon g \mapsto w\Phi^n(g)w^{-1}$.
Then the homeomorphism $w \cdot h_{\mathcal L}^n \colon T_{\mathcal L} \to T_{\mathcal L}$ is $\Psi$-equivariant.
By Condition~(2) ($h_{\mathcal L}$ is expanding), the map $w \cdot h_{\mathcal L}^n$ has at most one fixed point. In particular, the automorphism $\Psi$ preserves at most one nontrivial point stabilizer of $T_{\mathcal L}$ by $\Psi$-equivariance of $w \cdot h_{\mathcal L}^n$. Thus, $\mathcal{N}(\mathcal{L})$ has no $\phi$-twins. 
As the meet of two $\phi$-invariant subgroup systems, $\mathcal N(\mathcal L|\mathcal A)$ is $\phi$-invariant.
Consequently, it has no $\phi|\mathcal A$-twins.

Finally, by Condition~(3), $\mathcal N(\mathcal L)$ supports $\mathcal P(\phi)$. 
When we take the meets of both systems with $\mathcal A$, we deduce that $\mathcal N(\mathcal L|\mathcal A)$ supports $\mathcal P(\phi|\mathcal A)$.
\end{proof}

We now characterize the laminations that are carried by the nonattracting subgroup systems.

\begin{lem}\label{lem:nonattractingcarry}

For $[\Lambda_1], [\Lambda_2] \in \mathcal L^+_{\phi|\mathcal A}$,
\begin{center}$[\Lambda_1] \nsubseteq [\Lambda_2]$ (as subsets of $\BB(\mathcal A)$) if and only if $\mathcal N[\Lambda_1|\mathcal A]$ carries $[\Lambda_2]$.\end{center}
\end{lem}

\begin{proof}

Suppose $[\Lambda_1] \nsubseteq [\Lambda_2]$.
Let $T_1$ be the tree associated with the singleton $\{[\Lambda_1]\}$ given by \cref{Thm:NSS}. 
By \cref{lem:approximategenline}, there is an $\mathcal A$-peripheral element $g_2 \in \free$ that is weakly attracted to $[\Lambda_2]$ but not $[\Lambda_1]$ since $[\Lambda_1] \nsubseteq [\Lambda_2]$.
By \cref{Thm:NSS}(3), $g_2$ is elliptic in~$T_1$, i.e.~$\mathcal N[\Lambda_1]$ carries the axis $[[g_2]]$.
Since the $\mathcal A$-peripheral element $g_2$ is weakly attracted to $[\Lambda_2]$ and since $\BB(\mathcal A), \BB(\mathcal N[\Lambda_1]) \subseteq \BB(\free)$ are $\phi$-invariant closed subsets that contain $[[g_2]]$, we deduce that $\mathcal N[\Lambda_1|\mathcal A]$ carries $[\Lambda_2]$.

The converse follows from the claim that $\mathcal N[\Lambda_1]$ does not carry $[\Lambda_1]$.
Suppose, for a contradiction, that $\mathcal N[\Lambda_1]$ carries $[\Lambda_1]$, i.e.~$[\Lambda_1] \in \mathcal L^+_{\phi|\mathcal N[\Lambda_1]}$.
By \cref{lem:approximategenline}, some $\mathcal N[\Lambda_1]$-peripheral element $g_1 \in \free$ is weakly attracted to $[\Lambda_1]$, which contradicts the definition of $\mathcal N[\Lambda_1]$.
\end{proof}

As an application, we get $\mathcal L^+_{\phi|\mathcal N[\Lambda|\mathcal A]} = \mathcal L^+_{\phi|\mathcal A} \setminus \{ [\Lambda] \}$ when $[\Lambda] \in \mathcal L^+_{\phi|\mathcal A}$ is maximal. 
The next lemma is a key step in tying the poset of laminations with posets of subgroup systems.

\begin{lem}\label{lem:isoposetlaminonattracting}
For $[\Lambda_1], [\Lambda_2] \in \mathcal L^+_{\phi|\mathcal A}$, 
\begin{center}$[\Lambda_1] \subseteq [\Lambda_2]$ if and only if $\mathcal N[\Lambda_1|\mathcal A] \sqsubseteq  \mathcal N[\Lambda_2|\mathcal A]$.\end{center}
Equivalently, the collection $\mathcal N(\phi|\mathcal A)$ is naturally order-isomorphic to $\mathcal L^+_{\phi|\mathcal A}$.
\end{lem}

\begin{proof}
For $i \in \{1,2\}$, let $T_i$ be the tree associated with singleton $\{[\Lambda_i]\}$ given by \cref{Thm:NSS}. 
An element $g \in \free$ is loxodromic in $T_i$ if and only if it is weakly attracted to $[\Lambda_i]$ by \cref{Thm:NSS}(3).
If $[\Lambda_1]\subseteq [\Lambda_2]$, then any element $g \in \free$ that is elliptic in $T_1$ is also elliptic in $T_2$,
or equivalently, $\mathcal N[\Lambda_1] \sqsubseteq  \mathcal N[\Lambda_2]$; therefore, $\mathcal N[\Lambda_1|\mathcal A] \sqsubseteq  \mathcal N[\Lambda_2|\mathcal A]$. 
Conversely, if $[\Lambda_1]\nsubseteq [\Lambda_2]$,
then $\mathcal N[\Lambda_1|\mathcal A]$ carries $[\Lambda_2]$ but $\mathcal N[\Lambda_2|\mathcal A]$ does not by \cref{lem:nonattractingcarry}; therefore, $\mathcal N[\Lambda_1|\mathcal A] \nsqsubseteq  \mathcal N[\Lambda_2|\mathcal A]$.
\end{proof}

\subsection{Supporting subgroup systems}\label{subsec:support_ss}

For $[\Lambda]\in \mathcal L_\phi^+$,  the \emph{support} of $[\Lambda]$ is the subset 
\[ \Lambda^\downarrow \defeq \left\{ [\Lambda'] \in \mathcal L_\phi^+\,:\, [\Lambda'] \subseteq [\Lambda] ~\text{as subsets of}~\BB(\free) \right\}, \]
and the \emph{exterior} of $[\Lambda]$ is the subset $\Lambda^\perp \defeq \mathcal L_\phi^+ \setminus \Lambda^\downarrow$. 
By \cref{lem:nonattractingcarry}, we get $\mathcal L^+_{\phi|\mathcal N(\Lambda^\perp)} = \Lambda^\downarrow$.
The \emph{supporting subgroup system} $\mathcal S[\Lambda]$ for~$[\Lambda]$ is the meet of all free factor systems of~$\mathcal N(\Lambda^\perp)$ that carry~$\Lambda^\downarrow$, which is the free factor system of~$\mathcal N(\Lambda^\perp)$ defined using \cref{lem:transitivityffs,lem:meet_ffs,lem:meet_ffs_is_ffs}.
By uniqueness of the meet and $\phi$-invariance of~$\Lambda^\perp$, $\mathcal S[\Lambda]$ is $\phi$-invariant.
Note that $\mathcal L^+_{\phi|\mathcal S[\Lambda]} = \Lambda^\downarrow$.
The \emph{supporting systems} for $\phi$ is the collection $\mathcal S(\phi) \defeq \left\{ \mathcal S[\Lambda]\,:\, [\Lambda]\in \mathcal L_\phi^+ \right\}$.

The next lemma is an analog of \cref{lem:isoposetlaminonattracting} for supporting systems.

\begin{lem}\label{lem:supporting_poset_isomorphism}
For $[\Lambda_1], [\Lambda_2] \in \mathcal{L}_\phi^+$, 
\begin{center}$[\Lambda_1] \subseteq [\Lambda_2]$ if and only if $\mathcal{S}
[\Lambda_1] \sqsubseteq \mathcal{S}
[\Lambda_2]$.\end{center}
Equivalently, the collection $\mathcal S(\phi)$ is naturally order-isomorphic to $\mathcal L^+_{\phi}$.
\end{lem}

\begin{proof}
Suppose $[\Lambda_1] \nsubseteq [\Lambda_2]$. 
As $\mathcal L^+_{\phi|\mathcal S[\Lambda_i]} = \Lambda_i^\downarrow$ for $i \in \{1,2\}$, we know $[\Lambda_1]$ is carried by $\mathcal{S}[\Lambda_1]$ but not by $\mathcal{S}[\Lambda_2]$. 
This shows that $\mathcal{S}[\Lambda_1] \nsqsubseteq \mathcal{S}[\Lambda_2]$.
Conversely, suppose $[\Lambda_1] \subseteq [\Lambda_2]$. Then $\Lambda_1^\downarrow \subseteq \Lambda_2^\downarrow$ and $\Lambda_2^\perp \subseteq \Lambda_1^\perp$.
As $\Lambda_2^\perp \subseteq \Lambda_1^\perp$, we have $\mathcal N(\Lambda_1^\perp) \sqsubseteq \mathcal N(\Lambda_2^\perp)$.
As $\mathcal S[\Lambda_2]$ is a free factor system of $\mathcal N(\Lambda_2^\perp)$ that carries $\Lambda_2^\downarrow \supseteq \Lambda_1^\downarrow$,
the meet $\mathcal N(\Lambda_1^\perp) \wedge \mathcal S[\Lambda_2]$ is a free factor system of $\mathcal N(\Lambda_1^\perp)$
that carries $\Lambda_1^\downarrow$ by \Cref{lem:meet_ffs_is_ffs}.
By minimality of $\mathcal S[\Lambda_1]$ in $\mathcal N(\Lambda_1^\perp)$, 
\[ \mathcal S[\Lambda_1] \sqsubseteq \mathcal N(\Lambda_1^\perp) \wedge \mathcal S[\Lambda_2] \sqsubseteq \mathcal S[\Lambda_2]. \qedhere \]
\end{proof}

Our supporting subgroup system~$\mathcal S[\Lambda]$ is closely related to the supporting free factor systems defined by Bestvina--Feighn--Handel \cite[Definition 3.2.3]{BesFeiHan00}: 
For $\Lambda \in \mathcal{L}^+(\phi)$, let $\mathcal{F}(\Lambda)$ be the meet of all free factor systems of~$\free$ that carry $\Lambda$.
By uniqueness of the meet, $\mathcal F(\phi\Lambda) = \phi \mathcal F(\Lambda)$ for $\Lambda \in \mathcal{L}^+(\phi)$.
For $[\Lambda] \in \mathcal L^+_\phi$, we define $\mathcal{F}[\Lambda]$ to be the meet of all free factor systems of~$\free$ that support $\{ \mathcal F(\Lambda') : \Lambda' \in [\Lambda] \}$.
Note that $\mathcal{S}[\Lambda] \sqsubseteq \mathcal{F}[\Lambda]$, but they need not be equal in general as the following example shows. 

\begin{example}
Let $\phi$ be induced by a homeomorphism $f$ of a surface $\Sigma$ with boundary. Suppose the Thurston normal form for $f$ cuts $\Sigma$ into four $f$-invariant subsurfaces $\Sigma_0,\Sigma_1,\Sigma_2,\Sigma_3$, as shown in \Cref{fig:surface}, with attracting laminations $\Lambda_i$ supported on $\Sigma_i$ for $i=1,2,3$. 

\begin{figure}[ht]
    \centering
    \includegraphics[width=0.3\linewidth]{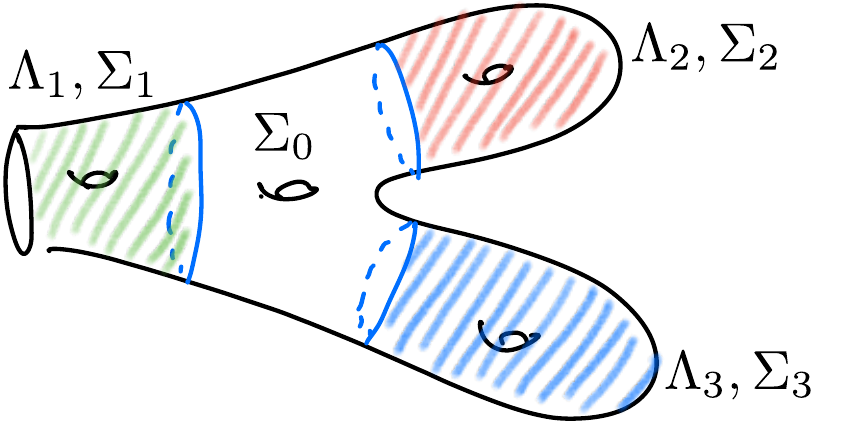}
  \caption{Illustrating support subgroup systems}
    \label{fig:surface}
\end{figure}

We obtain $\mathcal{S}[\Lambda_1]$ using the non-attracting subgroup systems to illustrate the definition of supporting subgroup systems. We have  $\Lambda_1^{\perp} = \{\Lambda_2, \Lambda_3\}$ and $\mathcal{N}[\Lambda_i] = \pi_1(\Sigma \setminus \Sigma_i)$ for $i=2,3$. Then $\mathcal{S}[\Lambda_1] = \pi_1(\Sigma_1)$ is the smallest free factor system of $\mathcal{N}[\Lambda_2] \wedge \mathcal{N}[\Lambda_3] = \pi_1(\Sigma_1 \cup \Sigma_0)$ that carries $\Lambda_1$. We see that $\mathcal{S}[\Lambda_1] = \pi_1(\Sigma_1) $ is a proper subgroup of $\mathcal{F}[\Lambda_1] = \pi_1(\Sigma)$, whereas $\mathcal{S}[\Lambda_i] = \mathcal{F}[\Lambda_i] = \pi_1(\Sigma_i)$ for $i=2,3$. 
\end{example}

Bestvina--Feighn--Handel used these free factor systems to give a natural bijection between $\mathcal L^+(\phi)$ and $\mathcal L^-(\phi) \defeq \mathcal L^+(\phi^{-1})$.
The latter is also referred to as the poset of \emph{repelling} laminations of $\phi$.
By \cite[Lemma 3.2.4]{BesFeiHan00}, for every $\Lambda_+ \in \mathcal{L}^+(\phi)$, there exists a unique $\Lambda_- \in \mathcal{L}^-(\phi)$ such that $\mathcal{F}(\Lambda_+) = \mathcal{F}(\Lambda_-)$.
We say $\Lambda_+$ and $\Lambda_-$ are \emph{paired}, and this defines the \emph{pairing bijection} between $\mathcal L^+(\phi)$ and $\mathcal L^-(\phi)$.
By uniqueness, the pairing bijection is equivariant with respect to the $\langle \phi \rangle$-action, and we get an induced pairing between $\mathcal L^+_\phi$ and $\mathcal L^-_\phi \defeq \mathcal L^+_{\phi^{-1}}$.

\begin{lem}[{see~\cite[Theorem~F]{HandelMosher20}}]\label{lem:non-attracting subgroup for repelling}
$\mathcal{N}[\Lambda_+] = \mathcal{N}[\Lambda_-]$ for paired lamination orbits $[\Lambda_{\pm}]$. 
\end{lem}

The lemma essentially is due to Handel--Mosher. We include a proof that uses splittings instead of the machinery of completely split relative train track maps.  

We start by defining relative growth. 
Suppose $\mathcal A$ is a $\phi$-invariant malnormal subgroup system of~$\free$ with finite type and $\mathcal F' \sqsubsetneq \mathcal A$ is a $\phi$-invariant proper free factor system.
Fix a free splitting~$\mathcal T$ of~$\mathcal A$ whose subgroup system of point stabilizers is~$\mathcal F'$.
The length~$\norm{[g]}_{\mathcal F'}$ of an $\mathcal A$-peripheral element of~$\free$ is the translation length of its conjugacy class in~$\mathcal T$.
An $\mathcal A$-peripheral element~$g$ has \emph{polynomial growth relative to~$\mathcal F'$} under $\phi$-iteration if the sequence $(\norm{\phi^n[g]}_{\mathcal F'})_n$ is bounded by a polynomial in~$n$.
The property is independent of the choice of the free splitting~$\mathcal T$, and we recover the original definition of polynomial growth when $\mathcal F' = \emptyset$.
The restriction~$\phi|\mathcal A$ is \emph{polynomially growing relative to~$\mathcal F'$} if every $\mathcal A$-peripheral element has polynomial growth relative to~$\mathcal F'$ under $\phi$-iteration.
For completeness, we insist~$\phi|\mathcal A$ is \emph{polynomially growing relative to~$\mathcal A$}.

\begin{proof}[Proof of \Cref{lem:non-attracting subgroup for repelling}]
    Suppose, for a contradiction, that $\mathcal{N}[\Lambda_+] \neq \mathcal{N}[\Lambda_-]$.
    Without loss of generality, we may assume some $g \in \free$ is  weakly attracted to $[\Lambda_-]$ (under backward $\phi$-iteration) but not $[\Lambda_+]$ (under forward $\phi$-iteration). 
    Then the $\phi$-invariant malnormal subgroup system $\mathcal N[\Lambda_+]$ carries~$[[g]]$.
    By $\phi$-invariance, $\mathcal N[\Lambda_+]$ carries each periodic line in the sequence $(\phi^{-k}[[g]])_{k \in \NN}$.
    Since the accumulation set of the sequence contains $[\Lambda_-]$, the closed subspace $\BB(\mathcal N[\Lambda_+]) \subseteq \BB(\free)$ contains~$[\Lambda_-]$.
    Let $\mathcal F \defeq \mathcal F[\Lambda_\pm]$ be the free factor system of~$\free$ defined above.
    Then $\mathcal N[\Lambda_+|\mathcal F]$ carries~$[\Lambda_-]$.
    However, $\mathcal N[\Lambda_+|\mathcal F] = \mathcal N[\Lambda_-|\mathcal F]$. 
    Indeed, choose a maximal $\phi$-invariant proper free factor system $\mathcal F' \sqsubsetneq \mathcal F$. By \cite[Proposition~2.2]{Mutanguha22}, $\mathcal N[\Lambda_+|\mathcal F]$ (resp. $\mathcal N[\Lambda_-|\mathcal F]$) is the unique maximum amongst malnormal subgroup systems $\mathcal H \sqsubseteq \mathcal F$ such that $\mathcal H$-peripheral elements of~$\free$ are polynomially growing relative to~$\mathcal F'$ under $\phi$-iteration (resp. under $\phi^{-1}$-iteration). 
    So $\mathcal F' \sqsubseteq \mathcal{N}[\Lambda_+|\mathcal F]$, and the restriction $\phi|\mathcal{N}[\Lambda_+|\mathcal F]$ is polynomially growing relative to $\mathcal F'$.
    By the equivalence of Conditions~($1{\Leftrightarrow}2$) in \cite[Proposition~3.1]{Mutanguhaexistence}, the restriction $\phi^{-1}|\mathcal{N}[\Lambda_+|\mathcal F]$ is polynomially growing relative to $\mathcal F'$ (see also \cite[Corollary~2.5]{JPPoly}).  
    Uniqueness and maximality of~$\mathcal N[\Lambda_-|\mathcal F]$ implies $\mathcal N[\Lambda_+|\mathcal F] \sqsubseteq \mathcal N[\Lambda_-|\mathcal F]$. 
    By symmetry, the two are equal, and $\mathcal N[\Lambda_-]$ carries $[\Lambda_-]$, which contradicts \cref{lem:approximategenline}.
\end{proof}

From this lemma and \Cref{lem:isoposetlaminonattracting}, we immediately get the following:
\begin{coro}\label{cor:isomposet pairing}
    The induced pairing between $\mathcal{L}^+_\phi$ and $\mathcal{L}^-_{\phi}$ is an order isomorphism. \qed
\end{coro}

Finally, we show that paired lamination orbits have the same supporting subgroup systems.

\begin{coro}\label{cor:suppporting subgroup for repelling}
$\mathcal{S}[\Lambda_+] = \mathcal{S}[\Lambda_-]$ for paired lamination orbits $[\Lambda_{\pm}]$. 
\end{coro}
\begin{proof}
    By \Cref{lem:non-attracting subgroup for repelling,cor:isomposet pairing}, we have 
    $\mathcal N \defeq \mathcal N(\Lambda_{+}^{\perp}) = \mathcal N(\Lambda_{-}^{\perp})$.
    As $\mathcal L^\pm_{\phi|\mathcal N} = \Lambda_{\pm}^{\downarrow}$ (\Cref{lem:nonattractingcarry}), the supporting subgroup system $\mathcal S[\Lambda_+]$ (resp.~$\mathcal S[\Lambda_-])$ is the minimal free factor system of $\mathcal N$ that carries $\mathcal L^+_{\phi|\mathcal N}$ (resp.~$\mathcal L^-_{\phi|\mathcal N}$).
    This is equivalent
    to stating that~$\phi|\mathcal N$ is polynomially growing relative to $\mathcal S[\Lambda_+]$, and so is $\phi^{-1}|\mathcal N$ as in the previous proof; 
    therefore, $\mathcal S[\Lambda_+]$ carries $\mathcal L^-_{\phi|\mathcal N}$.
    By minimality, $\mathcal S[\Lambda_-] \sqsubseteq \mathcal S[\Lambda_+]$. 
    Similarly, $\mathcal S[\Lambda_+] \sqsubseteq \mathcal S[\Lambda_-]$.
\end{proof}

%%%%%%%%%%%%%%%%%%%%%%%%%%%%
\section{Invariant malnormal subgroup systems of \texorpdfstring{$\free$}{the free group}}\label{sec:invariantmalnormal_ss}

Let $\phi \in \Out(\free)$.
In this section, we relate $\phi$-invariant proper malnormal subgroup systems supporting the polynomial subgroup system~$\mathcal P(\phi)$ with nonattracting subgroup systems of~$\phi$. 
Suppose $\mathcal{A}$ is a $\phi$-invariant malnormal subgroup system of $\free$ with finite type.  
We start with the case of $\phi$-invariant proper free factor systems of~$\mathcal A$.

\begin{lem}\label{Lem:freefactorsystemnonattracting}

Suppose that $\phi|\mathcal{A}$ does not preserve a free splitting of $\mathcal{A}$.
If $\mathcal{F} \sqsubsetneq \mathcal A$ is a $\phi$-invariant proper free factor system, then $\mathcal{F}\sqsubseteq \mathcal{N}[\Lambda|\mathcal A]$ for some $[\Lambda] \in \mathcal{L}^+_{\phi|\mathcal{A}}$. 
\end{lem}

\begin{proof}
We may suppose that $\mathcal{F} \sqsubsetneq \mathcal A$ is a maximal $\phi$-invariant proper free factor system. 
Equivalently, for every $[A] \in \mathcal{A}$, the outer automorphism $\phi^k|A$ is irreducible relative to~$\mathcal F | A$.
Choose $[A] \in \mathcal{A}$ such that $[A] \notin \mathcal{F}$.
Let $f \colon \Gamma \to \Gamma$ be a relative train track map for $\phi^k|A$ equipped with an $f$-invariant subgraph that induces the proper free factor system $\mathcal{F}|A$ of $A$ (see \Cref{app:A}).
We claim that the top stratum~$H_0$ of~$f$ is an EG stratum. Indeed, suppose towards a contradiction that $H_0$ is NEG.
Then collapsing each connected component of $\Gamma_1 = \overline{\Gamma\setminus H_0}$ to a point induces a $\phi^k|A$-invariant graph of groups $\mathbb{X}$ of $A$ with trivial edge groups. By Bass--Serre theory, $\mathbb{X}$ is covered by a $\phi^k|A$-invariant free splitting of~$A$. Since $\phi|\mathcal{A}$ does not preserve a free splitting of~$\mathcal{A}$, 
we have a contradiction. Hence, $H_0$ is an EG stratum.

By~\cite[Lemma~3.1.9]{BesFeiHan00}, some $\langle \phi^k|A \rangle$-orbit of attracting laminations $[\Lambda] \in \mathcal{L}^+_{\phi^k|A}$ is not carried by $\Gamma_{1}$. The attracting lamination $\Lambda$ for $\phi^k|A$ is also an attracting lamination for $\phi|\mathcal{A}$. 
We abuse notation and denote its $\langle \phi\rangle$-orbit by $[\Lambda] \in \mathcal L^+_{\phi|\mathcal A}$ as well.
By construction, no $\mathcal{F}$-peripheral conjugacy class $[g]$ is weakly attracted to $[\Lambda]$, i.e.~$\mathcal{F}\sqsubseteq \mathcal{N}[\Lambda]$;
thus, $\mathcal F \sqsubseteq \mathcal N[\Lambda|\mathcal A]$.
\end{proof}

Let $\mathcal{N} \sqsubsetneq \mathcal A$ be a proper malnormal subgroup system with finite type and $\mathcal{N}=\mathcal{P}_\mathcal{N} \sqcup \mathcal{H}_\mathcal{N}$ be the partition into conjugacy classes of noncyclic and cyclic subgroups respectively. Note that $\mathcal{P}_\mathcal{N}$ is still a proper malnormal subgroup system with finite type and, for every $[A] \in \mathcal{A}$, the group $A$ is hyperbolic relative to $\mathcal{P}_\mathcal{N} | A$;
therefore, by \cref{Thm:JSJ}, the JSJ forest $\mathcal T_{\mathcal N}$ for $(\mathcal{A}, \mathcal N)$ exists if $\mathcal{A}$ is one-ended relative to $\mathcal{N}$.

\begin{lem}\label{lem:norigidvertex}

Suppose $\mathcal{N} \sqsubsetneq \mathcal A$ is a $\phi$-invariant proper malnormal subgroup system with finite type such that $\mathcal{P}(\phi|\mathcal A) \sqsubseteq \mathcal{N}$. 
If $\mathcal{A}$ is one-ended relative to $\mathcal{N}$, then the JSJ forest $\mathcal T_\mathcal{N}$ has no rigid vertices.
\end{lem}

\begin{proof}
Suppose towards a contradiction that there exists $[A] \in \mathcal{A}$ such that $T_{\mathcal N}^A$ has a rigid vertex $v$. 
Pick an automorphism $\Psi \in \phi^k|A$.
Then there is a $\Psi$-equivariant isometry~$\iota$ of $T_{\mathcal N}^A$ by \cref{Thm:JSJ}(2); moreover, $\iota^m$ preserves the orbits of vertices of $T^A_{\mathcal N}$ for some $m \ge 1$. Thus, we have a natural homomorphism $\langle (\phi^k|A)^m \rangle \to \Out(\Stab(v))$ whose image is finite by \cref{Thm:JSJ}(3b). Therefore, for some $g \in \free$ and $n \ge 1$, $\Psi^{mn}(\Stab(v))=g^{-1}\Stab(v)g$ and $(\iota_g \circ \Psi^{mn})|\Stab(v) \in \Inn(\Stab(v))$, where $\iota_g \colon \free \to \free$ is left-conjugation by $g$. This shows that $\Stab(v) \le P$ for some  $[P] \in \mathcal{P}(\phi|\mathcal A) \sqsubseteq  \mathcal{N}$. Hence, $\Stab(v)$ is elementary. This contradicts the fact that $\Stab(v)$ is nonelementary for a rigid vertex.
\end{proof}

We now prove the main proposition of this section. 

\begin{prop}\label{Prop:invariantmalnormal}
Let $\phi \in \Out(\free)$ and $\mathcal{A}$ be a $\phi$-invariant malnormal subgroup system of $\free$ with finite type.  
Suppose that $\phi|\mathcal{A}$ does not preserve a free splitting of $\mathcal{A}$.
If $\mathcal{N} \sqsubsetneq \mathcal A$ is a $\phi$-invariant proper malnormal subgroup system with finite type such that $\mathcal{P}(\phi|\mathcal A) \sqsubseteq \mathcal{N}$, then $\mathcal{N} \sqsubseteq  \mathcal{N}[\Lambda|\mathcal A]$ for some $[\Lambda] \in \mathcal{L}^+_{\phi|\mathcal{A}}$. 
\end{prop}

\begin{proof}
Let $\mathcal{N} \sqsubsetneq \mathcal A$ be a $\phi$-invariant proper malnormal subgroup system with finite type such that $\mathcal{P}(\phi |\mathcal A) \sqsubseteq  \mathcal{N}$. 
By \cref{lem:transitivityffs,lem:meet_ffs,lem:meet_ffs_is_ffs}, the meet $\mathcal{F}(\mathcal{N})$ of all free factor systems of $\mathcal{A}$ that support $\mathcal{N}$ is a $\phi$-invariant
free factor system of~$\mathcal A$. 
If $\mathcal F(\mathcal N) \neq \mathcal A$, then the result follows from \cref{Lem:freefactorsystemnonattracting}.
Otherwise, $\mathcal{F}(\mathcal{N}) = \mathcal{A}$, i.e.~$\mathcal{A}$ is one-ended relative to~$\mathcal{N}$.

Let $\mathcal T_{\mathcal N}$ be the JSJ forest for $(\mathcal{A}, \mathcal N)$ (\cref{Thm:JSJ}).
By \cref{Thm:JSJ}$(1)$, there exists $[A] \in \mathcal{A}$ such that the tree $T_{\mathcal N}^A$ contains a nonelementary vertex $v$. 
By \cref{Thm:JSJ}$(3)$ and \cref{lem:norigidvertex}, the vertex $v$ is flexible. 
For some $m \ge 1$, the outer automorphism $(\phi^k|A)^m$ preserves the orbit $[v]$ of $v$, and
we have a well-defined natural homomorphism $\langle (\phi^k|A)^m \rangle \to \Out(\Stab(v))$ given by restriction. 
Denote by $\Sigma_v$ the compact surface associated with $v$.
By \cref{Thm:JSJ}$(3c)$, the iterate $(\phi^k|A)^m$ restricts to a mapping class $f_v$ of $\Sigma_v$.

We claim that $f_v$ is pseudo-Anosov. Indeed, by the definition of a flexible vertex, for every $[H] \in \mathcal{N}$, the intersection $H \cap \pi_1(\Sigma_v)$ is $\pi_1(\partial \Sigma_v)$-peripheral. As $\mathcal{P}(\phi|\mathcal A) \sqsubseteq  \mathcal{N}$, this implies that any $f_v$-periodic conjugacy class in $\pi_1(\Sigma_v)$ is $\pi_1(\partial\Sigma_v)$-peripheral. Thus, the mapping class $f_v$ is pseudo-Anosov by the Nielsen--Thurston classification~\cite[Theorem~4]{Thurston88}.

Let $\Lambda_v$ be the attracting singular foliation of~$\Sigma_v$ associated with $f_v$. Then $\Lambda_{v}$ determines an attracting lamination for $\phi|\mathcal{A}$, also denoted $\Lambda_v$. On one hand, if $g \in \free$ is $\mathcal{A}$-peripheral and weakly attracted to the orbit $[\Lambda_v] \in \mathcal L^+_{\phi|\mathcal A}$, then $g \in A$ up to conjugation and $\phi$-iteration. Moreover, if such an element~$g$ is an elliptic isometry of $T_{\mathcal N}^A$, then, up to an additional conjugation, 
$g$ fixes $v$ and is not $\pi_1(\partial\Sigma_v)$-peripheral.
On the other hand, every $\mathcal N |A$-peripheral element is elliptic in~$T_{\mathcal N}^A$, and $H \cap \pi_1(\Sigma)$ is $\pi_1(\partial\Sigma_v)$-peripheral for every $[H] \in \mathcal{N}$. Thus, no $\mathcal{N}$-peripheral element is weakly attracted to $[\Lambda_v]$, i.e.~$\mathcal{N} \sqsubseteq  \mathcal{N}[\Lambda_v|\mathcal A]$.
\end{proof}

\section{Group invariance of lamination depth}\label{sec:Group invariance}

Fix an outer automorphism $\phi \in \Out(\free)$ and its mapping torus $\GG \defeq M(\phi) = \free \rtimes_\phi \ZZ$.
We say an object/property associated with $\phi$ is \emph{canonical} in $\GG$ if it has a characterization involving only algebraic statements about $\GG$.
In particular, the algebraic statements do not reference the presentation $M(\phi)$.

For the rest of the section, $\mathcal A$ is a $\phi$-invariant malnormal subgroup system of $\free$ with finite type and $\mathcal G \defeq M(\phi|\mathcal A)$ is the associated mapping torus. 
As before, we say an object/property associated with $\phi|\mathcal A$ is \emph{canonical} in $\mathcal G$ if it has a characterization involving only algebraic statements about $\mathcal G$ and $\GG$.
Note that canonicity is transitive: if an object is canonical in~$\mathcal G$ and $\mathcal G$ is canonical in~$\GG$, then the object is canonical in~$\GG$.

\subsection{Growth type is canonical}\label{subsec:growthtypecanon}

For our first example of a property of $\phi$ that is canonical in $\GG$, we remark that $\phi$ preserves a free splitting of $\free$ if and only if $\GG$ has a cyclic splitting \cite[Lemma~3.1]{JPPoly}. 

More generally, the growth type of $\phi$ is canonical in $\GG$, and the algebraic characterization involves inductively choosing cyclic splittings.
Set $\mathcal H_0 = \{ [\GG] \}$ and, for $i \ge 0$, assume $\mathcal H_i$ is a free-by-cyclic subgroup system of $\GG$ with finite type.
Pick a cyclic splitting $\mathcal S_i$ of $\mathcal H_i$ (if one exists).
By Bass--Serre theory, the conjugacy classes of noncyclic vertex stabilizers of $\mathcal S_i$ form a proper free-by-cyclic subgroup system $\mathcal H_{i+1} \sqsubsetneq \mathcal H_i$.
This inductively produces a filtration $\mathcal H_0 \sqsupsetneq \mathcal H_1 \sqsupsetneq \cdots$,
and we stop when every vertex stabilizer of $\mathcal S_i$ is cyclic or $\mathcal H_i$ has no cyclic splittings.
We will say $\GG$ is \emph{polynomial} if this process stops when every vertex stabilizer is cyclic (for some/any choices of cyclic splittings).
The growth type of $\phi$ is canonical in $\GG$: $\phi$ is polynomially growing if and only if $\GG$ is polynomial \cite[Theorem~3.2]{JPPoly}.

Finally, the mapping torus $\mathcal P(M(\phi))$ (defined in \S\ref{sec:polynomial_growth}) is canonical in $\GG$ since it is the unique subgroup system of maximal polynomial free-by-cyclic subgroups of $\GG$ \cite[Proposition~3.6]{JPPoly}.
We may now denote it as $\mathcal P(\GG)$. 
Set $\mathcal P(\mathcal G) \defeq \mathcal G \wedge \mathcal P(\GG)$.
As the meet of $\mathcal G$ with a canonical subgroup system in~$\GG$, $\mathcal P(\mathcal G)$ is canonical in $\mathcal G$.

\subsection{Maximal proper malnormal mapping tori are canonical}
\label{subsec:maximalpropermalnormal}

For any $[\Lambda] \in \mathcal L_{\phi|\mathcal A}^+$, recall that $\mathcal N[\Lambda|\mathcal A] \sqsubseteq \mathcal A$ is a $\phi$-invariant proper malnormal subgroup system with finite type;
moreover, it supports $\mathcal P(\phi|\mathcal A)$ and has no $\phi|\mathcal A$-twins by \cref{cor:NSS}.
The mapping torus $M(\phi|\mathcal N[\Lambda|\mathcal A])$, now abbreviated $\mathcal M[\Lambda|\mathcal A]$, is a proper malnormal subgroup system in~$\mathcal G$ supporting $\mathcal P(\mathcal G)$ by \cref{lem:malnormalnotwin,lem:inclusionsubsys,lem:meet_of_mapping_tori}.

For this section, assume $\mathcal G$ has no cyclic splittings, and set $\mathcal L_{\phi|\mathcal A}^{\mathrm{max}} \defeq \left\{\text{maximal } [\Lambda] \in \mathcal L^+_{\phi|\mathcal A} \right\}$. 
We now give an algebraic characterization of the collection $\mathcal M^{\mathrm{max}}_{\mathcal G}$ of associated mapping tori $\mathcal M[\Lambda|\mathcal A]$ for $[\Lambda] \in \mathcal L_{\phi|\mathcal A}^{\mathrm{max}}$. 
This is the key observation underlying the rest of the paper:

\begin{theo}\label{thm:maximallamcanonical}
Let $\phi \in \Out(\free)$ and $\mathcal{A}$ be a $\phi$-invariant malnormal subgroup system of $\free$ with finite type. 
If $\mathcal G \defeq M(\phi|\mathcal{A})$ has no cyclic splittings, then the collection 
\[ \mathcal M^{\mathrm{max}}_{\mathcal G} \defeq \left\{\mathcal M[\Lambda|\mathcal A]\,:\,[\Lambda] \in \mathcal L_{\phi|\mathcal{A}}^{\mathrm{max}}\right\} \] is the set of maximal proper malnormal free-by-cyclic subgroup systems in $\mathcal G$ supporting $\mathcal P(\mathcal G)$. 
Moreover, the natural function $\mathcal L_{\phi|\mathcal{A}}^{\mathrm{max}} \to \mathcal M^{\mathrm{max}}_{\mathcal G}$ given by $[\Lambda] \mapsto \mathcal M[\Lambda|\mathcal A]$ is a bijection.
\end{theo}

Thus, the collection $\mathcal M^{\mathrm{max}}_{\mathcal G}$ 
is canonical in $\mathcal G$, and we may also say $\mathcal L_{\phi|\mathcal{A}}^\mathrm{max}$ is canonical in~$\mathcal G$.

\begin{proof}
Let $\mathcal H$ be a proper malnormal free-by-cyclic subgroup system in $\mathcal G$ supporting~$\mathcal P(\mathcal G)$.
Since $\mathcal P(\mathcal G)$ has finite type, some finite subsystem of $\mathcal H$ supports it.
For an arbitrary finite subsystem $\mathcal H' \subseteq \mathcal H$ that supports $\mathcal P(\mathcal G)$, \cref{lem:FbC_SS_MappingTorus} implies $\mathcal H' = M(\phi|\mathcal N')$ for some $\phi$-invariant proper malnormal subgroup system $\mathcal N' \sqsubsetneq \mathcal{A}$ with finite type. 
By \cref{lem:inclusionsubsys}, the system $\mathcal N'$ supports $\mathcal P(\phi|\mathcal{A})$ since $\mathcal P(\mathcal G) \sqsubseteq \mathcal H'$.
By \cref{Prop:invariantmalnormal}, $\mathcal N' \sqsubseteq \mathcal N[\Lambda|\mathcal A]$ for some $[\Lambda] \in \mathcal L^+_{\phi|\mathcal{A}}$; therefore, $\mathcal H' \sqsubseteq \mathcal M[\Lambda|\mathcal A]$. 
As $\mathcal H' \subseteq \mathcal H$ was arbitrary, we get $\mathcal H \sqsubseteq \mathcal M[\Lambda|\mathcal A]$. 
We may assume $[\Lambda]$ is in $\mathcal L_{\phi|\mathcal{A}}^{\mathrm{max}}$ by \cref{lem:inclusionsubsys,lem:isoposetlaminonattracting}.

Thus, the collection of maximal proper malnormal free-by-cyclic subgroup systems of $\mathcal G$ supporting $\mathcal P(\mathcal G)$ is a subset of $\mathcal M^{\mathrm{max}}_{\mathcal G}$.
This subset is not proper. 
Indeed, for $[\Lambda'] \in \mathcal L_{\phi|\mathcal{A}}^{\mathrm{max}}$, we show $\mathcal M[\Lambda'|\mathcal A]$ is a maximal proper malnormal free-by-cyclic subgroup system of~$\mathcal{G}$ supporting~$\mathcal P(\mathcal G)$.
Suppose $\mathcal M[\Lambda'|\mathcal A] \sqsubseteq \mathcal H$ for some proper malnormal free-by-cyclic subgroup system~$\mathcal H$ of~$\mathcal G$ supporting~$\mathcal P(\mathcal G)$.
By the previous paragraph, $\mathcal M[\Lambda'|\mathcal A] \sqsubseteq \mathcal H \sqsubseteq \mathcal M[\Lambda|\mathcal A]$ for some $[\Lambda] \in \mathcal L_{\phi|\mathcal{A}}^{\mathrm{max}}$.
Then $[\Lambda'] = [\Lambda]$ as $\big\{\mathcal M[\Lambda|\mathcal A]\,:\,[\Lambda] \in \mathcal L^+_{\phi|\mathcal{A}}\big\}$ is naturally order-isomorphic to $\mathcal L^+_{\phi|\mathcal{A}}$ by \cref{lem:inclusionsubsys,lem:isoposetlaminonattracting}.
So $\mathcal M[\Lambda'|\mathcal A] = \mathcal H$ as needed. 

The last statement in the theorem follows from the previous order-isomorphism.
\end{proof}

This leads to an algebraic characterization of $M(\phi|\mathcal N(\mathcal L_{\phi|\mathcal A}^{\mathrm{max}}|\mathcal A))$, abbreviated $\mathcal M(\mathcal L_{\phi|\mathcal A}^{\mathrm{max}}|\mathcal A)$.

\begin{coro}\label{Cor:canonintersection}
Let $\phi \in \Out(\free)$ and $\mathcal{A}$ be a $\phi$-invariant malnormal subgroup system of $\free$ with finite type. 
If $\mathcal G \defeq M(\phi|\mathcal{A})$ has no cyclic splittings, then the subgroup system $\mathcal M(\mathcal L_{\phi|\mathcal{A}}^{\mathrm{max}}|\mathcal A)$ is the meet of all maximal proper malnormal free-by-cyclic subgroup systems in~$\mathcal G$ supporting $\mathcal P(\mathcal G)$. 
\end{coro}

So $\mathcal M(\mathcal L_{\phi|\mathcal{A}}^{\mathrm{max}}|\mathcal A)$ is canonical in $\mathcal G$.

\begin{proof} 
Recall that $\mathcal N[\mathcal L_{\phi|\mathcal{A}}^{\mathrm{max}}] = \bigwedge\big\{ \mathcal N[\Lambda] : [\Lambda] \in \mathcal L_{\phi|\mathcal{A}}^{\mathrm{max}}\big\}$ by the discussion following \cref{Thm:NSS}.
Thus the mapping torus $\mathcal M(\phi|\mathcal N(\mathcal L_{\phi|\mathcal{A}}^{\mathrm{max}})) = \bigwedge\{ \mathcal M(\phi|\mathcal N[\Lambda]) : [\Lambda] \in \mathcal L_{\phi|\mathcal{A}}^{\mathrm{max}}\}$ by \cref{lem:meet_of_mapping_tori}.
After taking the meet with $\mathcal G = \mathcal M(\phi|\mathcal A)$ on both sides and applying the lemma again,
\[\mathcal M(\mathcal L_{\phi|\mathcal{A}}^{\mathrm{max}}|\mathcal A) = \bigwedge \left\{\mathcal{M}[\Lambda|\mathcal{A}] : [\Lambda] \in \mathcal L_{\phi|\mathcal{A}}^{\mathrm{max}}\right\} = \bigwedge \mathcal M^{\mathrm{max}}_{\mathcal G}. \]
The conclusion follows from the characterization of $\mathcal M^{\mathrm{max}}_{\mathcal G}$ in \cref{thm:maximallamcanonical}.
\end{proof}

As illustrated in the following example, \cref{Cor:canonintersection} can be sufficient to construct a {canonical} filtration of $\GG$ by free-by-cyclic subgroup systems whose length is the depth $\delta(\phi)$.

\begin{example}\label{ex:Filtration_F}
Recall from \cref{ex:lamination_poset} the automorphism~$\Phi$ of $\free= \langle a, b, c, d, e, f\rangle$ given by
\[a \mapsto ab,~ b \mapsto bab,~ c \mapsto bcd,~ d \mapsto dcd,~ e \mapsto aba^{-1}b^{-1}ef,~ f \mapsto fef. \]  

Let $\mathcal{L}^0 \defeq \big\{[\Lambda_{cd}], [\Lambda_{ef}]\big\}$ and $\mathcal{L}^1 \defeq \big\{[\Lambda_{ab}]\big\}$ be the `levels' of the poset of laminations $\mathcal{L}_{\phi}^+$, also indicated in \Cref{fig:Filtration_F}. The nonattracting subgroup systems $\mathcal{N}(\mathcal{L}^i)$ were given in \cref{ex:non-attracting_subgroup_system}, which we recall here
\begin{align*}
    \mathcal{N}(\mathcal{L}^0) & = \mathcal{N}[\Lambda_{cd}] \wedge \mathcal{N}[\Lambda_{ef}] = [\langle a, b \rangle] \\
    \mathcal{N}(\mathcal{L}^1) & = \mathcal{N}[\Lambda_{ab}] =  [\langle aba^{-1}b^{-1}, e, f \rangle].
\end{align*}
Set $\mathcal{F}_0 \defeq [\free]$, and inductively define $\mathcal{F}_{i+1} \defeq \mathcal{N}(\mathcal{L}^i) \wedge \mathcal{F}_i$ for $0 \leq i < \delta(\phi)$. We get the filtration $\mathcal{F}_0 \sqsupsetneq \mathcal{F}_1 \sqsupsetneq \mathcal{F}_2 $,   which is equal to
\[ [\free] \sqsupsetneq [\langle a, b \rangle] \sqsupsetneq [\langle aba^{-1}b^{-1} \rangle].\]

\begin{figure}
    \centering

\begin{tikzpicture}[scale=0.8][baseline=(A.center)]
  \tikzset{BarStyle/.style =   {opacity=.4,line width=4 mm,line cap=round,color=#1}}

\matrix (A) [matrix of math nodes,column sep=3 mm] 
{ \ & \ & \  & \ & \ & \ & \ & \ &\mathcal{F}_0=[\free] & \  \\
    {}[\Lambda_{cd}] & {}[\Lambda_{ef}]  & \ & \ & \mathcal{L}^0 & \ & \ & \ & \mathcal{F}_1=\mathcal{N}(\mathcal{L}^0)\wedge \mathcal{F}_0  \\
   \ &\ & \ & \ &\ & \ & \ & \ & \ & \\
   {}[\Lambda_{ab}] & \  &\ & \ & \mathcal{L}^1 & \ & \ & \ & \mathcal{F}_2= \mathcal{N}(\mathcal{L}^1)\wedge \mathcal{F}_1 \\
    \ & \ & \ &\ & \ & \ & \ & \ & \  \\
};

\draw[thick] (A-2-1) to (A-4-1);

\end{tikzpicture}

    \caption{Filtration for \cref{ex:Filtration_F}}
    \label{fig:Filtration_F}
\end{figure}

One can check that $\mathcal{F}_0 = [\free]$ has no $\phi$-invariant proper free factor system that carries $\mathcal{L}_{\phi}^+$. Equivalently, there is no $\phi$-invariant free splitting of $\free$. 
Thus, $\mathcal{G}_0\defeq \mathcal{M}(\phi|\mathcal{F}_0) = [\GG]$ has no cyclic splitting. Since $\mathcal{L}_{\phi|\mathcal{F}_0}^{\mathrm{max}} = \mathcal{L}^0$,  \cref{Cor:canonintersection} implies that $\mathcal{G}_1\defeq \mathcal{M}(\phi|\mathcal{F}_1)$ is canonical in~$\mathcal{G}_0$.

Observe that $\mathcal{F}_1$ has no $\phi|\mathcal{F}_1$-invariant proper free factor system that carries $\mathcal{L}_{\phi|\mathcal{F}_1}^+ = \{[\Lambda_{ab}]\}$. Indeed, a cyclic free factor cannot support an attracting lamination. So $\mathcal{G}_1$ has no cyclic splitting. Since $\mathcal{L}_{\phi|\mathcal{F}_1}^{\mathrm{max}} = \mathcal{L}^1$, the subgroup system $\mathcal{G}_2\defeq \mathcal{M}(\phi|\mathcal{F}_2)$ is canonical in $\mathcal{G}_1$ by \cref{Cor:canonintersection} again. 
We stop the filtration at the polynomial system $\mathcal{G}_2$. Thus, we get a canonical filtration  
\[ [\GG] = \mathcal{G}_0 \sqsupsetneq \mathcal{G}_1 \sqsupsetneq \mathcal{G}_2 \]
whose length is $2=\delta(\phi)$.
In other words, $\delta(\phi)$ is the number of canonical applications of \cref{Cor:canonintersection} needed to reduce $\GG$ to a polynomial free-by-cyclic subgroup system.
\end{example}

%%%%%%%%%%%%%%%%%%%%%%%%%%%%%%%%%%%%%%%%%%%%%%%%%%%%%%%%%%%%%%%%%%%%%%%%%%%%%%%%
\subsection{Lamination depth is canonical}\label{subsec:depthcanonical}

In \cref{thm:maximallamcanonical}, we assume that $\mathcal G$ has no cyclic splittings. In this section, we first show that we can canonically reduce to that case. Then we create canonical filtrations of~$\GG$ by free-by-cyclic subgroup systems whose lengths are the lamination depth~$\delta(\phi)$.

Let $\mathcal{B}$ be the meet of all free factor systems of $\mathcal A$ that carry $\mathcal{L}^+_{\phi|\mathcal A}$, which is the free factor system of $\mathcal A$ defined using \cref{lem:transitivityffs,lem:meet_ffs,lem:meet_ffs_is_ffs}.
Then $\mathcal{B}$ is the minimal free factor system of~$\mathcal A$ that carries $\mathcal{L}^+_{\phi|\mathcal A}$, and by uniqueness, it is $\phi$-invariant.
Since $\mathcal A$ is a malnormal subgroup system of $\free$ with finite type and $\mathcal B$ is a free factor system of $\mathcal A$, we get $\mathcal B$ is a malnormal subgroup system of $\free$ with finite type. 
Note that $M(\phi|\mathcal B)$ is empty if and only if $\phi|\mathcal A$ is polynomially growing.
The next proposition gives an algebraic characterization of the mapping torus $M(\phi|\mathcal{B})$ in~$\mathcal G$.

\begin{prop}\label{prop:passing to no cyclic splittings}
Let $\phi \in \Out(\free)$, 
$\mathcal{A}$ be a $\phi$-invariant malnormal subgroup system of $\free$ with finite type, and $\mathcal G \defeq M(\phi|\mathcal A)$.
If $\mathcal B$ is the $\phi$-invariant minimal free factor system  of $\mathcal A$ that carries $\mathcal L_{\phi|\mathcal A}^+$, then $M(\phi|\mathcal B)$ has no cyclic splittings.
Moreover, for any $[G] \in \mathcal{G}$ and any subgroup $H \leq G$ with no cyclic splitting, there is some $[G'] \in M(\phi|\mathcal B)$ with $H \leq G'$. 
\end{prop}

In particular, $M(\phi|\mathcal B)$ is the maximal free-by-cyclic subgroup system in $\mathcal G$ with no cyclic splittings. So $M(\phi|\mathcal B)$ is canonical in~$\mathcal G$.

\begin{proof}
For any $[A] \in \mathcal{A}$, we denote the first return outer class $\phi^k|A \in \phi|\mathcal A$ by $\psi \in \Out(A)$ (recall \cref{def:restrictions}).
Let $f \colon \Gamma_A \to \Gamma_A$ be a relative train track map for $\psi$ equipped with an $f$-invariant subgraph  $\Gamma_{\mathcal{B}|A}$ that induces the free factor system $\mathcal{B}|A$ of $A$ (see \Cref{app:A}).
Since $\mathcal{B}|A$ is the minimal free factor system of $A$ that carries $\mathcal{L}_{\psi}^+$, every EG stratum of~$f$ is in $\Gamma_{\mathcal{B}|A}$.

Suppose $M(\phi|\mathcal{B})$ has a cyclic splitting. This is equivalent to the following: 
$M(\psi| (\mathcal{B}|A))$ has a cyclic splitting for some $[A] \in \mathcal{A}$. 
This induces a $\psi$-invariant free splitting of $\mathcal{B}|A$. So there is a
proper free factor system $\mathcal N \sqsubsetneq \mathcal{B}|A $ which carries $\mathcal{L}_{\psi}^+$ 
(see, for instance, \cite[Lemma~4.4]{HandelMosher19}). 
We extend~$\mathcal N$ to a free factor system $\mathcal N' \sqsubsetneq \mathcal B$ of $\mathcal A$ that carries~$\mathcal L^+_{\phi|\mathcal A}$. This contradicts the minimality of~$\mathcal{B}$ in~$\mathcal{A}$. 
Thus, $M(\phi|\mathcal{B})$ has no cyclic splitting. 

For any $[A] \in \mathcal A$ and corresponding $G = M(\psi)$, let $H \le G $ be a subgroup with no cyclic splittings.  
The strata between $\Gamma_A$ and $\Gamma_{\mathcal{B}|A}$ determine a hierarchy of cyclic splittings whose vertex group systems form a filtration of~$G$ ending in $M(\psi| (\mathcal{B}|A))$, a subset of $M(\phi|\mathcal{B})$. Since $H$ has no cyclic splitting, the singleton $\{[H]\}$ is supported by the vertex group systems of each of the cyclic splittings in the hierarchy; therefore, $H \le G'$ for some $[G'] \in M(\phi|\mathcal{B})$.
\end{proof}

\begin{rmq}
In \cref{ex:Filtration_F}, each $\mathcal{F}_{i}$ was the minimal free factor system that carried $\mathcal{L}_{\phi|\mathcal{F}_{i}}^+$. This is not true in general. So, in the definitions that follow, we will pass to the minimal free factor system $\mathcal{F}_{i}'$ of $\mathcal{F}_{i}$ that carries its lamination. This, in turn, will allow us to conclude that $\mathcal{G}_{i}'$ has no cyclic splitting and $\mathcal{G}_{i+1}$ is canonical in $\mathcal{G}_{i}'$. 
\end{rmq}

Now, we describe two filtrations in $\free$. 
Set $\delta \defeq \delta(\phi)$, 
\[\mathcal L^{i} \defeq \left\{\text{maximal } [\Lambda] \in \mathcal L^+_\phi \setminus \cup_{j=0}^{i-1} \mathcal L^j \right\}\] 
for $0 \le i < \delta$, $\mathcal F_0 \defeq \{ [\free] \}$, and $\phi_0 \defeq \phi$.
So $\mathcal L^+_\phi = \bigcup_{i=0}^{\delta-1} \mathcal L^i$.
For some $i \ge 0$, assume $\mathcal F_i$ is a $\phi$-invariant malnormal subgroup system of $\free$ with finite type and $\mathcal L_{\phi|\mathcal F_{i}}^+ = \mathcal L^+_\phi \setminus \cup_{j=0}^{i-1} \mathcal L^j$. 
Let $\mathcal F_i'$ be the minimal free factor system of $\mathcal F_i$ that carries $\mathcal L_{\phi|\mathcal F_i}^+$;
therefore, $\mathcal L_{\phi|\mathcal F_i'}^{\mathrm{max}} = \mathcal L^i$.

If $i < \delta$, then set $\mathcal F_{i+1} \defeq \mathcal N(\mathcal L^i|\mathcal F_i')$.
Thus, $\mathcal F_{i+1} \sqsubsetneq \mathcal F_i'$ is a $\phi$-invariant proper malnormal subgroup system with finite type by \cref{cor:NSS}.
Note that $\mathcal L_{\phi|\mathcal F_{i+1}}^+ = \mathcal L^+_\phi \setminus \cup_{j=0}^{i} \mathcal L^j$ by \cref{lem:nonattractingcarry}.
Inductively, we get two $\phi$-invariant filtrations $(\mathcal F_i)_{i=0}^\delta$, $(\mathcal F_i')_{i=0}^\delta$ in $\free$ such that $\mathcal F_0 = \{[\free]\}, \mathcal F_\delta' = \emptyset$, and $\mathcal F_i \sqsupseteq \mathcal F_i' \sqsupsetneq \mathcal F_{i+1}$ for $0 \le i < \delta$.

\begin{defi}\label{def:canonicalfiltration}
    In~$\GG$, define the mapping tori $\mathcal G_i \defeq M(\phi|\mathcal F_i)$ and $\mathcal G_i' \defeq M(\phi|\mathcal F_i')$ for $0 \le i \le \delta$.
\end{defi}

By \cref{prop:passing to no cyclic splittings}, $\mathcal G_i' = M(\phi|\mathcal F_i')$ is canonical in $\mathcal G_i$ for $0 \le i \le \delta$.
In particular, each $\mathcal G_i'$ has no cyclic splittings.
By \cref{Cor:canonintersection}, $\mathcal G_{i+1} = \mathcal M(\mathcal L_{\phi|\mathcal F_i'}^{\mathrm{max}}|\mathcal F_i')$ is canonical in $\mathcal G_i'$ for $0 \le i < \delta$.
So we defined two canonical filtrations $(\mathcal G_i)_{i=0}^\delta$, $(\mathcal G_i')_{i=0}^\delta$ in $\GG$ such that
\[ \{[\GG]\} = \mathcal G_0 \sqsupseteq \mathcal G_0' \sqsupsetneq \cdots \sqsupsetneq \mathcal G_{\delta} \sqsupseteq \mathcal G_{\delta}' = \emptyset.\]
Their length $\delta = \delta(\phi)$ is canonical in $\GG$: it is the number of  canonical applications of both \cref{prop:passing to no cyclic splittings} and \cref{Cor:canonintersection} needed to reduce $\GG$ to $\emptyset$.
\begin{theo}\label{Thm:depthinvariant}
For $\phi \in \Out(\free)$, the depth $\delta(\phi)$ is a group invariant of $\GG \defeq \free \rtimes_\phi \ZZ$. \qed
\end{theo}

This allows us to define the \emph{depth} of~$\GG$ as $\delta(\GG) \defeq \delta(\phi)$.

\begin{coro}\label{commensurability}
The depth $\delta(\GG)$ is a commensurability invariant of the free-by-cyclic group~$\GG$.
\end{coro}

\begin{proof}
Suppose  $\GG = M(\phi)$ for some $\phi \in \Out(\free)$.
Any finite index subgroup~$H \le \GG$ is the subgroup $M(\Psi) = \langle A , xt^k \rangle$ for some $x \in \free$, $k\ge 1$, $\Psi \in \phi^k$, where $A \defeq H \cap \free$  is a $\Psi$-invariant finite index subgroup of~$\free$.
Let $\psi \in \Out(A)$ be the outer automorphism containing the restriction of $\Psi$ to $A$.
By \cref{Prop:powers} (and \cref{Thm:depthinvariant}), we have 
\[\delta(H) = \delta(\psi) = \delta(\phi^k) = \delta(\phi) = \delta(\GG). \qedhere\]
\end{proof}

Unpacking the proof of \cref{Thm:depthinvariant} reveals that, for $0 \le i < \delta$, there is a natural bijection (using \cref{thm:maximallamcanonical}) from~$\mathcal L^i$ to the canonical collection~$\mathcal M^{\mathrm{max}}_{\mathcal G_i'}$ of malnormal free-by-cyclic subgroup systems of~$\mathcal G_i'$ whose meet is~$\mathcal G_{i+1}$.
These bijections define a natural order-preserving bijection from~$\mathcal L^+_\phi$ to the canonical collection~$\mathcal C(\GG) \defeq \bigcup_i \mathcal M^{\mathrm{max}}_{\mathcal G_i'}$ of free-by-cyclic subgroup systems of~$\GG$.
So the set~$\mathcal L^+_\phi$ and strata~$\mathcal L^i$ are canonical in~$\GG$.
In particular, the poset~$\mathcal L^+_\phi$ is canonical in~$\GG$ when $\delta(\GG) \le 1$.
Unfortunately, the bijection need not be an order-isomorphism: 
the partial order on~$\mathcal C(\GG)$ induces a partial order on~$\mathcal L^+_\phi$ that forces $[\Lambda_1] \preccurlyeq [\Lambda_2]$ for all $[\Lambda_1] \in \mathcal L^i$, $[\Lambda_2] \in \mathcal L^j$ and $ j < i$.
In the last section, we will construct a natural order-isomorphism from~$\mathcal L^+_\phi$ to another canonical collection~$\mathcal H(\GG)$ of free-by-cyclic subgroup systems of~$\GG$.

\section{Group invariance of the poset of laminations}
\label{sec:grp_inv_poset}

Fix $\phi \in \Out(\free)$, and set $\GG \defeq M(\phi) = \free \rtimes_{\phi} \ZZ$. 
Our final goal is to show that the poset $\mathcal L_\phi^+$ of attracting lamination orbits is canonical in $\GG$ by constructing a canonical collection~$\mathcal H(\GG)$ of free-by-cyclic subgroup systems that is naturally order-isomorphic to $\mathcal L_\phi^+$.

 In \S\ref{subsec:support_ss}, 
 we defined the supporting subgroup system~$\mathcal S[\Lambda]$ of $[\Lambda] \in \mathcal L^+_\phi$. We now give an inductive construction of~$\mathcal S[\Lambda]$, which in turn will be used to construct the canonical collection $\mathcal{H}(\GG)$. Recall the definition in \S\ref{subsec:support_ss} of the support~$\Lambda^\downarrow$ and the exterior~$\Lambda^\perp$.

Let $\delta \defeq \delta(\phi)$ and $\mathcal L^{i} \defeq \left\{\text{maximal } [\Lambda] \in \mathcal L^+_\phi \setminus \cup_{j=0}^{i-1} \mathcal L^j \right\}$ for $0 \le i \le \delta$.
Fix $[\Lambda] \in \mathcal L_\phi^+$, and set $\mathcal S_{-1}[\Lambda] \defeq \{ [\free] \}$.
For $0 \le j \le \delta$, let $\mathcal F_j'[\Lambda]$ be the $\phi$-invariant minimal free factor system of $\mathcal S_{j-1}[\Lambda]$ that carries $\mathcal L^+_{\phi|S_{j-1}[\Lambda]}$, and define
\[ \mathcal S_j[\Lambda] \defeq \mathcal N(\Lambda^\perp \cap \mathcal L^{j}) \wedge \mathcal F_{j}'[\Lambda]. \]
Note that $\mathcal L^\delta = \emptyset$ and $\mathcal N(\emptyset) = \{[\free]\}$.
By \cref{cor:NSS}, $\mathcal S_j[\Lambda] \sqsubseteq \mathcal F_{j}'[\Lambda]$ is a $\phi$-invariant malnormal subgroup system with finite type and no $\phi|\mathcal F_{j}'[\Lambda]$-twins;
moreover, the remaining lamination orbits are
$\mathcal L^+_{\phi|\mathcal S_j[\Lambda]} = \mathcal L^+_{\phi} \setminus \cup_{k=0}^{j} (\Lambda^\perp \cap \mathcal L^{k})$ by \cref{lem:nonattractingcarry}.
In particular, 
\[ \mathcal L^+_{\phi|\mathcal S_{\delta-1}[\Lambda]} = \mathcal L^+_{\phi|\mathcal S_{\delta}[\Lambda]} = \Lambda^\downarrow.\]

\begin{lem}\label{lem:supportsystems_redefined}
For $[\Lambda]\in\mathcal L_\phi^+$, the {supporting subgroup system} $\mathcal S[\Lambda] = \mathcal S_{\delta(\phi)}[\Lambda]$.
\end{lem}
\begin{proof}
    Recall that $\mathcal S[\Lambda]$ is the minimal free factor system of $\mathcal N(\Lambda^\perp)$ that carries $\Lambda^\downarrow = \mathcal L^+_{\phi|\mathcal N(\Lambda^\perp)}$.
    Set $\delta \defeq \delta(\phi)$.
    To start, $\mathcal S_0[\Lambda]$ is a free factor system of $\mathcal N(\Lambda^\perp \cap \mathcal L^0)$ by \Cref{lem:meet_ffs_is_ffs}.
    If $\mathcal S_{j-1}[\Lambda]$ is a free factor system of $\mathcal N(\Lambda^\perp \cap \bigcup_{k=0}^{j-1} \mathcal L^k)$, then
    $\mathcal S_j[\Lambda] \defeq \mathcal N(\Lambda^\perp \cap \mathcal L^j) \wedge \mathcal F_{j}'[\Lambda]$ is a free factor system of $\mathcal N(\Lambda^\perp \cap \bigcup_{k=0}^{j} \mathcal L^k)$ by \Cref{lem:transitivityffs,lem:meet_ffs_is_ffs}.
    By induction, $ \mathcal S_\delta[\Lambda] \sqsubseteq \mathcal S_{\delta-1}[\Lambda]$ are both free factor systems of $\mathcal N(\Lambda^\perp)$.
    As $\mathcal S_\delta[\Lambda]$ carries $\mathcal L^+_{\phi|\mathcal S_{\delta}[\Lambda]} = \Lambda^\downarrow$,
    we get $\mathcal S[\Lambda] \sqsubseteq \mathcal S_\delta[\Lambda]$ by minimality of~$\mathcal S[\Lambda]$, and $\mathcal S[\Lambda]$ is a free factor system of $\mathcal S_{\delta-1}[\Lambda]$.
    Yet, $\mathcal S_\delta[\Lambda] = \mathcal F'_\delta[\Lambda]$ is the minimal free factor system of $\mathcal S_{\delta-1}[\Lambda]$ that carries $\mathcal L^+_{\phi|\mathcal S_{\delta-1}[\Lambda]} = \Lambda^\downarrow$.
    So $\mathcal S_\delta[\Lambda] \sqsubseteq \mathcal S[\Lambda]$, and the two subgroup systems are equal.
\end{proof}

\begin{example}\label{ex:SupportingSystem2}
    Consider the poset of laminations $\mathcal L_{\phi}^+$ shown in \Cref{fig:supporting_system2}, where $[\Lambda_i] = \Lambda_i$ for all~$i$. Let us look at the subgroup systems $\mathcal{S}_j[\Lambda_6]$. For brevity, we will just write $\mathcal{S}_j$. We will also assume $\mathcal{F}_j'[\Lambda_6] = \mathcal{S}_{j-1}[\Lambda_6]$ for simplicity. In the following computations, the first equality is obtained by identifying $\Lambda_6^{\perp}\cap \mathcal{L}^j$, and the second equality by identifying $\mathcal L_{\phi|S_{j-1}}^{\mathrm{max}}$ (indicated by red circles in \Cref{fig:supporting_system2}). 
    \begin{align*}
        \mathcal{S}_0 & = \mathcal N(\{\Lambda_1, \Lambda_2\}) \wedge \mathcal S_{-1} 
         =\mathcal N(\mathcal L^0) \wedge [\free]
         =\mathcal{F}_1 
         \\
        \mathcal{S}_1 & = \mathcal N(\{\Lambda_3, \Lambda_4, \Lambda_5\}) \wedge \mathcal{S}_0  =\mathcal N(\mathcal L^1) \wedge \mathcal{F}_1  =\mathcal{F}_2 
        \\
        \mathcal{S}_2 & = \mathcal N(\{\Lambda_7, \Lambda_8\}) \wedge \mathcal{S}_1  =\mathcal N(\mathcal L_{\phi|S_1}^{\mathrm{max}} \setminus \Lambda_6) \wedge \mathcal{F}_2 
        \\
        \mathcal{S}_3 & = \mathcal N(\{\Lambda_{11}\}) \wedge \mathcal{S}_2  =\mathcal N(\mathcal L_{\phi|S_2}^{\mathrm{max}} \setminus \Lambda_6) \wedge \mathcal{S}_2 
        \\
        \mathcal{S}_4 & = \mathcal N(\{\Lambda_{13}, \Lambda_{14}\}) \wedge \mathcal{S}_3  =\mathcal N(\mathcal L_{\phi|S_3}^{\mathrm{max}} \setminus \Lambda_6) \wedge \mathcal{S}_3 
        \\
        \mathcal{S}_5 & = \mathcal N(\emptyset) \wedge \mathcal{S}_4  = \mathcal N(\mathcal L_{\phi|S_4}^{\mathrm{max}} \setminus \Lambda_6) \wedge \mathcal{S}_4 = \mathcal S_4 
        \\
        & \mathcal{L}_{\phi|\mathcal{S}_4}^+  = \mathcal{L}_{\phi|\mathcal{S}_5}^+ = \Lambda_6^{\downarrow}
    \end{align*}

\tikzset{node style ge/.style={circle}}

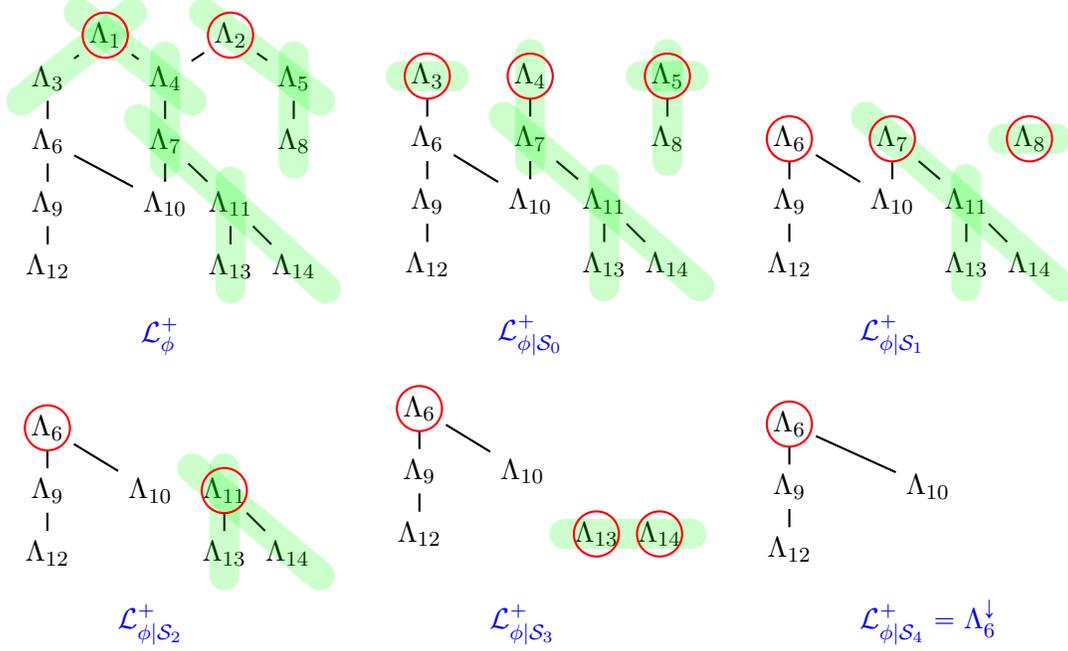
\begin{figure}
   \centering
       
\begin{minipage}[b]{0.31\textwidth}
\hspace{0.5cm}
\begin{tikzpicture}[scale=0.5][baseline=(A.center)]
  \tikzset{BarStyle/.style =   {opacity=.2,line width=4 mm,line cap=round,color=#1}}

\matrix (A) [matrix of math nodes,column sep=0 mm] 
{ & \Lambda_1 & \  & \Lambda_2 & \ \\
  \Lambda_3 & \  & \Lambda_4 & \ & \Lambda_5 \\
  \ &\ & \ & \ & \\
  \Lambda_6 & \  & \Lambda_7 & \ & \Lambda_8 \\
  \ & \ & \ & \ & \ \\
   \Lambda_9 & \  & \Lambda_{10} & \Lambda_{11} &\ \\
     \ & \ & \ & \ & \ \\
        \Lambda_{12} & \  & \ & \Lambda_{13} &\Lambda_{14} \\  
            \ & \ & \ & \ & \ \\
  \ & \ & \textcolor{blue}{\mathcal{L}^{+}_{\phi}} \ &  \ & \ \\
};

\draw[thick] (A-1-2) to (A-2-1); 
\draw[thick] (A-1-2) to (A-2-3); 
\draw[thick] (A-1-4) to (A-2-3); 
\draw[thick] (A-1-4) to (A-2-5); 
\draw[thick] (A-2-1) to (A-4-1); 
\draw[thick] (A-4-1) to (A-6-1); 
\draw[thick] (A-6-1) to (A-8-1);
\draw[thick] (A-4-1) to (A-6-3); 
\draw[thick] (A-2-3) to (A-4-3);
\draw[thick] (A-4-3) to (A-6-3); 
\draw[thick] (A-4-3) to (A-6-4); 
\draw[thick] (A-2-5) to (A-4-5); 
\draw[thick] (A-6-4) to (A-8-4); 
\draw[thick] (A-6-4) to (A-8-5);

 \draw [BarStyle=green] (A-1-4.north west) to (A-2-5.south east) ;
 \draw [BarStyle=green] (A-1-2.north west) to (A-2-3.south east) ;
 \draw [BarStyle=green] (A-4-3.north west) to (A-8-5.south east) ;
 \draw [BarStyle=green] (A-6-4.north) to (A-8-4.south) ;

\draw [BarStyle=green] (A-2-3.north) to (A-4-3.south) ;
\draw [BarStyle=green] (A-2-5.north) to (A-4-5.south);
 \draw [BarStyle=green] (A-1-2.north east) to (A-2-1.south west) ;

   \draw[red, thick] (A-1-2.center) circle[radius=0.6cm];
    \draw[red, thick] (A-1-4.center) circle[radius=0.6cm];

\end{tikzpicture}
\end{minipage}
\begin{minipage}[b]{0.31\textwidth}
\hspace{0.3cm}
\begin{tikzpicture}[scale=0.5][baseline=(A.center)]
  \tikzset{BarStyle/.style =   {opacity=.2,line width=4 mm,line cap=round,color=#1}}

\matrix (A) [matrix of math nodes, column sep=0 mm] 
{ 
  \Lambda_3 & \  & \Lambda_4 & \ & \Lambda_5 \\
  \ &\ & \ & \ & \\
  \Lambda_6 & \  & \Lambda_7 & \ & \Lambda_8 \\
  \ & \ & \ & \ & \ \\
   \Lambda_9 & \  & \Lambda_{10} & \Lambda_{11} &\ \\
     \ & \ & \ & \ & \ \\
        \Lambda_{12} & \  & \ & \Lambda_{13} &\Lambda_{14} \\    
            \ & \ & \ & \ & \ \\
  \ & \ & \textcolor{blue}{\mathcal{L}^{+}_{\phi\mid\mathcal{S}_0}} & \ & \ \\
};

\draw[thick] (A-1-1) to (A-3-1); 
\draw[thick] (A-3-1) to (A-5-1); 
\draw[thick] (A-5-1) to (A-7-1); 
\draw[thick] (A-1-3) to (A-3-3); 
\draw[thick] (A-3-3) to (A-5-3); 
\draw[thick] (A-1-5) to (A-3-5);
\draw[thick] (A-3-3) to (A-5-4); 
\draw[thick] (A-3-1) to (A-5-3);
\draw[thick] (A-5-4) to (A-7-4);
\draw[thick] (A-5-4) to (A-7-5);

 \draw [BarStyle=green] (A-1-1.west) to (A-1-1.east) ;
  \draw [BarStyle=green] (A-1-5.west) to (A-1-5.east) ;

 \draw [BarStyle=green] (A-1-3.north) to (A-3-3.south) ;

 \draw [BarStyle=green] (A-1-5.north) to (A-3-5.south) ;

  \draw [BarStyle=green] (A-3-3.north west) to (A-7-5.south east) ;

 \draw [BarStyle=green] (A-5-4.north) to (A-7-4.south) ;

   \draw[red, thick] (A-1-1.center) circle[radius=0.6cm];

 \draw[red, thick] (A-1-3.center) circle[radius=0.6cm];

 \draw[red, thick] (A-1-5.center) circle[radius=0.6cm];

\end{tikzpicture}
\end{minipage}
\begin{minipage}[b]{0.31\textwidth}
\begin{tikzpicture}[scale=0.5][baseline=(A.center)]
  \tikzset{BarStyle/.style =   {opacity=.2,line width=4 mm,line cap=round,color=#1}}

\matrix (A) [matrix of math nodes, column sep=0 mm] 
{ 
  \Lambda_6 & \  & \Lambda_7 & \ & \Lambda_8 \\
  \ & \ & \ & \ & \ \\
   \Lambda_9 & \  & \Lambda_{10} & \Lambda_{11} &\ \\
     \ & \ & \ & \ & \ \\
        \Lambda_{12} & \  & \ & \Lambda_{13} &\Lambda_{14} \\ 
            \ & \ & \ & \ & \ \\
  \ & \ & \textcolor{blue}{\mathcal{L}^{+}_{\phi\mid\mathcal{S}_1}} & \ & \ \\
};

\draw[thick] (A-1-1) to (A-3-1); 
\draw[thick] (A-1-1) to (A-3-3); 
\draw[thick] (A-3-1) to (A-5-1); 
\draw[thick] (A-1-3) to (A-3-3); 
\draw[thick] (A-1-3) to (A-3-4); 
\draw[thick] (A-3-4) to (A-5-5); 
\draw[thick] (A-3-4) to (A-5-4); 

 \draw [BarStyle=green] (A-1-3.north west) to (A-5-5.south east) ;

  \draw [BarStyle=green] (A-1-5.west) to (A-1-5.east) ;

 \draw [BarStyle=green] (A-3-4.north) to (A-5-4.south) ;

 \draw[red, thick] (A-1-1.center) circle[radius=0.6cm];

 \draw[red, thick] (A-1-3.center) circle[radius=0.6cm];

 \draw[red, thick] (A-1-5.center) circle[radius=0.6cm];

\end{tikzpicture}
\end{minipage}
\begin{minipage}[b]{0.31\textwidth}
\hspace{0.5cm}
\begin{tikzpicture}[scale=0.5][baseline=(A.center)]
  \tikzset{BarStyle/.style =   {opacity=.2,line width=4 mm,line cap=round,color=#1}}

\matrix (A) [matrix of math nodes, column sep=0 mm] 
{
  \Lambda_6 & \  & \ & \ & \ \\
  \ & \ & \ & \ & \ \\
   \Lambda_9 & \  & \Lambda_{10} & \Lambda_{11} &\ \\
     \ & \ & \ & \ & \ \\
        \Lambda_{12} & \  & \ & \Lambda_{13} &\Lambda_{14} \\     
            \ & \ & \ & \ & \ \\
  \ & \ & \textcolor{blue}{\mathcal{L}^{+}_{\phi\mid\mathcal{S}_2}} & \ & \ \\
};

\draw[thick] (A-1-1) to (A-3-1); 
\draw[thick] (A-1-1) to (A-3-3); 
\draw[thick] (A-3-1) to (A-5-1); 
\draw[thick] (A-3-4) to (A-5-4); 
\draw[thick] (A-3-4) to (A-5-5);

 \draw [BarStyle=green] (A-3-4.north west) to (A-5-5.south east) ;

  \draw [BarStyle=green] (A-3-4.north) to (A-5-4.south) ;

\draw[red, thick] (A-3-4.center) circle[radius=0.6cm];

\draw[red, thick] (A-1-1.center) circle[radius=0.6cm];

\end{tikzpicture}
\end{minipage}
\begin{minipage}[b]{0.31\textwidth}
\hspace{0.2cm}
\begin{tikzpicture}[scale=0.5][baseline=(A.center)]
  \tikzset{BarStyle/.style =   {opacity=.2,line width=4 mm,line cap=round,color=#1}}
\matrix (A) [matrix of math nodes,column sep=0 mm] 
{\Lambda_6 & \  & \ & \ & \ \\
  \ & \ & \ & \ & \ \\
   \Lambda_9 & \  & \Lambda_{10} & \ &\ \\
     \ & \ & \ & \ & \ \\
        \Lambda_{12} & \  & \ & \Lambda_{13} &\Lambda_{14} \\      
        \ & \ & \ & \ & \ \\
             \ & \ & \ & \ & \ \\
  \ & \ & \textcolor{blue}{\mathcal{L}^{+}_{\phi\mid\mathcal{S}_3}} & \ & \ \\
   };
\draw[thick] (A-1-1) to (A-3-1); 
\draw[thick] (A-3-1) to (A-5-1); 
\draw[thick] (A-1-1) to (A-3-3); 
\draw [BarStyle=green] (A-5-4.west) to (A-5-5.east) ;
  \draw[red, thick] (A-1-1.center) circle[radius=0.6cm];
\draw[red, thick] (A-5-4.center) circle[radius=0.6cm];
\draw[red, thick] (A-5-5.center) circle[radius=0.6cm];
 \end{tikzpicture}
\end{minipage}
\begin{minipage}[b]{0.31\textwidth}
\begin{tikzpicture}[scale=0.5][baseline=(A.center)]
  \tikzset{BarStyle/.style =   {opacity=.2,line width=4 mm,line cap=round,color=#1}}
\matrix (A) [matrix of math nodes, column sep=0 mm] 
{                          
\Lambda_6 & \  &  & \ & \\
  \ & \ & \ & \ &  \\
   \Lambda_9 & \  & \Lambda_{10} &  &\ \\
     \ & \ & \ & \ & \ \\   
   \Lambda_{12} & \  & \ & \ &\ \\ 
             \ & \ & \ & \ & \ \\
  \ & \ & \textcolor{blue}{\mathcal{L}^{+}_{\phi\mid\mathcal{S}_4}=\Lambda_{6}^{\downarrow}} & \ & \ \\
};
\draw[thick] (A-1-1) to (A-3-1); 
\draw[thick] (A-3-1) to (A-5-1); 
\draw[thick] (A-1-1) to (A-3-3); 
\draw[red, thick] (A-1-1.center) circle[radius=0.6cm];

\end{tikzpicture}
\end{minipage}

\caption{Schematic for building $\mathcal S[\Lambda_6]=\mathcal S_5[\Lambda_6]$. For each $j$, the highlighted sets denote $\Lambda_6^{\perp} \cap \mathcal L_{\phi|S_{j}}^+$, and the circled laminations denote the elements of $\mathcal L_{\phi|S_{j}}^{\mathrm{max}}$. }
    \label{fig:supporting_system2}
\end{figure}

    Note that $[\Lambda_6] \in \mathcal{L}^2$ and $S_j[\Lambda_6] = \mathcal{F}_{j+1}$ for $0 \leq j <2$.
\end{example}

Now consider the corresponding mapping tori in $\GG$.
For $[\Lambda] \in \mathcal L_\phi^+$, set $\mathcal H[\Lambda] \defeq M(\phi|\mathcal S[\Lambda])$. 
Define the collection 
\[ \mathcal H(\phi) \defeq \left\{ \mathcal H[\Lambda] \,:\, [\Lambda] \in \mathcal L^+_\phi \right\}. \]
By \cref{lem:inclusionsubsys,lem:supporting_poset_isomorphism}, $\mathcal H(\phi)$ is naturally order-isomorphic to $\mathcal S(\phi)$ and $\mathcal L_\phi^+$.

\begin{theo}\label{thm:posetlaminationinvariant}
Let $\phi \in \Out(\free)$ and $\GG=\free \rtimes_\phi \ZZ$. 
The collection $\mathcal H(\phi)$ is canonical in $\GG$.

In particular, the poset $\mathcal L_\phi^+$ is a group invariant of $\GG$.
\end{theo}

As a canonical collection, we can finally define $\mathcal H(\GG) \defeq \mathcal H(\phi)$.

\begin{proof}
    By the preceding discussion, the poset $\mathcal L^+_\phi$ is a group invariant of $\GG$ once we know $\mathcal H(\phi)$ is canonical in~$\GG$.
    Set $\delta \defeq \delta(\phi)$.
    Recall the definition of the canonical mapping tori $\mathcal G_i \defeq M(\phi|\mathcal F_i)$ and $\mathcal G_i' \defeq M(\phi|\mathcal F_i')$ in $\GG$ for $0 \le i \le \delta$ (\cref{def:canonicalfiltration}).
    It will be enough to prove that, for $0 \le i < \delta$, the stratum
    \[ \mathcal H(\phi)^i \defeq \left\{ \mathcal H[\Lambda] \,:\, [\Lambda] \in \mathcal L^i \right\}\]
    is canonical in~$\mathcal G_i'$.

    Fix $0 \le i < \delta$.
    For $0 \le j \le \delta$, define the mapping tori
    \[ \mathcal H(\phi)^i_j \defeq \left\{ M(\phi|\mathcal S_j[\Lambda]) \,:\, [\Lambda] \in \mathcal L^i \right\}. \]
    By construction, $\mathcal S_{j-1}[\Lambda] = \mathcal F_{j}$ and $\mathcal F_j'[\Lambda] = \mathcal F_j'$ for any $[\Lambda] \in \mathcal L^i$ and $0 \le j \le i$.
    So the collection 
    $\{ \mathcal S_i[\Lambda] \,:\, [\Lambda] \in \mathcal L^i \}$
    consists of the systems in $\mathcal F_i'$ that are the meets of all but one of $\left\{ \mathcal N[\Lambda|\mathcal F_i'] \,:\, [\Lambda] \in \mathcal L^i \right\}$ (see \cref{ex:SupportingSystem2}).
    As $\mathcal L^i = \mathcal L^{\mathrm{max}}_{\phi|\mathcal F_i'}$,
    the collection 
    $\mathcal H(\phi)^i_i$ consists of the systems in $\mathcal G_i'$ that are the meets of all but one of $\left\{ \mathcal M[\Lambda|\mathcal F_i'] \,:\, [\Lambda] \in \mathcal L^{\mathrm{max}}_{\phi|\mathcal F_i'} \right\}$.
    By \cref{thm:maximallamcanonical}, the latter collection is canonical in $\mathcal G_i'$, and hence, $\mathcal H(\phi)^i_i$ is canonical in $\mathcal G_i'$.
    For $j > i$, it suffices to algebraically characterize $\mathcal H(\phi)_j^i$ in terms of $\mathcal H(\phi)_{j-1}^i$.
    Indeed, since $\mathcal H(\phi)_i^i$ is canonical in $\mathcal G_i'$ and canonicity is transitive, 
    the algebraic characterization inductively implies $\mathcal H(\phi)^i = \mathcal H(\phi)^i_{\delta}$ is canonical in $\mathcal G_i'$.

    For $j > i$ and $[\Lambda] \in \mathcal L^i$, the system $\mathcal S_j[\Lambda]$ is alternatively characterized as follows:
    let $\mathcal F_{j}'[\Lambda]$ be the $\phi$-invariant minimal free factor systems of $\mathcal S_{j-1}[\Lambda]$ that carries $\mathcal L_{\phi|\mathcal S_{j-1}[\Lambda]}^+$;
    then $\mathcal S_j[\Lambda]$ is the unique subgroup system in $\mathcal F_{j}'[\Lambda]$ that is not supported by $\mathcal F_{i+1}$ but is the meet of all but one of the nonattracting subgroup systems of $\mathcal L_{\phi|\mathcal F_{j}'[\Lambda]}^{\mathrm{max}}$ (see \cref{ex:SupportingSystem2}).
    Similarly, we characterize $\mathcal H(\phi)^i_j$ in terms of $\mathcal H(\phi)^i_{j-1}$.
    For any $\mathcal H \in \mathcal H(\phi)^i_{j-1}$, i.e. $\mathcal H = M(\phi|\mathcal S_{j-1}[\Lambda])$ for some $[\Lambda] \in \mathcal L^i$, set $\mathcal H' \defeq M(\phi|\mathcal F_{j}'[\Lambda])$; 
    by \cref{prop:passing to no cyclic splittings}, $\mathcal H'$ is canonical in $\mathcal H$.
    Finally, by \cref{thm:maximallamcanonical} again,
    the mapping torus $M(\phi|\mathcal S_j[\Lambda]) \in \mathcal H(\phi)^i_j$ is the unique free-by-cyclic subgroup system in $\mathcal H'$ not supported by $\mathcal G_{i+1}$ but is the meet of all but one of the maximal proper malnormal free-by-cyclic subgroup systems in $\mathcal H'$ supporting $\mathcal P(\mathcal H')$.
\end{proof}

The \emph{depth spectrum} of the free-by-cyclic group~$\GG$ is $\delta\mathcal S(\GG) \defeq \delta\mathcal S(\phi)$, which is well-defined by \Cref{thm:posetlaminationinvariant}.
As with \Cref{commensurability}, we deduce the following from \Cref{Prop:powers}.

\begin{coro}\label{cor:depthspectrum}
    The depth spectrum $\delta\mathcal{S}(\GG)$ is a commensurability invariant of~$\GG$. \qed
\end{coro}

\appendix
\crefalias{section}{appendix}
\crefalias{subsection}{appendix}

\section{Attracting laminations and relative train tracks} \label{app:A}

The aim of this appendix is to prove \cref{Prop:powers} (in \cref{app:passingtofinite}) and \cref{lem:approximategenline} (in \cref{app:genericapprox}). Both proofs require some general background regarding relative train tracks and attracting laminations, which we recall now.

Let $\Gamma$ be a connected finite graph whose fundamental group is isomorphic to $\free$. A \emph{marking} is an isomorphism $m \colon \pi_1(\Gamma) \to \free$. The marking induces an equivariant homeomorphism between the Gromov boundary $\partial_\infty \widetilde \Gamma$ of the universal cover $\widetilde \Gamma$ of $\Gamma$ and the Gromov boundary of $\free$. This in turn induces a homeomorphism between $\BB(\free)$ and the space $\BB(\Gamma) \defeq \pi_1(\Gamma)\backslash\partial^2\widetilde\Gamma$ of biinfinite, unoriented, reduced paths in $\Gamma$. 
We fix 
this identification $\BB(\free) \cong \BB(\Gamma)$.

Let $\phi \in \Out(\free)$ be a given outer automorphism. 
Following~\cite[Definition~3.1.5]{BesFeiHan00}, a line $\ell \in \BB(\free)$ is \emph{generic} for $\phi$ if the following hold:

\begin{itemize}
\item $\ell$ is \emph{birecurrent}, that is, every finite subpath of $\ell$ occurs infinitely often in both ends of $\ell$; and

\item there is an integer $k \ge 1$ and neighborhood $U \subseteq \BB(\free)$ of $\ell$ such that $\phi^k(U) \subseteq U$ and $\{ \phi^{kn}(U) : n \ge 1\}$ is a neighborhood basis for $\ell$; 
\end{itemize}

The cited definition has a third condition ($\ell$ is not the axis of an element) that is made redundant by the second condition under our assumption (in~\S\ref{subsec:freebycyclic_mappingtori}) that~$\free$ is nonabelian. 
An \emph{attracting lamination} for~$\phi$ is the closure in $\BB(\free)$ of a generic line for~$\phi$.

\medskip

We now turn to the relevant background regarding relative train track maps. Let $f \colon \Gamma \to  \Gamma$ be a relative train track representing $\phi$. The graph $\Gamma$ comes equipped with an $f$-invariant filtration 
\[\Gamma=\Gamma_0 \supsetneq \Gamma_1 \supsetneq \ldots \supsetneq \Gamma_{K} = \emptyset. \]
For $0\le r < K$, set $H_r \defeq \overline{\Gamma_r\setminus {\Gamma}_{r+1}}$ be the \emph{$r$-th stratum} of $\Gamma$. For a filtration associated with a relative train track map, a stratum is \emph{exponentially growing} (EG),  \emph{nonexponentially growing} (NEG), or a \emph{zero stratum} depending on the spectral radius of the transition matrix associated with the stratum. An edge path $\rho$ \emph{crosses} a stratum $H_r$ if $\rho$ contains an edge of $H_r$. 

If $\sigma$ is a (not necessarily reduced) edge path in $\Gamma$, we denote by $[\sigma]$ its associated reduced edge path: it is the immersed path homotopic to $\sigma$ relative to its endpoints.
If $\sigma$ is a closed edge path, we also denote by $[\sigma]$ its associated cyclically reduced edge path, which is the immersed closed path freely homotopic to $\sigma$.
Let $f_\#$ be the map sending an edge path $\sigma$ of $\Gamma$ to $[f(\sigma)]$. 

Bestvina--Feighn--Handel~\cite[Lemma~3.1.10]{BesFeiHan00} proved that, for every attracting lamination orbit $[\Lambda]\in \mathcal{L}_\phi^+$, there exists a unique EG stratum $H_r$ such that the highest stratum crossed by any generic line of $[\Lambda]$ is $H_r$. Moreover, there exists a generic line $\ell \in \Lambda$ obtained as an exhaustion by $f_\#$-iterates of some (equivalently any) edge $e$ of $H_r$~\cite[Lemma~3.1.10(4)]{BesFeiHan00}.

A \emph{turn} is an unordered pair of oriented edges of $\Gamma$ originating at a common vertex. A turn has \emph{depth $r$} if the edges defining it are contained in the subgraph $\Gamma_r$ but not $\Gamma_{r+1}$. A turn is \emph{nondegenerate} if it consists of distinct oriented edges, and is \emph{degenerate} otherwise.  An
edge path $\rho$ \emph{crosses a turn $\{e_1,e_2\}$} if it contains the subpath $\overline e_1{e}_2$ {(or $\overline e_2 e_1$)}. For $e \in E\Gamma$, let $Tf(e)$ be the first edge of the reduced (nontrivial) path $f(e)$. The relative train track~$f$ induces a map $Tf$ of the turns by sending $\{e_1,e_2\}$ to $\{Tf(e_1),Tf(e_2)\}$. 
A turn is 
\emph{legal} (with respect to $f$) if its $Tf$-iterates are all nondegenerate. A path $\rho$ is \emph{r-legal} if all the depth $r$ turns it crosses are legal. We recall the following lemma due to Bestvina--Handel. 

\begin{lem}\cite[Lemma~5.8]{BesHan92}\label{lem:nocancelEG}
Let $f \colon \Gamma \to \Gamma$ be a relative train track map and $H_r \subseteq \Gamma$ an EG stratum. Suppose that $\sigma=a_1b_1\cdots a_{\ell}b_{\ell}$ is the decomposition of an $r$-legal path into subpaths $a_j \subseteq H_r$ and $b_j \subseteq \Gamma_{r+1}$ (where $a_1$ and $b_{\ell}$ might be trivial). Then, for every $i \in \{1,\ldots,\ell\}$, the path $f(a_i)$ is a reduced edge path, and \[f_\#(\sigma)=f(a_1)f_\#(b_1)f(a_2)f_\#(b_2) \cdots f(a_{\ell})f_\#(b_{\ell}). \qedhere\]
\end{lem}

Let $A \le \free$ be a finite index subgroup and $p_A \colon \Gamma_A \to \Gamma$ the cover of $\Gamma$ associated with the conjugacy class~$[A]$. 
Suppose  that $\phi[A] = [A]$. 
Then, there exists a lift $f_A \colon \Gamma_A \to \Gamma_A$ of $f$ such that $p_A \circ f_A = f \circ p_A$, i.e.~the following diagram commutes. 
\[\begin{tikzcd}
\Gamma_A \arrow[d,"p_A"' ] \arrow[r, "f_A"] & \Gamma_A \arrow[d,"p_A" ] \\
\Gamma \arrow[r,"f"' ] & \Gamma
\end{tikzcd}\]

We now explain how a relative train track structure on $f_A$ is induced from that of $f$. 

\begin{lem}[see {\cite[Lemma~4.5]{CoulonHilion}}]\label{lem:coveringtraintrack}
Assume, up to taking a power of $\phi$, that the following holds:

\begin{itemize}
\item every EG stratum of $f$ has a primitive transition matrix; 
\item every NEG stratum $H_r$ of $f$ consists of a single edge $e$ with $f(e)=eu$ for some closed edge path $u \subseteq \Gamma_{r+1}$; 
\item for every vertex $v$ of $\Gamma$, $f(v)$ is fixed by $f$.
\end{itemize}

Then the lift $f_A \colon \Gamma_A \to \Gamma_A$ admits a filtration which makes it a relative train track map. Moreover, for any EG stratum $H_A \subseteq \Gamma_A$, there is a {unique} EG stratum $H_r \subseteq \Gamma$ such that:
\begin{itemize}
    \item $H_A$ is contained in $p_A^{-1}(H_r)$; {$p_A(H_A) = H_r$};
    \item $f_A(H_A) \subseteq H_A \cup p_A^{-1}(\Gamma_{r+1})$.
\end{itemize}

Conversely, for EG strata $H_r \subseteq \Gamma$, the subgraph $p_A^{-1}(H_r)$ is a union of EG strata of~$\Gamma_A$. \qed
\end{lem}

\begin{rmq}
    The uniqueness of~$H_r$ and last statement of the lemma is contained in the proof of \cite[Lemma~4.5]{CoulonHilion}, see \cite[p.~1985]{CoulonHilion} in particular.
\end{rmq}

\subsection{Virtual invariance of depth}\label{app:passingtofinite}

 The goal is this section is to prove the virtual invariance of the depth:

\begin{propn}[\ref{Prop:powers}]
    Suppose $A \le \free$ has finite index, $\Phi \colon \free \to \free$ is an automorphism such that $\Phi(A) = A$, and $\Psi \colon A \to A$ is the restriction of $\Phi$ to $A$.
    Set $\phi \defeq [\Phi]$ in $\Out(\free)$ and $\psi \defeq [\Psi]$ in $\Out(A)$. Then $\delta(\phi^m) = \delta(\phi) = \delta(\psi)$ and  $\delta\mathcal S(\phi^m) = \delta\mathcal S(\phi) = \delta\mathcal S(\psi)$ for all $m \ge 1$.
    
    In fact, there is a natural surjective strict-order preserving map $\pi \colon \mathcal L^+(\psi) \to \mathcal L^+(\phi)$, and chains in $\mathcal L^+(\phi)$ can be lifted (via~$\pi$) to chains of the same length in $\mathcal L^+(\psi)$.
\end{propn}

\begin{proof}
We immediately get $\delta\mathcal S(\phi^m) = \delta\mathcal S(\phi)$ since $\mathcal L^+(\phi^m) = \mathcal L^+(\phi)$.
Since $\delta(\phi) = \max \delta\mathcal S(\phi)$, it will suffice to show $\delta\mathcal S(\phi) = \delta\mathcal S(\psi)$.

Let $N = [\free : A]$ be the index of $A$.
Using the natural identification $\partial_\infty A = \partial_\infty \free$ that is equivariant with respect to the inclusion $A \le \free$, we get a natural $N$-to-1 open continuous surjection $q \colon \BB(A) \to \BB(\free)$.
Note that $\Phi$ and its restriction to $A$ induce the same homeomorphism of $\partial_\infty \free$ and $q \circ \psi = \phi \circ q$.

We first describe a natural map $\pi \colon \mathcal L^+(\psi) \to \mathcal L^+(\phi)$.
Let $f \colon \Gamma \to \Gamma$ be a relative train track map representing~$\phi$, and choose a lift $f_A \colon \Gamma_A \to \Gamma_A$ representing~$\psi$. 
Assume that~$f$ satisfies the hypotheses of \Cref{lem:coveringtraintrack}---this is possible up to taking a power of~$\phi$ as $\delta\mathcal S(\phi)=\delta\mathcal S(\phi^m)$ for every $m \geq 1$. 
Thus, by \Cref{lem:coveringtraintrack}, the lift $f_A \colon \Gamma_A \to \Gamma_A$ of $f$ is a relative train track map representing $\psi \in \Out(A)$. 
For every EG stratum $H_A \subset \Gamma_A$,
there is a unique EG stratum $H_r \subseteq \Gamma$ (with lamination $\Lambda_r$) such that $p_A(H_A) = H_r$; thus, we can set $\pi(\Lambda_A) = \Lambda_r$, for every lamination $\Lambda_A$ such that $H_A$ is the highest stratum it crosses.

For surjectivity of~$\pi$, let $H_r \subseteq \Gamma$ be the EG stratum associated with some $\bar \Lambda \in \mathcal L^+(\phi)$
and $e \in E\Gamma_A$ a lift of an edge in $H_r$. By \Cref{lem:coveringtraintrack} again, the edge $e$ is contained in some EG stratum of $f_A$. Hence the $(f_A)_\#$-iterates of $e$ 
exhaust a generic line $\ell'$ of some attracting lamination $\bar \Lambda_A \in \mathcal L^+(\psi)$. The image $q(\ell')$ is a generic line of some attracting lamination of $\phi$. As the highest stratum crossed by $\ell'$ is in the preimage of $H_r$, the highest stratum crossed by $q(\ell')$ is $H_r$. Thus, $\pi(\bar \Lambda_A)=\bar \Lambda$. So $\bar \Lambda$ is in the image of $\pi$, and $\pi$ is surjective as $\bar \Lambda$ was arbitrary.

Next, we prove $\pi$ is strict-order preserving, i.e.~$\Lambda_1 \subsetneq \Lambda_2$ implies $\pi(\Lambda_1) \subsetneq \pi(\Lambda_2)$ for $\Lambda_1, \Lambda_2$ in $\mathcal L^+(\psi)$.
For this, we find it easier to work with the abstract spaces $\BB(A), \BB(\free)$ directly rather than using relative train tracks.
Let $\ell_1, \ell_2 \in \BB(A)$ be generic lines for $\psi$ with $\Lambda_1, \Lambda_2 \in \mathcal L^+(\psi)$ their respective closures, and $U \subseteq \BB(\free)$ an arbitrary neighborhood of $q(\ell_1)$.
Suppose $\Lambda_1 \subseteq \Lambda_2$, i.e.~any neighborhood of $\ell_1$ contains $\ell_2$.
Then $q^{-1}(U)$ is a neighborhood of $\ell_1$ and must contain $\ell_2$;
therefore, $U$ contains $q(\ell_2)$, and $\pi(\Lambda_1) \subseteq \pi(\Lambda_2)$. 
Suppose, in addition, that $\pi(\Lambda_1) = \pi(\Lambda_2)$ and $V \subseteq \BB(A)$ is an arbitrary neighborhood of $\ell_2$.
Since $q$ is an open map, $q(V)$ is a neighborhood of $q(\ell_2)$ and must contain $q(\ell_1)$.
In particular, $V$ intersects the preimage $q^{-1}(q(\ell_1))$, which is a set of size $N$ that contains $\ell_1$.
So $\ell_2$ is in the closure of $q^{-1}(q(\ell_1))$, which is a union of $N$ closed subsets.
In other words, $\Lambda_2$ is in the closure $\Lambda_1' \subseteq \BB(A)$ of some $\ell_1' \in q^{-1}(q(\ell_1))$.
Thus, $\Lambda_1 \subseteq \Lambda_2 \subseteq \Lambda_1'$;
moreover, as subsets of $\partial^2 \free$, $\ell_1' = g \cdot \ell_1$ for some $g \in \free$.
Choose $k \ge 1$ such that $g^k \in A$.
By transitivity, $\Lambda_1 \subseteq \Lambda_2 \subseteq g^k \cdot \Lambda_1 = \Lambda_1$ and $\Lambda_1 = \Lambda_2$;
therefore, $\pi$ is strict-order preserving.

We conclude by observing that chains in $\mathcal L^+(\phi)$ can be lifted to chains in $\mathcal L^+(\psi)$, and so $\delta\mathcal S(\psi) = \delta\mathcal S(\phi)$.
Suppose $\bar \ell_1, \bar \ell_2 \in \BB(\free)$ are generic lines for $\phi$,  laminations $\bar \Lambda_1, \bar \Lambda_2 \in \mathcal L^+(\phi)$ are their respective closures such that $\bar \Lambda_1 \subseteq \bar \Lambda_2$, and $\ell_1 \in q^{-1}(\bar \ell_1)$ is an arbitrary lift of $\bar \ell_1$.
In the previous paragraph, we show  (using openness of $q$) that the closure $\Lambda_1 \in \pi^{-1}(\bar \Lambda_1)$ of $\ell_1$ is in some $\Lambda_2 \in \pi^{-1}(\bar \Lambda_2)$ and we are done.
\end{proof}

\subsection{Approximation of a generic line}\label{app:genericapprox}

Fix $\phi \in \Out(\free)$, and let $\mathcal A$ be a $\phi$-invariant malnormal subgroup system of $\free$ with finite type. 
We will prove \cref{lem:approximategenline}, which gives an approximation of a generic line $\ell$ of $[\Lambda]\in \mathcal{L}_{\phi|\mathcal A}^+$ by some periodic line $[[g]] \in \BB(\mathcal A)$.
First, we prove the standard case $\mathcal A = [\free]$.

\begin{lem}\label{lem:approximategenline_standard}
For $[\Lambda] \in \mathcal L^+_\phi$, some element $g \in \free$ has the following property: 
\begin{quote}for any $[\Lambda']\in \mathcal L^+_\phi$, $g$ is weakly attracted to $[\Lambda']$ if and only if  $[\Lambda'] \subseteq [\Lambda]$. 
\end{quote}
\end{lem}

\begin{proof}
    Let $f\colon \Gamma \to \Gamma$ be a relative train track representing~$\phi$ and $H_r$ the EG stratum of~$\Gamma$ associated with $[\Lambda]$.
    Pick an oriented edge $e$ of $H_r$ and an integer $M \ge 1$ such that the reduced $r$-legal path~$f_\#^M(e)$ contains at least two copies of~$e$---such an $M$ exists since $H_r$ is an EG stratum.
    Note that, for any $n \ge 1$, any finite subpath of $f_\#^n(e)$ is a finite subpath of $f_\#^M$-iterates of $f_\#^n(e)$ by \cref{lem:nocancelEG}.
    
    Define~$\sigma$ to be the cyclically reduced subpath in~$f_\#^M(e)$ joining the origins of two distinct copies of~$e$.
    We will consider~$\sigma$ as a cycle in~$\Gamma$ and denote by $[g]$ the conjugacy class of elements in~$\free$ represented by~$\sigma$.
    By construction, $\sigma e$ is a subpath of~$f_\#^M(e)$, so it is $r$-legal.
    As~$e$ is the initial edge of~$\sigma$, \cref{lem:nocancelEG} implies $f_\#^n(\sigma)$ is cyclically reduced and has $f_\#^n(e)$ as an initial subpath for all $n\ge 1$.

    Recall that we identified $\BB(\free)$ with $\BB(\Gamma)$, the space of biinfinite unoriented reduced paths in~$\Gamma$.
    Let~$k$ be the size of the $\phi$-orbit $[\Lambda]$, and assume generic lines of~$\Lambda$ cross the edge~$e$.
    By~\cite[Lemma~3.1.10]{BesFeiHan00}, an unoriented path~$\rho$ is a finite subpath of the generic lines of~$\Lambda$ if and only if it is a subpath of some unoriented $f_\#^k$-iterate of~$e$.

    Let $[\Lambda'] \in \mathcal L^+_\phi$, and suppose the element~$g$ is weakly attracted to~$[\Lambda']$.
    Pick an (arbitrary) unoriented finite subpath~$\rho'$ of a generic line of~$\Lambda'$.
    Then~$\rho'$ is a subpath of the periodic line~$\phi^N[[g]]$ for infinitely many~$N \ge 1$.
    For large~$N$, the reduced path~$f_\#^N(e)$ will be longer than~$\rho'$;
    therefore,~$\rho'$ (with some orientation) is a subpath of a cyclic permutation of the oriented path~$f_\#^N(\sigma)$ for infinitely many~$N \ge 1$.
    Since~$e$ is the initial edge of~$\sigma$, we deduce that~$\rho'$ is a subpath of the reduced path $f_\#^N(\sigma)f_\#^N(e) = f_\#^N(\sigma e)$ for some large~$N \ge 1$---the latter is a subpath of $f_\#^{N+M}(e)$ as $\sigma e$ is a subpath of $f_\#^M(e)$.
    Thus,~$\rho'$ is a subpath of the generic lines of~$\phi^N(\Lambda)$.
    As~$\rho'$ was arbitrary, we have shown that the generic lines of~$\Lambda'$ are in the closure of the generic lines of~$[\Lambda]$, and hence, $\Lambda' \subseteq [\Lambda]$.
    This extends to the $\phi$-orbit: $[\Lambda'] \subseteq [\Lambda]$.

    Conversely, suppose $[\Lambda'] \subseteq [\Lambda]$.
    Without loss of generality, we may assume $\Lambda' \subseteq \Lambda$.
    Pick an unoriented finite subpath~$\rho'$ of a generic line of~$\Lambda'$.
    Then~$\rho'$ is a subpath of a generic line of~$\Lambda$.
    Equivalently,~$\rho'$ (with some orientation) is a subpath of some $f_\#^k$-iterate of~$e$.
    By the observations in the first and second paragraphs,~$\rho'$ is a subpath of infinitely many $f_\#$-iterates of~$e$, and hence~$\sigma$.
    As~$\rho'$ was arbitrary, the element~$g$ is weakly attracted to~$\Lambda'$, and this extends to~$[\Lambda']$.
\end{proof}

\begin{lemn}[\ref{lem:approximategenline}]
For $[\Lambda] \in \mathcal L^+_{\phi|\mathcal A}$, some $\mathcal A$-peripheral element $g \in \free$ has the following property: 
\begin{quote}for any $[\Lambda']\in \mathcal L^+_{\phi|\mathcal A}$, $g$ is weakly attracted to $[\Lambda']$ if and only if  $[\Lambda'] \subseteq [\Lambda]$. 
\end{quote}
\end{lemn}

\begin{proof}
Let $\mathcal A$ be a $\phi$-invariant malnormal subgroup system of $\free$ with finite type and $[\Lambda] \in \mathcal L^+_{\phi|\mathcal A}$.
Choose an orbit representative $\Lambda \in \mathcal L^+(\phi|\mathcal A)$.
Then there is a unique conjugacy class $[A] \in \mathcal A$ such that $\Lambda \subseteq \BB(A)$.
Let $\psi \in \Out(A)$ denote the restriction $\phi^k|A \in \phi|\mathcal A$ (see \cref{def:restrictions}).
Since the action of $\psi$ on $\BB(A)$ is a restriction of the first return map of the action of $\phi$ on $\BB(\mathcal A)$, we get $\Lambda \in \mathcal L^+(\psi)$.
By \cref{lem:approximategenline_standard}, some element $g \in A$ has this property with respect to $\BB(A)$:
\begin{quote}for any $[\Lambda'] \in \mathcal L^+_\psi$, $g$ is weakly attracted to $[\Lambda']$ (under forward $\psi$-iteration) if and only if $[\Lambda'] \subseteq [\Lambda]$ (as $\psi$-orbits in $\BB(A)$).
\end{quote}
Note that $g$ is an $\mathcal A$-peripheral element of $\free$.
Let $[\Lambda'] \in \mathcal L^+_{\phi|\mathcal A}$ be an arbitrary $\phi$-orbit of attracting laminations carried by $\mathcal A$.
If $[\Lambda']$ and $\BB(A)$ are disjoint, then there is a partition of $\BB(\mathcal A)$ into two $\phi$-invariant closed subsets: one containing $[[g]]$ and the other $[\Lambda']$.
So $g$ is not weakly attracted (under forward $\phi$-iteration) to $[\Lambda']$, and $[\Lambda'] \nsubseteq [\Lambda]$.
Now suppose some $\phi$-orbit representative $\Lambda'$ is carried by $A$.
Using the previous observation about the action of $\psi$ on $\BB(A)$, we can conclude the proof:
\begin{enumerate}
    \item $[\Lambda'] \subseteq [\Lambda]$ as $\psi$-orbits in $\BB(A)$ if and only if $[\Lambda'] \subseteq [\Lambda]$ as $\phi$-orbits in $\BB(\mathcal A)$; and
    \item $g$ is weakly attracted to $[\Lambda']$ under forward $\psi$-iteration if and only if $g$ is weakly attracted to $[\Lambda']$ under forward $\phi$-iteration. \qedhere
\end{enumerate}
\end{proof}

\section{Canonical topological trees}\label{app:toptree}

We explain how to deduce \cref{Thm:NSS} from the results in the cited reference \cite{Mutanguha22}.

Fix an outer automorphism $\phi \in \Out(\free)$ and a representative $\Phi \in \phi$.
For induction, set $N_1 \defeq \free$,  $\phi_1 \defeq \phi$, and $\mathcal L^+_1 \defeq \mathcal L^+_\phi$.
By \cite[Proposition~2.2]{Mutanguha22}, there is a maximal lamination orbit $[\Lambda_1] \in \mathcal L^+_1$ and an $\mathbb R$-tree $(T_1, d_1)$ such that: 
\begin{enumerate} 
\item $N_1$ acts minimally on $T_1$ by isometries with trivial arc stabilizers;
\item there is a (unique) $\Phi$-equivariant $\lambda_1$-homothety $h_1 \colon (T_1, d_1) \to (T_1, d_1)$ with $\lambda_1>1$; and
\item an element $g \in N_1$ is elliptic in $T_1$ if and only if it is not weakly attracted to $[\Lambda_1]$.
\end{enumerate}
By \cref{lem:approximategenline_standard}, the subgroup system $\mathcal N_2 \defeq \mathcal N(T_1)$ of nontrivial point stabilizers of $T_1$ carries the remaining laminations $\mathcal L^+_1 \setminus \{ [\Lambda_1] \} = \mathcal L^+_{\phi_1 | \mathcal N_2}$.
For each $[N_2] \in \mathcal N_2$, let $\phi_2 \in \Out(N_2)$ be the corresponding restriction in $\phi_1|\mathcal N_2$ and $\mathcal L^+_2 \subseteq \mathcal L^+_1$ the subset identified with $\mathcal L^+_{\phi_2}$.
Applying \cite[Proposition~2.2]{Mutanguha22} again, there is a lamination $[\Lambda_2] \in \mathcal L^+_2$ and an $\mathbb R$-tree $(T_2, d_2)$ as before.

The main construction of \cite{Mutanguhaexistence} describes how to equivariantly substitute the trees~$T_2$ into the points $p_2 \in T_1$ (with nontrivial stabilizers $N_2$).
This process is topological as it forgets the metrics $d_1, d_2$ and produces a topological tree $\widehat T_2$ whose minimal non-nesting $N_1$-action has trivial arc stabilizers.
Equivariantly collapsing the copies of $T_2$ in $\widehat T_2$ recovers the $T_1$, and the pullback of $d_1$ via the quotient map $\widehat T_2 \to T_1$ is a pseudometric on $\widehat T_2$ that we continue to denote $d_1$.
Consequently, the equivariant copies of $T_2$ are precisely the nondegenerate maximal subsets of $\widehat T_2$ with $d_1$-diameter 0;
moreover, these subsets inherit the $F_2$-invariant metrics $d_2$.
Denote this structure by the tuple $(\widehat T_2, (d_1, d_2))$.

Inductively apply \cite[Proposition~2.2]{Mutanguha22} to get a tuple $(\widehat T_n, (d_1, \ldots, d_n))$ such that:
\begin{enumerate}
    \item there is a $\Phi$-equivariant homeomorphism $\hat h_n \colon \widehat T_n \to \widehat T_n$ that is a $\lambda_i$-homothety for $\lambda_i >1$, with respect to the pseudometric $d_i$ for each $i$; and
    \item an element $g \in \free$ is elliptic in $\widehat T_n$ if and only if it is not weakly attracted to any lamination in $\mathcal L^+_\phi$ (i.e.~$g$ has polynomial growth under $\phi$-iteration).
\end{enumerate}
The previous statement is a summary of \cite[Theorem~3.1 and Proposition~3.4]{Mutanguha22};
moreover, the underlying topological tree $T_\phi = \widehat T_n$ is unique \cite[Corollary~3.9]{Mutanguha22}.

For any pseudometric $d_i$ in the tuple, the corresponding lamination $[\Lambda_i] \in \mathcal L^+_\phi$ is naturally realized by an $\free$-invariant collection of maximal arcs in~$T_\phi$, and elements in this collection will be called $T_\phi$-leaves of~$[\Lambda_i]$.
By construction, an arc in $T_\phi$ contains a subarc with positive $d_i$-diameter if and only if it contains a nondegenerate subarc of a $T_\phi$-leaf of $[\Lambda_i]$.
Define $g \in \free$ to be \emph{$d_i$-loxodromic} if its axis in~$T_\phi$ contains a subarc with positive $d_i$-diameter.
By \cite[Proposition~3.5]{Mutanguha22}, being $d_i$-loxodromic is equivalent to being weakly attracted to~$[\Lambda_i]$.
Borrowing notation from measure theory, we write $d_i \ll d_j$ if any arc in~$T_\phi$ containing a subarc with positive $d_i$-diameter must also contain a subarc with positive $d_j$-diameter. 
Finally, \cite[Claim~3.6]{Mutanguha22} states that $d_i \ll d_j$ if and only if $[\Lambda_j] \subseteq [\Lambda_i]$.

\begin{thmn}[\ref{Thm:NSS}]
For $\phi \in \Out(\free)$ and $\mathcal L \subseteq \mathcal L^+_\phi$, there is a canonical topological tree $T_{\mathcal L}$ with a non-nesting continuous action by $\free$ such that:
\begin{enumerate}
\item the action is minimal with trivial arc stabilizers;
\item for any automorphism $\Phi \colon \free \to \free$ representing~$\phi$, there is an expanding $\Phi$-equivariant homeomorphism $h_{\mathcal L} \colon T_{\mathcal L} \to T_{\mathcal L}$; and
\item an element $g \in \free$ is elliptic in $T_{\mathcal L}$ if and only if it is not weakly attracted to any $[\Lambda] \in \mathcal L$.
\end{enumerate}
\end{thmn}

\begin{proof}
    Suppose a subset $\mathcal L \subseteq \mathcal L^+_\phi$ is given.
    Construct the topological tree~$T_{\mathcal L}$ by collapsing in $T_\phi$ the leaves of all $[\Lambda] \in \mathcal L^+_\phi$ that do not contain a lamination in~$\mathcal L$.
    Equivalently, we equivariantly collapse all pseudometrics $d$ such that $d \not\ll d_i$ for all pseudometrics~$d_i$ corresponding to $[\Lambda_i] \in \mathcal L$.
    So $T_{\mathcal L}$ inherits a non-nesting minimal $\free$-action with trivial arc stabilizers, and the induced $\Phi$-equivariant homeomorphism $h_{\mathcal L} \colon T_{\mathcal L} \to T_{\mathcal L}$ is still a $\lambda_i$-homothety for the non-collapsed pseudometrics~$d_i$.
    By construction, an element $g \in \free$ is elliptic in $T_{\mathcal L}$ if and only if it is not weakly attracted to any $[\Lambda] \in \mathcal L$.
    As~$T_\phi$ is unique, the construction of $T_{\mathcal L}$ is canonical.
\end{proof}

\bibliographystyle{alpha}
\bibliography{biblio}

\end{document}